\documentclass[a4paper,twoside]{article}
\usepackage{a4}
\usepackage{amssymb}
\usepackage{amsmath}
\usepackage{upref}
\usepackage[active]{srcltx}
\usepackage[pagebackref,colorlinks,citecolor=blue,linkcolor=blue]{hyperref}
\usepackage[dvipsnames]{color}
\allowdisplaybreaks[2] 
\newcount\minutes \newcount\hours
\hours=\time
\divide\hours 60
\minutes=\hours
\multiply\minutes -60
\advance\minutes \time
\newcommand{\klockan}{\the\hours:{\ifnum\minutes<10 0\fi}\the\minutes}
\newcommand{\tid}{\today\ \klockan}
\newcommand{\prtid}{\smash{\raise 10mm \hbox{\LaTeX ed \tid}}}
\renewcommand{\prtid}{}
\makeatletter
\def\sectionmark#1{} 
\def\subsectionmark#1{}
\newcommand{\sectnr}{\ifnum \c@secnumdepth >\z@
                 \thesection.\hskip 1em\relax \fi}
\def\@evenhead{\footnotesize\rm\thepage\hfil\leftmark\hfil\llap{\prtid}}
\def\@oddhead{\footnotesize\rm\rlap{\prtid}\hfil\rightmark\hfil\thepage}
\def\tableofcontents{\section*{Contents} 
 \@starttoc{toc}}
\makeatother
\makeatletter
\def\@biblabel#1{#1.}
\makeatother
\makeatletter
\let\Thebibliography=\thebibliography
\renewcommand{\thebibliography}[1]{\def\@mkboth##1##2{}\Thebibliography{#1}
\addcontentsline{toc}{section}{References}
\frenchspacing 
\setlength{\@topsep}{0pt}
\setlength{\itemsep}{0pt}%
\setlength{\parskip}{0pt plus 2pt}%
}
\makeatother
\makeatletter
\def\mdots@{\mathinner.\nonscript\!.%
 \ifx\next,.\else\ifx\next;.\else\ifx\next..\else
 \nonscript\!\mathinner.\fi\fi\fi}
\let\ldots\mdots@
\let\cdots\mdots@
\let\dotso\mdots@
\let\dotsb\mdots@
\let\dotsm\mdots@
\let\dotsc\mdots@
\def\vdots{\vbox{\baselineskip2.8\p@ \lineskiplimit\z@
    \kern6\p@\hbox{.}\hbox{.}\hbox{.}\kern3\p@}}
\def\ddots{\mathinner{\mkern1mu\raise8.6\p@\vbox{\kern7\p@\hbox{.}}%
    \raise5.8\p@\hbox{.}\raise3\p@\hbox{.}\mkern1mu}}
\makeatother
\makeatletter
\let\Enumerate=\enumerate
\renewcommand{\enumerate}{\Enumerate%
\setlength{\@topsep}{0pt}
\setlength{\itemsep}{0pt}%
\setlength{\parskip}{0pt plus 1pt}%
\renewcommand{\theenumi}{\textup{(\alph{enumi})}}%
\renewcommand{\labelenumi}{\theenumi}%
}
\let\endEnumerate=\endenumerate
\renewcommand{\endenumerate}{\endEnumerate\unskip}
\makeatother
\makeatletter
\def\@seccntformat#1{\csname the#1\endcsname.\quad}
\makeatother
\makeatletter
\renewcommand*\l@section[2]{%
  \ifnum \c@tocdepth >\z@
    \addpenalty\@secpenalty
    \addvspace{0em \@plus\p@}
    \setlength\@tempdima{1.5em}%
    \begingroup
      \parindent \z@ \rightskip \@pnumwidth
      \parfillskip -\@pnumwidth
      \leavevmode \bfseries
      \advance\leftskip\@tempdima
      \hskip -\leftskip
      #1\nobreak\hfil \nobreak\hb@xt@\@pnumwidth{\hss #2}\par
    \endgroup
  \fi}
\makeatother
\newcommand{\authortitle}[2]{\author{#1}\title{#2}\markboth{#1}{#2}}
\newcommand{\auth}[2]{{#1, #2.}}
\newcommand{\art}[6]{{\sc #1, \rm #2, \it #3 \bf #4 \rm (#5), \mbox{#6}.}}

\newcommand{\arttoappear}[3]{{\sc #1, \rm #2, to appear in \it #3}}
\newcommand{\book}[3]{{\sc #1, \it #2, \rm #3.}}
\newcommand{\artprep}[3]{{\sc #1, \rm #2, \rm #3.}}

\newcommand{\AND}{{\rm and }}
\RequirePackage{amsthm}
\newtheoremstyle{descriptive}%
  {\topsep}   
  {\topsep}   
  {\rmfamily} 
  {}          
  {\bfseries} 
  {.}         
  { }         
  {}          
\newtheoremstyle{propositional}%
  {\topsep}   
  {\topsep}   
  {\itshape}  
  {}          
  {\bfseries} 
  {.}         
  { }         
  {}          
\newtheoremstyle{remarkstyle}%
  {\topsep}   
  {\topsep}   
  {\rmfamily}  
  {}          
  {\itshape} 
  {.}         
  { }         
  {}          
\theoremstyle{propositional}
\newtheorem{thm}{Theorem}[section]
\newtheorem{prop}[thm]{Proposition}
\newtheorem{lem}[thm]{Lemma}
\newtheorem{cor}[thm]{Corollary}
\theoremstyle{descriptive}
\newtheorem{deff}[thm]{Definition}
\newtheorem{example}[thm]{Example}
\newtheorem{remark}[thm]{Remark}
\makeatletter
\renewenvironment{proof}[1][\proofname]{\par
  \pushQED{\qed}%
  \normalfont 
  \trivlist
  \item[\hskip\labelsep
        \itshape
    #1\@addpunct{.}]\ignorespaces
}{%
  \popQED\endtrivlist\@endpefalse
}
\makeatother
\newcommand{\setm}{\setminus}
\renewcommand{\subsetneq}{\varsubsetneq}
\renewcommand{\emptyset}{\varnothing}
\def\vint{\mathop{\mathchoice%
          {\setbox0\hbox{$\displaystyle\intop$}\kern 0.22\wd0%
           \vcenter{\hrule width 0.6\wd0}\kern -0.82\wd0}%
          {\setbox0\hbox{$\textstyle\intop$}\kern 0.2\wd0%
           \vcenter{\hrule width 0.6\wd0}\kern -0.8\wd0}%
          {\setbox0\hbox{$\scriptstyle\intop$}\kern 0.2\wd0%
           \vcenter{\hrule width 0.6\wd0}\kern -0.8\wd0}%
          {\setbox0\hbox{$\scriptscriptstyle\intop$}\kern 0.2\wd0%
           \vcenter{\hrule width 0.6\wd0}\kern -0.8\wd0}}%
          \mathopen{}\int}
\newcommand{\CpY}{C_p^Y}
\newcommand{\CpXeps}{C_p^{\clXeps}}
\newcommand{\OCpXeps}{C_p^{X_\eps}}
\DeclareMathOperator{\diam}{diam}
\DeclareMathOperator{\Capp}{Cap}
\DeclareMathOperator{\dist}{dist}
\DeclareMathOperator{\Lip}{Lip}
\DeclareMathOperator{\Tr}{Tr}
\DeclareMathOperator{\supp}{supp}
\newcommand{\bdry}{\partial}
\newcommand{\bdy}{\bdry}
\newcommand{\be}{\beta}
\newcommand{\z}{\zeta}
\newcommand{\ka}{\kappa}
\renewcommand{\approx}{\simeq}
\newcommand{\simge}{\gtrsim}
\newcommand{\simle}{\lesssim}
\newcommand{\Bppal}{{B^\theta_{p,p}}}

{\catcode`p =12 \catcode`t =12 \gdef\eeaa#1pt{#1}}      
\def\accentadjtext#1{\setbox0\hbox{$#1$}\kern   
                \expandafter\eeaa\the\fontdimen1\textfont1 \ht0 }
\def\accentadjscript#1{\setbox0\hbox{$#1$}\kern 
                \expandafter\eeaa\the\fontdimen1\scriptfont1 \ht0 }
\def\accentadjscriptscript#1{\setbox0\hbox{$#1$}\kern   
                \expandafter\eeaa\the\fontdimen1\scriptscriptfont1 \ht0 }
\def\accentadjtextback#1{\setbox0\hbox{$#1$}\kern       
                -\expandafter\eeaa\the\fontdimen1\textfont1 \ht0 }
\def\accentadjscriptback#1{\setbox0\hbox{$#1$}\kern     
                -\expandafter\eeaa\the\fontdimen1\scriptfont1 \ht0 }
\def\accentadjscriptscriptback#1{\setbox0\hbox{$#1$}\kern 
                -\expandafter\eeaa\the\fontdimen1\scriptscriptfont1 \ht0 }
\def\itoverline#1{{\mathsurround0pt\mathchoice
        {\rlap{$\accentadjtext{\displaystyle #1}
                \accentadjtext{\vrule height1.593pt}
                \overline{\phantom{\displaystyle #1}
                \accentadjtextback{\displaystyle #1}}$}{#1}}
        {\rlap{$\accentadjtext{\textstyle #1}
                \accentadjtext{\vrule height1.593pt}
                \overline{\phantom{\textstyle #1}
                \accentadjtextback{\textstyle #1}}$}{#1}}
        {\rlap{$\accentadjscript{\scriptstyle #1}
                \accentadjscript{\vrule height1.593pt}
                \overline{\phantom{\scriptstyle #1}
                \accentadjscriptback{\scriptstyle #1}}$}{#1}}
        {\rlap{$\accentadjscriptscript{\scriptscriptstyle #1}
                \accentadjscriptscript{\vrule height1.593pt}
                \overline{\phantom{\scriptscriptstyle #1}
                \accentadjscriptscriptback{\scriptscriptstyle #1}}$}{#1}}}}
\newcommand{\al}{\alpha}
\newcommand{\alp}{\alpha}
\newcommand{\ga}{\gamma}
\newcommand{\Ga}{\Gamma}
\newcommand{\de}{\delta}
\newcommand{\eps}{\varepsilon}
\newcommand{\la}{\lambda}
\newcommand{\sig}{\sigma}
\newcommand{\Om}{\Omega}
\newcommand{\muh}{\hat{\mu}}
\newcommand{\mube}{{\mu_\be}}
\newcommand{\sbe}{{s_\be}}
\newcommand{\snu}{{s_\nu}}
\newcommand{\vh}{\hat{v}}
\newcommand{\Xhat}{\widehat{X}}
\newcommand{\Fhat}{\widehat{F}}
\newcommand{\p}{{$p\mspace{1mu}$}}   
\newcommand{\R}{\mathbf{R}}
\newcommand{\clX}{\itoverline{X}}
\newcommand{\gat}{\widetilde{\ga}}
\newcommand{\gah}{\widehat{\ga}}
\newcommand{\limplus}{{\mathchoice{\vcenter{\hbox{$\scriptstyle +$}}}
  {\vcenter{\hbox{$\scriptstyle +$}}}
  {\vcenter{\hbox{$\scriptscriptstyle +$}}}
  {\vcenter{\hbox{$\scriptscriptstyle +$}}}
}}
\newcommand{\limminus}{{\mathchoice{\vcenter{\hbox{$\scriptstyle -$}}}
  {\vcenter{\hbox{$\scriptstyle -$}}}
  {\vcenter{\hbox{$\scriptscriptstyle -$}}}
  {\vcenter{\hbox{$\scriptscriptstyle -$}}}
}}
\newcommand{\limpm}{{\mathchoice{\vcenter{\hbox{$\scriptstyle \pm$}}}
  {\vcenter{\hbox{$\scriptstyle \pm$}}}
  {\vcenter{\hbox{$\scriptscriptstyle \pm$}}}
  {\vcenter{\hbox{$\scriptscriptstyle \pm$}}}
}}
\newcommand{\Np}{N^{1,p}}
\newcommand{\tNp}{\widetilde{N}^{1,p}}
\newcommand{\EE}{\mathcal{E}}%
\newcommand{\LL}{\mathcal{L}}%
\newcommand{\ut}{\tilde{u}}
\newcommand{\clXeps}{\clX_\eps}
\newcommand{\uhat}{\hat{u}}
\newcommand{\bdyeps}{\bdy_\eps}
\newcommand{\dXeps}{\bdyeps X}
\newcommand{\dheps}{\hat{d}_\eps}
\newcommand{\distheps}{\widehat{\dist}_\eps}
\newcommand{\diamheps}{\widehat{\diam}_\eps}
\newcommand{\zm}{z_\limminus}
\newcommand{\zp}{z_\limplus}
\newcommand{\Zms}{Z_\limminus}
\newcommand{\Zps}{Z_\limplus}
\newcommand{\clZ}{\itoverline{Z}}

\makeatletter
\newcommand{\setcurrentlabel}[1]{\def\@currentlabel{#1}}
\makeatother
\newcounter{saveenumi}
\numberwithin{equation}{section}
\newcommand{\eqv}{\mathchoice{\quad \Longleftrightarrow \quad}{\Leftrightarrow}
                {\Leftrightarrow}{\Leftrightarrow}}
\newcommand{\imp}{\mathchoice{\quad \Longrightarrow \quad}{\Rightarrow}
                {\Rightarrow}{\Rightarrow}}
\newenvironment{ack}{\medskip{\it Acknowledgement.}}{}

\begin{document}

\authortitle{Anders Bj\"orn, Jana Bj\"orn
	and Nageswari Shanmugalingam}
{Extension and trace results for doubling metric measure spaces and their
    hyperbolic fillings}
  \author{
Anders Bj\"orn \\
\it\small Department of Mathematics, Link\"oping University, SE-581 83 Link\"oping, Sweden\\
\it \small anders.bjorn@liu.se, ORCID\/\textup{:} 0000-0002-9677-8321
\\
\\
Jana Bj\"orn \\
\it\small Department of Mathematics, Link\"oping University, SE-581 83 Link\"oping, Sweden\\
\it \small jana.bjorn@liu.se, ORCID\/\textup{:} 0000-0002-1238-6751
\\
\\
Nageswari Shanmugalingam
\\
\it \small  Department of Mathematical Sciences, University of Cincinnati,
P.O.\ Box 210025,\\
\it \small   Cincinnati, OH 45221-0025, U.S.A.\/{\rm ;}
\it \small shanmun@uc.edu, ORCID\/\textup{:} 0000-0002-2891-5064
}

\date{}
\maketitle

\noindent{\small
{\bf Abstract} In this paper we study connections between Besov spaces 
of functions
on a compact metric space $Z$, equipped with a doubling measure,
and the Newton--Sobolev space of functions
on
a uniform domain $X_\eps$.
This uniform domain is obtained as a uniformization of a (Gromov) hyperbolic
filling of $Z$. To do so, we construct a family of hyperbolic fillings 
in the style of  Bonk--Kleiner~\cite{BK} and
Bourdon--Pajot~\cite{BP}. Then
for each parameter $\beta>0$ we construct a lift $\mube$ of the doubling
measure $\nu$ on $Z$ to $X_\eps$, and show that 
$\mube$
is doubling and supports a $1$-Poincar\'e inequality. We then show that for 
each $\theta$ with $0<\theta<1$ and $p\ge 1$ there is a choice of $\beta=p(1-\theta)\log\al$
such that the Besov space $B^\theta_{p,p}(Z)$ is the trace space of the 
Newton--Sobolev space $N^{1,p}(X_\eps,\mube)$ when $\eps=\log\al$.
Finally, we exploit the 
tools of potential theory on $X_\eps$ to obtain 
fine properties of functions in $B^\theta_{p,p}(Z)$, such as their
  quasicontinuity 
  and quasieverywhere existence of $L^q$-Lebesgue points
with $q=s_\nu p/(s_\nu-p\theta)$,
where $s_\nu$ is a doubling dimension associated with the measure
$\nu$ on $Z$.
Applying this to 
compact subsets of Euclidean spaces improves upon a result of
Netrusov~\cite{Netr} in $\R^n$.}



\tableofcontents

\section{Introduction}

Much of the current trend in first-order analysis (such as Sobolev
type spaces and potential theory) on metric measure spaces assumes
that the underlying space is
at least locally compact, doubling and supports a Poincar\'e type inequality, see
for example~\cite{BBbook} and~\cite{HKST}. 
On doubling spaces that do not support any Poincar\'e inequality, 
there are other possible choices of function spaces that are, however,
nonlocal.
This means that the energy of functions in these spaces depends on their global
behavior and can be nonzero on subsets where the functions vanish or
are constant.
Examples of such spaces include Besov, Triebel--Lizorkin
and (fractional) Haj\l asz--Sobolev spaces.
As locality is a highly useful tool in potential theory and
variational problems, it is desirable to seek an alternative approach to
studying function spaces on nonsmooth metric spaces without
Poincar\'e inequalities.

Using a construction termed \emph{hyperbolic filling},
Bonk--Kleiner~\cite[Theorem~11.1]{BK} and  
Bourdon--Pajot~\cite{BP} connected compact doubling metric spaces to Gromov hyperbolic spaces.
\emph{Gromov hyperbolicity} is a notion of negative curvature in the nonsmooth
metric setting and, unlike Alexandrov curvature which covers all
scales, it captures negative curvature at large scales by
requiring 
that every point in a geodesic triangle is within a bounded distance
from the other two sides.
Gromov hyperbolicity has proven to be a highly useful tool in studying 
the conformal geometry of hyperbolic groups 
(Bridson--Haefliger~\cite{BH} and Gromov~\cite{GromovBook}) and in understanding
uniform domains 
(Bonk--Heinonen--Koskela~\cite{BHK-Unif}, 
Bonk--Schramm~\cite{BoSc} and Herron--Shanmugalingam--Xie~\cite{HSX}). 
We refer the interested reader to~\cite{BH} and
Buyalo--Schroeder~\cite{BuSch} for more on synthetic notions of
curvature in the metric setting and the hyperbolic filling technique,
respectively. 

In the current paper we contribute to the study of nonlocal analysis
on compact doubling metric measure spaces by
introducing measures to the above construction of hyperbolic fillings
and by subsequently
linking the nonlocal Besov spaces on compact doubling metric measure spaces
to the Newtonian (Sobolev) spaces on uniformizations of their
hyperbolic fillings, see Theorem~\ref{thm-main-intro}.

One of the main results of the paper is that \emph{every} Besov space
$\Bppal(Z)$ with $0<\theta<1$ and $p\ge1$ arises as a trace of a
Newtonian space on a uniform domain, equipped with a doubling measure
supporting a $1$-Poincar\'e inequality. 
Specifically, we prove the following theorem.

\begin{thm} \label{thm-main-intro} 
Let $Z$ be a compact metric space equipped with a doubling measure $\nu$,
$X$ be a hyperbolic filling of $Z$ with parameters
$\alpha,\tau>1$,
and
$X_\eps$ be
the uniformization of $X$ with parameter $\eps=\log\alpha$.

Then for each parameter $\beta>0$ we can equip $X_\eps$
with a measure $\mube$ induced by $\nu$
so that $\mube$ is doubling and supports a $1$-Poincar\'e
inequality both on $X_\eps$ and its completion $\clX_\eps$. 
Moreover, for each $0<\theta<1$ and $1\le p<\infty$, 
the Besov space $\Bppal(Z)$  is the
trace space of the Newtonian space $\Np(X_\eps,\mube)$ with
$\beta=\eps p(1-\theta)$.
\end{thm}

More precisely,  $\Tr\circ E$ is the identity map on
  $\Bppal(Z)$, where
\[
\Tr:N^{1,p}(X_\eps,\mube)\to \Bppal(Z) \quad \text{and} \quad
E:\Bppal(Z)\to N^{1,p}(X_\eps,\mube)
\]
are bounded linear trace and extension operators, respectively.

In fact, we construct  
$\Tr$ for each 
$\theta\le 1-\beta/\eps p$ (see Theorem~\ref{thm-trace-fund-refined}), 
and  $E$ for each
$\theta\ge 1-\beta/\eps p$ (see Theorem~\ref{thm:filling-Extension}).
Roughly speaking, $\mube$ is constructed so that its
$\be/\eps$-codimensional Hausdorff measure is equivalent to the measure $\nu$ on $Z$.
Thus Theorem~\ref{thm-main-intro} follows
immediately from combining Theorems~\ref{thm-trace-fund-refined} 
and~\ref{thm:filling-Extension}.

The smoothness exponent $\theta=1-\beta/\eps p>0$ exactly
corresponds to, and generalizes, the case of $d$-sets in unweighted $\R^n$, considered
by Jonsson--Wallin~\cite{JW80}, \cite{JW84}. We also show that 
for $\nu$-a.e.~$z\in Z$ the trace $\Tr u(z)$ is
achieved in three different ways, namely, 
as averaged pointwise limits~\eqref{eq-def-un-An} and~\eqref{eq-def-trace},
by Lebesgue point integral averages~\eqref{eq-ext-as-Leb-pt}
and as a pointwise restriction from the Newtonian space $\Np(\clX_\eps)$.

Our study includes $p=1$ and shows that the Besov space $B^\theta_{1,1}(Z)$ is the trace space
of the Newtonian space $N^{1,1}(X_\eps,\mube)$ when 
$\be=(1-\theta)\eps=(1-\theta)\log\al$.
This is in contrast to the result of Gagliardo~\cite{Ga} that the
trace space of the Sobolev space
$W^{1,1}(\Omega)$ is $L^1(\bdy\Omega)$ when $\Omega$ is a Lipschitz domain in $\R^n$.
This trace operator is nonlinear, which is necessary according to a
result due to Peetre~\cite{Peetre}.
See Mal\'y~\cite[Section~7]{MalyBesov} and 
Mal\'y--Shanmugalingam--Snipes~\cite[Theorem~1.2]{MSS} for metric space analogs of this.
The key difference between the setting of~\cite{MSS} and the current paper is that
in \cite{MSS} the measure on $\bdy \Om$ has codimension $1$ relative to 
the measure on $\Om$, while we have codimension $\be/\eps<p$ 
which precludes us from having  $\beta/\eps=1$ when $p=1$. 
See also \cite[Section~7]{MalyBesov} for the importance of this difference.

To prove the above theorem, for each choice of $\al,\tau>1$ and $0<\eps\le\log\al$,
we construct a hyperbolic filling $X$ of the metric space $Z$.
Roughly speaking, is an infinite graph whose vertices 
correspond to maximal $\al^{-n}$-separated subsets of $Z$ for each positive
integer $n$.  The role of $\tau$ is to define nearness between two points in each of
these sets as in \eqref{eq-tilde-m-n}.
We then equip $X$ with the uniformized metric 
\[
d_\eps(x,y) = \inf_\ga \int_\ga e^{-\eps d(\cdot,v_0)}\,ds,
\]
where $d(\,\cdot\,,v_0)$ denotes the graph distance to
the root $v_0$ of the hyperbolic filling and the
infimum is taken over all curves in $X$ joining $x$ to $y$.

Along the way, we explore how the choice of parameters affects the
structure of the hyperbolic fillings $X$ of $Z$ and their uniformizations $X_\eps$:
\begin{itemize}
  \setlength{\itemsep}{0pt}%
  \setlength{\parskip}{0pt plus 1pt}%
\renewcommand{\theenumi}{\textup{(\alph{enumi})}}%
\renewcommand{\labelenumi}{\theenumi}%
\item Hyperbolic fillings are 
  Gromov hyperbolic for all $\al,\tau>1$ (Theorem~\ref{thm-Gromov-hyp})
but not when $\tau=1$  (Example~\ref{ex-not-hyperbolic}).
\item 
The uniformization $X_\eps$ is a uniform space for all $\eps\le\log\al$ 
  (Theorem~\ref{thm:alph-hyp-fill-eps-uniformize-new}).
\item The  boundary of the uniformization $X_\eps$, with $\eps=\log\al$, 
is biLipschitz equivalent to $Z$ when $Z$ is compact
(Proposition~\ref{prop-Z-biLip-Xeps}).
\end{itemize}

Subsequently, for a doubling measure $\nu$
on $Z$, we construct a lift of $\nu$ to a measure $\mu$ on $X$
which is uniformly locally doubling and supports a 
  uniformly local $1$-Poincar\'e inequality.
We then show that for every $\be>0$, the corresponding weighted measure 
\begin{equation}   \label{eq-def-mube-intro}
d\mube(x) = e^{-\be d(x,x_0)}\, d\mu(x) 
    \simeq \dist_\eps(x,\bdy_\eps X)^{\be/\eps}\, d\mu(x)
\end{equation}
is globally doubling and supports a
global $1$-Poincar\'e inequality on $X_\eps$ and its
closure $\clX_\eps$, see Theorem~\ref{thm-muh-be-doubl-all-be}.
This gives us the flexibility to choose $\beta$ for each $0<\theta<1$ and $p\ge 1$
so that $\theta=1-\beta/p\eps$ and thus see the nonlocal Besov space $\Bppal(Z)$ 
as the trace space of the Newtonian space $\Np(X_\eps,\mube)$, with the advantage that the
Newtonian energy is local, and that the theory for Newtonian spaces is more
  developed than the theory for Besov spaces on metric spaces.

Invoking the regularity properties of
Newtonian spaces, we then easily obtain several regularity results for Besov functions on $Z$:

\begin{cor}  \label{cor-harvest-intro}
Let $Z$ be a compact metric space equipped with a doubling measure~$\nu$.
Then for every $0<\theta<1$, Lipschitz functions are dense in
$\Bppal(Z)$ and every function in $\Bppal(Z)$ has a
representative that is quasicontinuous with respect to the Besov capacity.

If $p> s_\nu/\theta$, where $s_\nu$ is the growth exponent of $\nu$
from~\eqref{eq-def-s-nu}, then functions in $\Bppal(Z)$ have H\"older continuous
representatives.

If $\nu$ is in addition reverse-doubling, then functions in
  $\Bppal(Z)$ belong to $L^q(Z)$ for $q=\snu p/(\snu-p\theta)$
  and have $L^q$-Lebesgue points outside 
a set of zero $\Bppal(Z)$-capacity. 
\end{cor}

Our results apply also to compact subsets of $\R^n$. 
On $\R^n$, the corresponding Sobolev-type higher integrability result
is due to Peetre~\cite[Th\'eor\`eme~8.1]{PeetreEsp},
while Netrusov~\cite[Proposition~1.4]{Netr} obtained
the Lebesgue point result for $q<np/(n-p\theta)$.
Even though $\R^n$ is not compact, the
above corollary allows us to improve upon
Netrusov's result in the Euclidean setting by 
including $q=np/(n-p\theta)$, see Proposition~\ref{prop-Netrusov}.
These results show that our exponent $q$ is optimal.

The above density and quasicontinuity results are known to hold when
the Besov space is defined in terms of atomic decompositions 
or sequences of fractional Haj\l asz gradients
(see for example
Han--M\"uller--Yang~\cite[Definition~5.29]{HMY},
Koskela--Yang--Zhou~\cite[Definitions~1.2 and~4.4]{KYZ} 
and Heikkinen--Koskela--Tuo\-minen~\cite[Definition~2.5,
    Theorems~1.1 and~1.2]{HKT2017}).
Such spaces coincide with our Definition~\ref{def:Besov} 
when $Z$ is unbounded and ``reverse-doubling'', by \cite[Theorem~4.1]{KYZ}.
In particular, the definitions are equivalent 
if $Z$ is uniformly perfect with $\nu$
doubling, for example when $Z=\R^n$. 
With our assumptions on $Z$, it is not clear whether those definitions agree
with the nonlocal integral definition considered here. 
Note that the integral in Definition~\ref{def:Besov} coincides with the
ones defining the classical fractional Sobolev spaces in Euclidean spaces
and is naturally related to nonlocal minimization problems for the fractional
\p-Laplacian, as in Caffarelli--Silvestre~\cite{CS} and Ferrari--Franchi~\cite{FF}.

Higher integrability and H\"older continuity of Besov functions for large $p$ appears also in
Mal\'y~\cite[Corollary~3.18]{MalyBesov}, where it is
obtained as a consequence of embeddings into Haj\l asz--Sobolev
spaces, provided by Lemma~6.1 in 
Gogatishvili--Koskela--Shan\-mu\-ga\-lin\-gam~\cite{GKS}.
Our approach based on hyperbolic fillings is different.
Traces of Newtonian functions on uniform domains in metric spaces are
also studied in~\cite[Theorem~1.1]{MalyBesov} by means of Lebesgue
point averages and Poincar\'e inequalities.
Our proof in the setting of hyperbolic fillings is more direct and
rather elementary (albeit a bit lengthy) and is based only on the basic
properties of upper gradients.
In particular, the Poincar\'e inequality is not used.
Moreover, we show that for \emph{any} compact doubling space $Z$, the
Besov space $\Bppal(Z)$ with $0<\theta<1$ is the trace of some
Newtonian space, and that the trace can be obtained as a pointwise
restriction from the Newtonian space $\Np(\clX_\eps)$, see 
Theorem~\ref{thm-main-intro}.

Let us compare our definition of Besov spaces with some other 
function spaces on boundaries of hyperbolic fillings.
Assuming $Z$ to be uniformly perfect and Ahlfors
$Q$-regular, the Besov space considered in Bourdon--Pajot~\cite{BP}
corresponds to our $\Bppal(Z)$ with $\theta=Q/p$ and is shown to be
isomorphically equivalent to the first cohomology group $\ell_p H^1(X)$.
For Ahlfors $Q$-regular spaces $Z$,
the papers \cite{BS}, \cite{BSS} and~\cite{S} by Bonk, Saksman and
Soto define certain function spaces on $Z$ by means of
Poisson-type extensions as in our
Theorem~\ref{thm:filling-Extension}, and using the counting measure
on the collection of all edges in the hyperbolic filling.
In \cite{BS} they show that if $Z$ supports a $Q$-Poincar\'e
inequality then their space $A^p(Z)$ coincides with the Haj\l asz--Sobolev space $M^{1,Q}(Z)$.
The function spaces considered in~\cite{BSS} are of Triebel--Lizorkin
type, while the ones in~\cite{S} are identified with the Haj\l asz--Besov
spaces $\dot{N}^s_{p,q}(Z)$, defined by atomic decompositions in 
Koskela--Yang--Zhou~\cite[Definition~1.2]{KYZ}.

While these results are interesting, from our point of view it is somewhat 
unsatisfactory that the energy of functions considered
in \cite{BS}, \cite{BSS} and~\cite{S}
does not take into full account the measure $\nu$ on $Z$ and that
$\nu$ is not related to a measure on the hyperbolic filling.

Unlike in~\cite{S}, our definition of Besov spaces is based on
integrals directly on the metric space $Z$, rather than 
on sequence spaces, see Definition~\ref{def:Besov}. 
Moreover, the smoothness of the corresponding Poisson extension on
the hyperbolic filling $X$, is controlled by a measure on $X$ that is compatible with
the measure $\nu$ on $Z$. 

The structure of this paper is as follows.  
The necessary background related to metric notions and Gromov
hyperbolic spaces is given in Section~\ref{sect-Gromov}, while notions related to
Newtonian  and Besov spaces are given in Section~\ref{sect-prelim-not}. 

In Section~\ref{sect-hyp-fill} we describe the construction of the hyperbolic
filling $X$ of a general bounded
metric space $Z$, associated with the parameters $\al, \tau>1$, and show that it indeed forms a Gromov hyperbolic space.
Subsequently, in Section~\ref{sect-bdyepsX} we describe the
uniformization $X_\eps$ of $X$, with parameter $\eps>0$, 
in the style of Bonk--Heinonen--Koskela~\cite{BHK-Unif}, adapted to the setting of hyperbolic fillings. In this section we also
explore links between the boundary $\bdy_\eps X$ of the uniformized space and the original
metric space $Z$.  In particular, the results in Section~\ref{sect-bdyepsX}
show why the bound $\eps \le \log \alp$ is natural.

The primary goal of Section~\ref{sect-uniformize} is to prove
that the uniformization $X_\eps$ of $X$ yields a uniform space when $\eps\le \log\alpha$.
The general results of~\cite{BHK-Unif} imply
that $X_\eps$ is a uniform space for sufficiently
small $\eps$, but  our direct proof for hyperbolic fillings covers
all $\eps \le \log \alp$,  which is vital for our further results. 
Observe that general Gromov hyperbolic spaces do not always
yield a uniform space when uniformized, see for example
Lindquist--Shanmugalingam~\cite[Section~4]{LS}. 

Given that all $\al,\tau>1$ generate a hyperbolic filling of $Z$,
it is worth exploring how the choice of
these parameters affects the structure of the hyperbolic filling. 
The rough similarity between an arbitrary locally compact roughly
starlike Gromov hyperbolic space $X$ and the
hyperbolic filling $\Xhat$ of its uniformized boundary $\bdy_\eps X$
is for small $\eps$ proved in Section~\ref{sect-quasiisom-equiv},
without any limitations on~$\al$ in terms of~$\eps$.
Trees and hyperbolic fillings of their uniformized boundaries, as well
as some counterexamples, are considered in Section~\ref{sec:trees}.
In Section~\ref{sec:geodesics-large-tau} we show that if
$\tau\ge(\al+1)/(\al-1)$, then we have good control over geodesics in
the hyperbolic filling. As the discussion in these three sections is provided to 
explore the hyperbolic fillings and uniformization further, 
those who are interested only in the theory of Besov spaces may skip these three
sections without confusion.

In the rest of the paper we consider a compact metric space $Z$ equipped with a doubling measure $\nu$,
and its hyperbolic filling $X$ as well as the uniformization $X_\eps$ for $0<\eps\le\log\al$. 
Following the description  of notions related to
Newtonian and Besov spaces given in Section~\ref{sect-prelim-not}, 
we describe in Section~\ref{sect-lift-up} our method of lifting up the
measure $\nu$ on $Z$ to a measure $\mu$ on $X$. 
In that section we also show that the uniformization $\mube$ of the measure $\mu$, 
given for $\beta>0$ by \eqref{eq-def-mube-intro}, is globally doubling and globally
supports a $1$-Poincar\'e inequality on the uniformized space $X_\eps$
and its completion $\clXeps$.

The trace and extension results from Theorem~\ref{thm-main-intro} are  proved in their 
specific forms as Theorems~\ref{thm-trace-fund}, \ref{thm-trace-fund-refined}
and~\ref{thm:filling-Extension}, respectively. 
Finally, the results stated in Corollary~\ref{cor-harvest-intro}
are obtained in Section~\ref{sect-Besov-applications} by exploiting the perspective of the Besov 
spaces as traces of Newtonian spaces.

The third author communicated the results of this paper with Butler, who
made use of some of the techniques of this paper together with the
tools of the Buseman function to
independently derive some of the results we obtain in
Sections~\ref{sect-hyp-fill}--\ref{sect-uniformize}, 
with a focus on unbounded doubling metric spaces, see~\cite{butler}.
We do not address the issue of unbounded doubling metric spaces
here, but the interested readers may consult~\cite{butler}.
However, his construction
of the hyperbolic filling differs slightly from ours in that he 
requires \eqref{eq-BS-tau}  instead of~\eqref{eq-tilde-m-n+1}.
The results in~\cite{butler} also require that
\[
    \tau \ge \min \biggl\{3, \frac{1}{1-1/\alp}\biggr\}
\]
(the parameter $a$ in~\cite{butler} corresponds to our
  $1/\alp$), but we do not require any such
constraint except in Section~\ref{sec:geodesics-large-tau}.

\begin{ack} 
Parts of this research project were conducted during
2017 and 2018 when N.~S. was a guest professor at Link\"oping University, partially funded
by the  Knut  and  Alice  Wallenberg  Foundation,  and  during  the  parts  of 2019 when 
A. B. and J. B. were Taft Scholars at the University of Cincinnati. 
The authors would like to thank these institutions for their kind
support and hospitality. A.~B. and J.~B. were partially supported by
the Swedish Research Council grants 2016-03424 resp.\ 621-2014-3974 and 2018-04106.  
N.~S. was partially supported by the National Science Foundation (U.S.A.)
grants DMS-1500440 and DMS-1800161.
\end{ack}

\section{Gromov hyperbolic spaces}
\label{sect-Gromov}

In this section we will introduce Gromov hyperbolic spaces and uniform spaces and discuss
relevant background results. In the later part of the paper we will need
background results on upper gradients, Newtonian (Sobolev) spaces, Besov
spaces, Poincar\'e inequalities etc.
This background discussion will be provided in Section~\ref{sect-prelim-not}.

A \emph{curve} is a continuous mapping from an interval. Unless stated otherwise,
we will only consider curves which are defined on compact intervals.
We denote the length of a curve $\ga$  by $\ell(\ga)$,
and a curve is \emph{rectifiable} if it has finite length.
Rectifiable curves can be parametrized by arc length $ds$.

A metric space $X=(X,d)$  is  \emph{geodesic} if 
for each $x,y\in X$ there is a curve $\gamma$ with end points $x$ and
$y$ and length $\ell(\gamma) = d(x,y)$. $X$ is a \emph{length space} if
\[ 
  d(x,y)=\inf_\ga \ell(\ga) 
\quad \text{for all } x,y \in X,
\] 
where the infimum is taken over all curves $\ga$ from $x$ to $y$.

A metric space is \emph{proper} if all closed bounded sets are compact.
We denote balls in $X$ by $B(x,r)=\{y \in X: d(y,x) <r\}$ and the
scaled concentric ball by $\la B(x,r)=B(x,\la r)$. 
In metric spaces it can happen that
balls with different centers and/or radii denote the same set. 
We will however adopt the convention that a ball comes with
a predetermined center and radius. 
Sometimes (especially  when dealing with several
different spaces simultaneously) 
we will write $B_X$ and $d_X$ to indicate that these notions are taken with respect
to the metric space $X$. When we say that $x \in \ga$ we mean that
$x=\ga(t)$ for some $t$. If $\ga$ is noninjective, then this $t$ may not be unique, but we are
always implicitly referring to a  specific such $t$.
If $x_1,x_2 \in \ga$, then $\ga_{x_1,x_2}$ denotes the subcurve between $x_1$ and $x_2$.

\begin{deff}
A complete unbounded geodesic metric space $X$ is
\emph{Gromov hyperbolic} if
there is  a \emph{hyperbolicity constant} $\delta\ge0$ such that whenever 
$[x,y]$, $[y,z]$ and $[z,x]$ are geodesics in $X$,
every point $w\in[x,y]$ lies within a distance $\delta$ of $[y,z]\cup[z,x]$.
\end{deff}

The ideal Gromov hyperbolic space is a metric tree, which 
is Gromov hyperbolic with $\de=0$.
A \emph{metric tree} is a tree where each edge is considered  to be a geodesic of unit length.

\begin{deff}
An unbounded metric space  $X$ is \emph{roughly starlike} if there are some
$x_0\in X$ and $M>0$ such that whenever $x\in X$ there is a geodesic ray
$\gamma$ in $X$, starting from $x_0$, such that $\dist(x,\gamma)\le M$.
A \emph{geodesic ray} is a curve $\ga:[0,\infty) \to X$ with infinite length
  such that $\ga|_{[0,t]}$ is a geodesic for each $t > 0$.
\end{deff}
  
If $X$ is a roughly starlike Gromov hyperbolic space, then
the roughly starlike condition holds for every choice of $x_0$,
although $M$ may change.

\begin{deff}\label{def:uniform}
A nonempty open set $\Om\subsetneq X$ in a metric space $X$
is an \emph{$A$-uniform domain}, with $A\ge1$,  
if for every pair $x,y\in\Om$
there is a rectifiable arc length parametrized
curve $\ga: [0,\ell(\ga)] \to \Om$ with $\ga(0)=x$ and
$\ga(\ell(\ga))=y$ such that $\ell(\ga) \le A d(x,y)$ and
\begin{equation} \label{eq-twisted-cone}
   d_\Om(\ga(t)) \ge \frac{1}{A} \min\{t, \ell(\ga)-t\} \quad \text{for } 0 \le t \le \ell(\ga),
\end{equation}
where
\[ 
  d_\Om(z)=\dist(z,X \setm \Om),
  \quad z \in \Om.
\] 
The curve $\ga$ is said to be an \emph{$A$-uniform curve}. A noncomplete
metric space $(\Om,d)$ is \emph{$A$-uniform} if it is an $A$-uniform
domain in its completion.
\end{deff}

The completion of a locally compact uniform space is always proper,
by Proposition~2.20 in Bonk--Heinonen--Koskela~\cite{BHK-Unif}.
Unlike the definition used in \cite{BHK-Unif},
we do not require uniform spaces to be locally compact.

\medskip

\emph{We assume for the rest of this section that $X$ is a
roughly starlike Gromov $\de$-hyperbolic space. We also fix a point $x_0 \in X$ and let
$M$ be the constant in the roughly starlike condition with respect to $x_0$.}

\medskip

The point $x_0$ will serve as a center for the uniformization $X_\eps$ of $X$.
Following Bonk--Heinonen--Koskela~\cite{BHK-Unif},
we define, for a fixed $\eps>0$,  the \emph{uniformized metric} $d_\eps$ on $X$ as
\[ 
  d_\eps(x,y) = \inf_\ga \int_\ga \rho_\eps\,ds,
  \quad \text{where } \rho_\eps(x)=e^{-\eps d(x,x_0)}
\] 
and the infimum is taken over all rectifiable curves $\ga$ in $X$ joining $x$ to $y$.
Note that if $\ga$ is a compact curve in $X$, then  $\rho_\eps$ is bounded from above 
and away from $0$ on $\ga$, and in particular $\ga$ is rectifiable with respect to 
$d_\eps$ if and only if it is rectifiable with respect to $d$.

\begin{deff}
The set $X$, equipped with the metric $d_\eps$, is denoted by $X_\eps$
and called the \emph{uniformization} of $X$, even when we do not
know whether it is a uniform space.
We let $\clXeps$ be the completion of $X_\eps$, and let
$\partial_\eps X = \clXeps\setminus X_\eps$ be the
\emph{uniformized boundary} of $X_\eps$ (or $X$).
\end{deff}

The uniformization $X_\eps$ need not be a uniform space, as
shown in Lindquist--Shanmugalingam~\cite[Section~4]{LS}.
If $X$ is locally compact and $\eps$  is sufficiently  small, then
$\partial_\eps X$ as a set is
independent of $\eps$ and depends only on the Gromov hyperbolic
structure of $X$, see e.g.\ 
\cite[Section~3]{BHK-Unif}.
The notation adopted in~\cite{BHK-Unif} is $\partial_G X$.
On the other hand, if
$\eps$ is large, then it is possible for $\partial_\eps X$ to change,
see for example Proposition~\ref{prop-eps>log-alp} below.

When writing e.g.\ $B_\eps$, $\diam_\eps$ and $\dist_\eps$,
the subscript $\eps$ indicates that these notions are taken with respect to
$(\clXeps,d_\eps)$.
We also define 
\[
d_\eps(x)=\dist_\eps(x,\partial_\eps X).
\]
The length of the curve $\ga$ with respect to $d_\eps$
is denoted by $\ell_\eps(\ga)$.
The arc length $ds_\eps$ with respect to $d_\eps$ satisfies
\[ 
  ds_\eps = \rho_\eps\,ds.
\] 
It follows that $X_\eps$ is a length space, and thus also $\clXeps$
is a length space.
By a direct calculation (or \cite[(4.3)]{BHK-Unif}),
$\diam_\eps \clXeps= \diam_\eps X_\eps  \le 2/\eps$.

The following important theorem is due to
Bonk--Heinonen--Koskela~\cite{BHK-Unif}; see \cite[Theorem~2.6]{BBSunifPI} for this version.
By the Hopf--Rinow theorem 
(see Gromov~\cite[p.~9]{GromovBook} for a suitable version),  $X$ is proper
if and only if $X$ is locally compact.

\begin{thm} \label{thm-eps0}
Assume that $X$ is locally compact.
There is a constant $\eps_0(\de)>0$, determined by $\de$ alone, such that
if $0 < \eps \le \eps_0(\de)$, then $X_\eps$ is 
an $A$-uniform space for some $A$ depending only on $\de$.
Moreover, $\clXeps$ is a compact geodesic space.

  If $\de=0$, then $\eps_0(0)$ can be chosen arbitrarily large.
\end{thm}

There is also a converse, again due to 
Bonk--Heinonen--Koskela~\cite{BHK-Unif}. Namely,
if $(Y,d)$ is a locally compact uniform space, then equipping $Y$ 
with the \emph{quasihyperbolic metric} $k_Y$ gives a Gromov hyperbolic space, where
\[
k_Y(x,y):=\inf_\gamma\int_\gamma\frac{ds(t)}{\dist_d(\gamma(t),\bdy Y)}
   \quad \text{for }x,y\in Y,
\]
with the infimum taken over all rectifiable curves $\gamma$ in $Y$ with end points
$x$ and~$y$.

We recall, for further reference, the following key estimates from \cite{BHK-Unif}.

\begin{lem}\label{lem:dist-to-eps-bdry}
\textup{(\cite[Lemma~4.16]{BHK-Unif})}
Assume that $X$ is locally compact.
Let $\eps>0$. If $x\in X$, then
\begin{equation}\label{eq-BHK-d-rho}
  \frac{e^{-\eps d(x,x_0)}}{e\eps}\le \dist_\eps(x,\partial_\eps X)=:d_\eps(x)
   \le C_0  \frac{e^{-\eps d(x,x_0)}}{\eps},
\end{equation}
where $C_0=2e^{\eps M}-1$. In particular, $\eps d_\eps(x) \simeq \rho_\eps(x)$, and
$x\to\bdy_\eps X$ with respect to $d_\eps$ if and only if $d(x,x_0)\to\infty$.
\end{lem}

\begin{cor}
  \label{cor-comp-d-deps}
  \textup{({\cite[Corollary~2.9 and its proof]{BBSunifPI}})}
Assume that $X$ is locally compact and that
$0 < \eps \le \eps_0(\de)$,
where $\eps_0(\de)$ is given by Theorem~\ref{thm-eps0}.
Let $x,y\in X$. Then 
\[
\frac{d_\eps(x,y)^2}{d_\eps(x)d_\eps(y)}\simle \exp(\eps d(x,y)).
\]
If $\eps d(x,y)\ge1$ then 
\begin{equation*}    
\exp(\eps d(x,y)) \simeq \frac{d_\eps(x,y)^2}{d_\eps(x)d_\eps(y)},
\end{equation*}
where the comparison constants depend only on $\de$, $M$ and $\eps_0(\de)$.
\end{cor}   

Here and later,  we write $a \simle b$ if there is an implicit
constant $C>0$ such that $a \le Cb$, and analogously $a \simge b$ if $b \simle a$.
We also use the notation $a \simeq b$ to mean $a \simle b \simle a$.

In the later part of the paper we will equip uniformizations of
Gromov hyperbolic spaces 
and their boundaries with doubling measures.
In the first part of the paper, the following metric doubling condition will instead
play a role in a few places, but for most results
no doubling assumption is needed.

A metric space $(Y,d)$ is  \emph{doubling} (or \emph{metrically doubling})
if there is a constant $N_d\ge 1$ such that whenever $z\in Y$ and $r>0$, 
the ball $B(z,r)$ can be covered by at most
$N_d$ number of balls with radius $\tfrac12r$.
Doubling is a uniform version of total boundedness.
In particular, if $Y$ is complete and doubling, then $Y$ is also \emph{proper}.

A Borel regular measure $\mu$ on $Y$ is \emph{doubling}
if  there is a constant $C>0$ such that 
\[ 
  0 < \mu(2B)\le C \mu(B) < \infty
\quad \text{for all balls $B\subset Y$}.
\] 
If $Y$ carries a doubling measure, then $Y$ is necessarily doubling.
The converse is not true in general.
However, if $Y$ is a complete doubling measure space,
then Luukkainen--Saksman~\cite{LuuSak} has shown that $Y$ carries a doubling measure.
For more on doubling spaces and doubling measures,
see Heinonen~\cite[Section~10.13]{Heinonen}.

\section{Construction of hyperbolic fillings}
\label{sect-hyp-fill}

The technique of hyperbolic fillings of doubling 
metric spaces was first considered in Buyalo--Schroeder~\cite[Chapter~6]{BuSch}, 
and then used in Bourdon--Pajot~\cite{BP} and Bonk--Saksman~\cite{BS}. 
The constructions are different in these papers, see below.

We construct the hyperbolic filling as follows:
We assume that a bounded metric space $Z$  is given,
and fix the parameters $\alpha, \tau >1$ and a point $z_0 \in Z$.
By scaling we can assume that $0 \le  \diam Z<1$.
As mentioned in Section~\ref{sect-Gromov},
we later want to equip $Z$ with a doubling measure,
but to begin with no such requirement is needed.

We set $A_0=\{z_0\}$ and note that $Z=B_Z(z_0,1)$.
By a recursive construction
using Zorn's lemma or the Hausdorff maximality principle,
for each positive integer $n$ we can choose 
a maximal $\al^{-n}$-separated set $A_n\subset Z$
such that $A_n\subset A_m$ when $m\ge n \ge 0$. 
A set $A\subset Z$ is \emph{$\al^{-n}$-separated} if $d_Z(z,z')\ge
\al^{-n}$ whenever $z,z'\in A$ are distinct.
Then the balls $B_Z(z,\tfrac12\al^{-n})$, $z\in A_n$, are pairwise disjoint.
Since $A_n$ is maximal, the balls $B_Z(z,\al^{-n})$, $z\in A_n$, cover $Z$.
Here and from now on, $n$ and $m$ will always be nonnegative integers.

Next, we define the ``vertex set''
\begin{equation} \label{eq-Vn}
   V=\bigcup_{n=0}^\infty V_n,
   \quad \text{where }
  V_n=\{(x,n): x \in A_n\}.
\end{equation}
Note that a point $x\in A_n$ belongs to $A_k$ for all $k\ge n$, and so shows 
up as the first coordinate in infinitely many points in $V$. 
Given two different vertices $(x,n),(y,m) \in V$, we say that $(x,n)$ is a
\emph{neighbor} of $(y,m)$ 
(denoted $(x,n)\sim (y,m)$) if and only if 
$|n-m|\le 1$ and
\begin{align}
\tau B_Z(x,\alpha^{-n})\cap \tau B_Z(y,\alpha^{-m}) \ne\emptyset, 
&\quad \text{if } m=n, \label{eq-tilde-m-n}  \\
B_Z(x,\alpha^{-n})\cap B_Z(y,\alpha^{-m}) \ne\emptyset, 
& \quad \text{if } m=n\pm 1. 
\label{eq-tilde-m-n+1}
\end{align}

We let \emph{the hyperbolic filling} $X$ be the graph formed 
by the vertex set $V$ together with the above neighbor relation (edges), 
and consider $X$ to be a \emph{metric graph} where the edges are unit intervals.
As usual for graphs, we do not consider a vertex to be its own neighbor.
The distance between two points in $X$ is the length of the shortest
curve between them. Since $X$ is a metric graph, it is easy to 
see that these shortest curves exist, and thus $X$ is a geodesic space.

If $(x,n) \sim (y,n+1)$ we say that $(y,n+1)$ is a \emph{child} of $(x,n)$ while
$(x,n)$ is a \emph{parent} of $(y,n+1)$. 
(We use this terminology also for rooted trees.) 
In general, each vertex has at least one child, and all vertices
but for the root $v_0=(z_0,0)$ have at least one parent.
An edge $(x,n) \sim (y,m)$ is  \emph{horizontal} if $m=n$ and
\emph{vertical} if $m=n \pm 1$.

We will show that the hyperbolic filling $X$ is always a Gromov
hyperbolic space, but first we compare
our construction with those in
Buyalo--Schroeder~\cite[Chapter~6]{BuSch}, 
Bourdon--Pajot~\cite{BP} and Bonk--Saksman~\cite{BS}.
In \cite{BP}, Bourdon and Pajot
use the same construction
as we do with $\alpha=e$ and $\tau=1$.
It is pointed out in~\cite{BS} that the choice $\tau=1$ causes
problems in the proof of the hyperbolicity of the ``hyperbolic filling'',
more specifically in the proof of \cite[Lemme~2.2]{BP}.
Indeed, in Example~\ref{ex-not-hyperbolic} we construct a
``hyperbolic filling'' with $\tau=1$ and $\alp=2$ which is not
a Gromov hyperbolic space.

According to Bonk--Saksman~\cite{BS}, 
it is enough to enlarge the balls with a factor $>1$,
but they  make the specific choices $\alp=\tau=2$.
Their construction is however slightly different from ours:
Instead of \eqref{eq-tilde-m-n+1}, they require that (with $\tau=2$)
\begin{equation} \label{eq-BS-tau}   
\tau B_Z(x,\alpha^{-n})\cap \tau B_Z(y,\alpha^{-m}) \ne\emptyset, 
 \quad \text{even if } m=n\pm 1.
\end{equation}
Thus, the hyperbolic fillings in \cite{BS} contain more vertical
edges than those considered in this paper (with $\al=\tau=2$).

Buyalo and Schroeder, in~\cite[Chapter~6]{BuSch}, 
use a similar construction with $\alp \ge 6$ 
(i.e.\ $r=1/\al\le\tfrac16$ in their notation) and $\tau=2$, but impose
 a different condition when $m=n \pm 1$, namely (with $\tau=2$)
\[
\tau B_Z(y,\alpha^{-m})  \subset  \tau B_Z(x,\alpha^{-n}),
 \quad \text{if } m=n+1.
\]
Buyalo and Schroeder show that their hyperbolic filling is Gromov hyperbolic.
Bonk and Saksman~\cite{BS} refer to Bourdon--Pajot~\cite{BP} for a proof, 
but mention that the proof in \cite{BP} is problematic for $\tau=1$, as considered 
in \cite{BP}. Both \cite{BP} and~\cite{BS} have stronger assumptions on $Z$ than here.

When $\tau\ge (\al+1)/(\al-1)$, we have more concrete information on
 geodesics in the hyperbolic filling, see Section~\ref{sec:geodesics-large-tau}.
However, a wider choice of $\tau>1$ yields a wider variety of hyperbolic fillings.
For example, if $Z$ is a Cantor set,
obtained as the uniformized boundary of an infinite 
tree and equipped with the induced ultrametric, then its hyperbolic filling with the choice
of $\tau<\al$ gives back the original tree, whereas the choice of
$\tau\ge\al$ does not give a tree as the hyperbolic filling of $Z$, see Section~\ref{sec:trees}.
On the other hand, the estimates regarding traces and extensions of Sobolev
and Besov spaces in Sections~\ref{sect-traces}--\ref{sect-extension}
(related to the uniformization of hyperbolic fillings)
are not affected by the precise values of $\alpha$ and $\tau$.

\medskip

\emph{In the rest of the section, we assume that $Z$ is a metric
space with $\diam Z < 1$ and that $X$ is a 
hyperbolic filling, as constructed above, with parameters $\alp,\tau>1$.}

\medskip

We consider the projection maps $\pi_1:V\to Z$ and $\pi_2:V\to\{0,1,2,\cdots\}$ 
given by $\pi_1((x,n))=x$ and $\pi_2((x,n))=n$.
We also set $v_0:=(z_0,0)$ and use the \emph{Gromov product}
\[
(v|w)_{v_0}=\tfrac12[d_X(v_0,v)+d_X(v_0,w)-d_X(v,w)],  \quad v,w  \in V.
\]

It follows easily from the construction that the
hyperbolic filling is connected. 
The following lemma is a more precise version of this.

\begin{lem}\label{lem:lens-new}
For all $v\in V$ we have $d_{X}(v,v_0)=\pi_2(v)$. 
In particular, $V$ is connected in the graph sense, and $X$ in the metric sense.
\end{lem}

\begin{proof}
The first claim is clear if $v=v_0$.
So suppose that $v=(x,n)$ for some positive integer $n$. 
By the construction of $A_j$, there are $x_j\in A_j$, such that
$x\in B_Z(x_j,\al^{-j})$, $j=0,1,\ldots,n$. In particular, $x=x_n$, 
\[
v_j:=(x_j,j)\in V \quad \text{and} \quad
x\in B_Z(x_{j-1},\al^{-(j-1)}) \cap B_Z(x_j,\al^{-j}) 
\]
for $j=1,\ldots,n$. It follows that $v_0 \sim v_1 \sim \ldots \sim v_n$,
and thus $d_X(v,v_0)\le n$.  As all other paths connecting these two points 
have length at least $n$, we have the required conclusion of the first claim.
The last part follows directly.
\end{proof}

Using that every vertex has at least one child,
the following consequence of the construction in the proof
of  Lemma~\ref{lem:lens-new} is immediate. We will use similar properties
many times in this paper without further ado.

\begin{cor} \label{cor-basic-facts}
\begin{enumerate}
\item
For every vertex $v$ there is a geodesic ray starting at $v_0$ and containing $v$.
\item
Every geodesic ray starting at $v_0$ consists solely of vertical edges.  
\item
Any geodesic from any $x\in X$ to the root $v_0$ contains at most a half of 
a horizontal edge.
\item
$X$ is roughly starlike with $M=\tfrac12$.
\end{enumerate}
\end{cor}

Next, we provide a proof of the hyperbolicity for all parameters $\alp,\tau >1$.
The ideas are similar to those in~\cite{BP} and~\cite{BuSch}.
In particular the following lemma was inspired by
\cite[Lemme~2.2]{BP}. As mentioned above,
when $\tau=1$ it is possible for the ``hyperbolic filling'' to be nonhyperbolic,
see Example~\ref{ex-not-hyperbolic} below.

\begin{lem}   \label{lem-comp-d_Z-(v|w)}
Let $v=(z,n)$ and $w=(y,m)$ be two vertices in $X$.
Then 
\[ 
   \al^{-(v|w)_{v_0}} \simeq d_Z(z,y) + \al^{-n} + \al^{-m},
\] 
with comparison constants depending only on $\al$ and $\tau$.
\end{lem}

\begin{proof}
Without loss of generality, we can assume that $n\le m$.
If $z=y$ then $d_X(v,w)=m-n$ and therefore
$(v|w)_{v_0} = \tfrac12 (n+m-(m-n)) = n$,
and so the statement holds in this case. Assume therefore that $z\ne y$.

Let $l$ be the smallest nonnegative integer such that $\al^{-l}\le \tau-1$, and
$k$ be the smallest nonnegative integer such that $\al^{-k-1} < d_Z(z,y)$. 
For $j=0,1,\ldots$\,, let $z_j,y_j\in A_j$ be such that 
\[
d_Z(z,z_j)<\al^{-j} \quad \text{and} \quad
d_Z(y,y_j)<\al^{-j}.
\]
Clearly, we can choose $z_j=z$ for $j\ge n$ and $y_j=y$ for $j\ge m$.
We shall distinguish two cases:

If $k_0:=\min\{k-l,n\}\ge0$, 
then $\al^{k_0-k}\le \al^{-l}\le \tau -1$, and the triangle inequality shows that
\[
d_Z(z,y_{k_0}) \le d_Z(z,y) + d_Z(y,y_{k_0}) < \al^{-k} + \al^{-k_0} 
= (\al^{k_0-k} + 1) \al^{-k_0} \le \tau \al^{-k_0}.
\] 
Hence $z\in \tau B_Z(z_{k_0},\al^{-k_0}) \cap \tau B_Z(y_{k_0},\al^{-k_0})$,
from which it follows that
\[
(z,n) \sim (z_{n-1},n-1) 
\sim \ldots \sim (z_{k_0},k_0) \sim (y_{k_0},k_0)
\sim \ldots \sim (y_{n-1},m-1) 
\sim (y,m)
\]
where the middle edge may collapse into a single vertex. Thus, 
\[
d_{X}(v,w)\le  n+m+1-2k_0,
\]
and consequently,
\(
(v|w)_{v_0} = \tfrac12 (n+m-d_X(v,w)) \ge k_0-\tfrac12.
\)
If $k_0<0$ then clearly $(v|w)_{v_0}\ge0>k_0$.
In both cases we thus have that
\[
\al^{-(v|w)_{v_0}} \le \al^{-k_0+1/2} \le \al^{1/2} (\al^{l-k} + \al^{-n}) 
\le \al^{l+3/2} (d_Z(z,y) + \al^{-n} + \al^{-m}).
\]
Note that $l$ only depends on $\alp$ and $\tau$.

Conversely, let
\(
w_0 \sim w_1 \sim \ldots \sim w_N
\)
be a geodesic from $v=w_0$ to $w=w_N$. Note that 
\[
d_X(v,w)= N\ge m-n \quad \text{and} \quad (v|w)_{v_0} = \tfrac12(n+m-N).
\]
Moreover, by the construction of the hyperbolic filling,
$\pi_2(w_j)\ge n-j$ and $\pi_2(w_{N-i})\ge m-i$
for all $i,j=0,1,\ldots,N$. Therefore
\[
\al^{-\pi_2(w_j)} \le \al^{j-n} \quad \text{and} \quad
\al^{-\pi_2(w_{N-i})} \le \al^{i-m}.
\]
The triangle inequality then yields that for all $k_1=0,1,\ldots,N$,
\begin{align*}
d_Z(z,y) + \al^{-n} + \al^{-m} 
&\le  \sum_{j=1}^N d_Z(\pi_1(w_{j-1}),\pi_1(w_j)) + \al^{-n} + \al^{-m} \\
&< \al^{-n} + \sum_{j=1}^N (\tau \al^{-\pi_2(w_{j-1})} + \tau \al^{-\pi_2(w_j)}) + \al^{-m}\\
&\le 2\tau \sum_{j=0}^{k_1-1} \al^{j-n} + 2\tau \sum_{i=0}^{N-k_1} \al^{i-m} \\
&< \frac{2\tau}{\al-1} (\al^{k_1-n} + \al^{N-k_1-m+1}),
\end{align*}
where the sum $\sum_{j=0}^{k_1-1}$ is empty when $k_1=0$.
Choosing $k_1$ to be the smallest integer $\ge \tfrac12(N+n-m)$ gives that 
\[
  d_Z(z,y) + \al^{-n} + \al^{-m} < \frac{4\tau \al}{\al-1} \al^{(N-n-m)/2}
  = \frac{4\tau \al}{\al-1} \al^{-(v|w)_{v_0}}.
\qedhere
\]
\end{proof}

\begin{thm}  \label{thm-Gromov-hyp}
There is a constant $C\ge0$, depending only on $\al$ and $\tau$, such that
if $u$, $v$ and $w$ are three vertices in $V$, then
\begin{equation}   \label{eq-char-GH-new}
(v|w)_{v_0}\ge \min\{(v|u)_{v_0},(w|u)_{v_0}\}-C.
\end{equation}
In particular, $X$ is Gromov hyperbolic.
\end{thm}

\begin{proof}
It suffices to prove \eqref{eq-char-GH-new} since 
Gromov hyperbolicity is equivalent to it, see
Bonk--Heinonen--Koskela~\cite{BHK-Unif} 
or Bridson--Haefliger~\cite[p.~411, Proposition~1.22]{BH}.

Let $v=(z,n)$, $w=(y,m)$ and $u=(x,k)$. Then clearly,
\[
d_Z(z,y) + \al^{-n} + \al^{-m} 
\le (d_Z(z,x) + \al^{-n} + \al^{-k})  + (d_Z(x,y) + \al^{-k} + \al^{-m}),
\] 
and Lemma~\ref{lem-comp-d_Z-(v|w)} implies that
\[
\al^{-(v|w)_{v_0}} \simle \al^{-(v|u)_{v_0}} + \al^{-(u|w)_{v_0}} 
\le 2\al^{-\min\{(v|u)_{v_0},(u|w)_{v_0}\}}.
\]
Taking logarithms concludes the proof.
\end{proof}

\section{The uniformized boundary \texorpdfstring{$\bdy_\eps X$}{de X}}
\label{sect-bdyepsX}

\emph{In this section, we assume that $Z$ is a metric space 
  with $\diam Z <1$, and let $X$ be a 
  hyperbolic filling of $Z$ with 
  parameters $\al, \tau>1$.}

\medskip

In this section we will look at how $Z$ relates to the
uniformized boundary $\bdy_\eps X$ of $X$,
see Section~\ref{sect-Gromov} for the definitions.
We use the root $v_0=(z_0,0)$ as the uniformization center $x_0$.
We will show that if $\eps \le \log \alp$ and $Z$ is complete, then
$\bdy_\eps X$ is snow-flake equivalent to $Z$,
with exponent $\sigma=\eps/{\log \alp}$,
see Proposition~\ref{prop-Z-biLip-Xeps} for further details.
In particular, $\bdy_\eps X$ and $Z$ are biLipschitz equivalent 
if $\eps = \log \alp$.

Before showing the equivalence of $\bdy_\eps X$ and $Z$,
we take a look at the case $\eps > \log \alp$.
Towards the end of the section we will also
study how the degree of the vertices in $X$ depends on properties of $Z$.
In this section we are not concerned with whether the 
uniformization $X_\eps$ is a uniform domain  or not.
This question will be considered in Section~\ref{sect-uniformize}.

Note that $\rho_\eps\simeq\rho_\eps(v)$ on every edge $[v,w]\subset X$ and
thus, using also the basic facts from Corollary~\ref{cor-basic-facts},
for all $x\in X$,
\begin{equation}   \label{eq-deps-by-constr}
  d_\eps(x)  \simeq \int_{d_X(x,v_0)}^\infty \rho_\eps\,ds = 
  \frac{1}{\eps} \rho_\eps(x),
\end{equation}
with equality if $x$ is a vertex.

The following result shows that if $\eps > \log \alp$,
then  $\bdy_\eps X$ often becomes just one point.

\begin{prop} \label{prop-eps>log-alp}
Assume that there is $L<\infty$ such that any two points
in the $d_Z$-completion $\clZ$ of $Z$ can be connected by a curve in
$\clZ$ of length at most $L$. If $\eps > \log \alp$, then $\bdy_\eps X$ consists of just one point.
\end{prop}

This shows that when $\eps > \log \alp$ and $Z$ has at least two points,
there is no natural connection between $Z$ and $\bdy_\eps X$.
On the other hand, it is easy to see that if $Z$ consists of finitely many points,
then $\bdy_\eps X$ is also finite and there is a natural
biLipschitz map between these sets. 
(We leave the details regarding finite sets $Z$ to the interested reader.)
In particular, some connectivity assumption is necessary in
Proposition~\ref{prop-eps>log-alp}.

\begin{proof}
Let $F_n=\{x \in X : d_X(x,v_0) \le n\}$ and let
$x,x' \in X \setm F_n$ be arbitrary.
Then there are  $d_X$-geodesics from $x$ and $x'$ to $v_0$
which contain vertices $v,v' \in V_n$, respectively.
By the connectivity assumption on $\clZ$,
there is a sequence $\{w_j\}_{j=0}^m$ of points in $\clZ$ 
such that $v=(w_0,n)$, $v'=(w_m,n)$, $m\le 3L\alp^{n}/(\tau-1)$ and
(if $m \ge 1$)
\[
  d_Z(w_j,w_{j-1})\le \tfrac13(\tau-1) \alp^{-n}, \quad j=1,\ldots,m.
\]
For each $j=0,\ldots,m$ there are points
$w'_j\in Z$ and $z_j \in A_n$ such that 
\[
  d_Z(w'_j,w_j)\le \tfrac13(\tau-1) \alp^{-n} \quad \text{and} \quad
  d_Z(w'_j,z_j) < \alp^{-n}.
\]
We can choose $z_0=w'_0=w_0$ and $z_m=w'_m=w_m$.
As
\begin{align*}
  d_Z(w'_j,z_{j-1}) &\le d_Z(w'_j,w_j)+d_Z(w_j,w_{j-1})+d_Z(w_{j-1},w'_{j-1})+d_Z(w'_{j-1},z_{j-1})\\
  &< (\tau-1) \alp^{-n} + \alp^{-n} = \tau \alp^{-n}
\end{align*}
when $j \ge 1$, we see that $v=(z_0,n) \sim (z_1,n) \sim \ldots \sim (z_m,n)=v'$.
It follows that
\begin{align*}
 d_\eps(x,x') &\le  d_\eps(x,v) + d_\eps(v,v') + d_\eps(v',x')
 < 2 \int_n^\infty e^{-\eps t} \, dt + m e^{-\eps n} \\
 &< \frac{2e^{-\eps n}}{\eps}  + e^{-\eps n} \frac{3L\alp^{n}}{\tau -1}
 = \frac{2e^{-\eps n}}{\eps} + \frac{L}{\tau -1} e^{n(\log \alp -\eps)}.
\end{align*}
Taking supremum over all $x,x'\in X\setm F_n$ shows that
\[
\diam_\eps \bdy_\eps X \le \diam_\eps (X \setm F_n) \to 0,
\quad \text{as } n \to \infty,
\]
since $\eps > \log \alp$.
\end{proof}

The easiest example of a space $Z$ applicable
in Proposition~\ref{prop-eps>log-alp} is $Z=[0,\frac12]$.
The following example shows that it is possible
to apply this proposition to a noncompact complete space.

\begin{example} \label{ex-countable-union-intervals}
Let $Z$ consist of countably many copies of $[0,\frac14]$,
all glued together at $0$, and equipped with the inner length metric,
so that $Z$ is geodesic. It is easy to see that $Z$ is noncompact and 
complete, and that $\diam Z= \frac12$.
Thus Proposition~\ref{prop-eps>log-alp} applies.
\end{example}

Before proceeding, we deduce the following lemma which 
will be used several times in this paper.

\begin{lem}   \label{lem-length-z-to-y}
Assume that $z,y\in Z$ are such that $d_Z(z,y) \le \al^{-k}$ for some 
nonnegative integer $k$.
For $j=0,1,\ldots$\,, let $z_j,y_j\in A_j$ be such that $d_Z(z,z_j)<\al^{-j}$
and $d_Z(y,y_j)<\al^{-j}$. Then 
\begin{equation} \label{eq-1-lem-length}
v_0\sim (z_1,1)\sim\ldots\sim(z_k,k)\sim(z_{k+1},k+1)\sim\ldots.
\end{equation}

Let $l$ be the smallest nonnegative integer such that $\al^{-l}\le\tau-1$.
Then, for any $m,n \ge h:=\max\{k-l,0\}$,
\begin{equation} \label{eq-2-lem-length}
(z_n,n) \sim \ldots \sim (z_h,h) \sim (y_h,h)
\sim \ldots \sim (y_m,m),
\end{equation}
where the middle edge collapses into a single vertex if $z_h=y_h$.
This path $\ga$ has lengths $\ell_X(\ga)$ and
$\ell_\eps(\ga)$ \textup{(}with respect to $d_X$ resp.\ $d_\eps$\textup{)}
satisfying
\begin{align}
  d_X((z_n,n),(y_m,m)) &\le \ell_X(\ga) 
  \le n+m+1-2h
  \le n+m+ 1+2l-2k, \label{eq-lX-path}\\
  d_\eps((z_n,n),(y_m,m)) &\le \ell_\eps(\ga) 
  \le \frac{4}{\eps} e^{-\eps h}
  \le \frac{4}{\eps} e^{-\eps (k-l)}. \nonumber
\end{align}
\end{lem}

Note that $z_j$ and $y_j$  exist by the construction of $A_j$,
and that $y_0=z_0$ with $z_0$ being the unique element in $A_0$ as before.

\begin{proof}
Since $z\in B_Z(z_j,\al^{-j})\cap B_Z(z_{j+1},\al^{-(j+1)})$ for each $j=0,1,\ldots$\,,
we directly see that \eqref{eq-1-lem-length} holds.
By the triangle inequality and the choice of $l$ we see that
\[
  d_Z(z,y_{h}) \le d_Z(z,y) + d_Z(y,y_{h}) < \al^{-k} + \al^{l-k}
  \le\tau \al^{l-k}.
\] 
Moreover, $z\in \tau B_Z(z_{h},\al^{-h})$
and so $(z_{h},h)\sim(y_{h},h)$ or $z_{h}=y_{h}$.
Therefore \eqref{eq-2-lem-length} follows from
\eqref{eq-1-lem-length} (and the corresponding path for $y$).
The estimate \eqref{eq-lX-path} follows directly from \eqref{eq-2-lem-length}.

To estimate $\ell_\eps(\ga)$, recall that by Lemma~\ref{lem:lens-new},
$d_X(v,v_0)=\pi_2(v)$ for all $v\in V$
and that $\rho_{\eps}(x)=e^{-\eps d_X(x,v_0)}$ for all $x \in X$. Hence
\[ 
  \ell_\eps(\ga) \le \int_{h}^n e^{-\eps t} \,dt + \int_{h}^m e^{-\eps t} \,dt
  + 2\int_{h}^{h+1/2} e^{-\eps t} \,dt
  \le \frac{4}{\eps} e^{-\eps h} 
  \le \frac{4}{\eps} e^{-\eps (k-l)}, 
\] 
where the last integral estimates the $d_\eps$-length of the 
(possibly collapsed) horizontal edge  $(z_{h},h)\sim(y_{h},h)$.
\end{proof}

We next show that $\partial_\eps X$
is snowflake equivalent to the completion $\clZ$ of $Z$. 
In particular, $\partial_\eps X$ is \emph{biLipschitz} equivalent to
$\clZ$ when $\eps=\log\alpha$.

\begin{prop}  \label{prop-Z-biLip-Xeps}
Fix $0<\eps\le \log\al$. Then for all vertices $v,w \in X$,
\begin{equation}   \label{eq-dZ-le-deps}
  d_Z(\pi_1(v),\pi_1(w))^\sig  \le C_1 d_\eps(v,w),
  \quad \text{where } 
  C_1=(2\tau\alp)^\sig   
  \text{ and }
  \sigma=\frac{\eps}{\log\al}\le 1.
\end{equation}
Moreover,  $\partial_\eps X$ is snowflake-equivalent to the completion $\clZ$
of $Z$, that is, there is a natural homeomorphism $\Psi:\clZ\to\partial_\eps X$
such that for all $z,y\in \clZ$,  
\begin{equation} \label{eq-Psi}
  \frac{d_Z(z,y)^\sig}{C_1} \le d_\eps(\Psi(z),\Psi(y)) \le C_2 d_Z(z,y)^\sig,
\end{equation}
where $C_2=4\alp^{(l+1)\sig}/\eps$  
and $l$ is the smallest nonnegative integer such that $\al^{-l}\le\tau-1$.
\end{prop}

\begin{proof}
Let 
\[
  w_0\sim w_1 \sim \cdots \sim w_N
\]
be a path $\ga$ in $X$ connecting $w_0=v$ to $w_N=w$.
We can assume without loss of generality that $\pi_2(w_N)\le\pi_2(w_0)$. 
Then by the construction of the hyperbolic filling, 
\begin{equation}  \label{eq-dZ-ge-sum}
  d_Z(\pi_1(w_0),\pi_1(w_N))\le \sum_{i=1}^N d_Z(\pi_1(w_{i-1}),\pi_1(w_{i}))
  \le 2\tau \sum_{i=1}^N \alpha^{-\pi_2(w_i)}.
\end{equation}
Moreover, for each $i$,
\[
  \ell_\eps([w_{i-1},w_{i}]) 
  \ge  \int_{\pi_2(w_{i})}^{\pi_2(w_i)+1} e^{-\eps t}\,dt 
  > e^{-\eps(\pi_2(w_{i})+1)}
   =  e^{-\eps} (\al^{-\pi_2(w_i)})^\sig.
\]
Summing over all $i$ and using the elementary inequality
$(\sum_{i=1}^N a_i )^\sig \le \sum_{i=1}^N a_i^\sig$ for $\sig\le1$,
together with \eqref{eq-dZ-ge-sum}, yields
\[ 
\ell_\eps(\gamma) = \sum_{i=1}^N \ell_\eps([w_{i-1},w_{i}])
\ge e^{-\eps} 
   \sum_{i=1}^N (\al^{-\pi_2(w_i)})^\sig 
\ge e^{-\eps} 
   \biggl( \frac{d_Z(\pi_1(w_0),\pi(w_N))}{2\tau} 
        \biggr)^\sig.
\] 
Taking the infimum over all such paths $\gamma$ gives \eqref{eq-dZ-le-deps}.

Let $z\in \clZ$ and find a sequence $z_j\in A_j$ such that $d_Z(z_j,z)<\al^{-j}$.
Then from Lemma~\ref{lem-length-z-to-y} we see that
$v_0\sim(z_1,1)\sim\cdots\sim(z_j,j)\sim(z_{j+1},j+1)\sim\cdots$\,. 
Moreover, when $i>j$,
\[
 d_\eps((z_i,i),(z_j,j))\le\int_j^\infty e^{-\eps t} \,dt
 = \frac{e^{-\eps j}}{\eps}.
\]

It follows that $\{(z_j,j)\}$ is a Cauchy sequence in $X_\eps$. 
Set $\Psi(z)=\lim_{j\to\infty}(z_j,j)$. From the construction 
of $X_\eps$ it is clear that $\Psi(z)\notin X_\eps$, 
and so this point belongs to $\partial_\eps X$. 
If $z_j^*\in A_j$ is such that $d_Z(z_j^*,z)<\al^{-j}$, then 
$z \in B_Z(z_j,\alp^{-j}) \cap B_Z(z_j^*,\alp^{-j})$ and thus
$(z_j^*,j)\sim(z_j,j)$. Hence 
$d_\eps((z_j^*,j),(z_j,j)) \simle e^{-\eps j}$ and so
$\lim_{j\to\infty}(z_j^*,j)=\lim_{j\to\infty}(z_j,j)$, which shows that 
$\Psi(z)$ is well-defined and gives the map $\Psi:\clZ\to\partial_\eps X$.

When $z,y\in\clZ$ with $z\ne y$ and $k$ is a nonnegative integer  
such that $\al^{-k-1} < d_Z(z,y) \le \al^{-k}$, we pick
$z_j$ and $y_j$ as in Lemma~\ref{lem-length-z-to-y}. Then
applying Lemma~\ref{lem-length-z-to-y}, we obtain
(upon noting that $e^\eps=\alp^\sig$)
\[ 
d_\eps(\Psi(z),\Psi(y))= \lim_{n\to\infty}d_\eps((z_n,n),(y_n,n))
\le \frac{4e^{l\eps}}{\eps} e^{-\eps k}
< \frac{4\alp^{(l+1)\sig}}{\eps} d_Z(z,y)^\sigma,
\] 
i.e.\ the last inequality in \eqref{eq-Psi} holds.

Conversely, applying \eqref{eq-dZ-le-deps} to $v=(y_n,n)$ and $w=(z_n,n)$,
and letting $n\to\infty$, shows the first inequality in \eqref{eq-Psi}. In particular, 
$\Psi$ is injective.

Finally, it only remains to show that $\Psi$ is surjective.
Suppose that $\{x_j\}_{j=1}^\infty$ is a Cauchy sequence of points
in $X_\eps$ with limit $x \in \bdy_\eps X$. Then \eqref{eq-deps-by-constr}
shows that $\lim_{j\to\infty} d_X(v_0,x_j)=\infty$. 
By the construction of $X$, for each $j$ we can find
a vertex $v_j \in X$ such that $d_X(v_j,x_j)\le \tfrac12$. 
Then, with $z_j:=\pi_1(v_j)\in Z$,
we know that $\{v_j\}_{j=1}^\infty$ is a Cauchy 
sequence in $X_\eps$, and so by~\eqref{eq-dZ-le-deps} we
also have that $\{z_j\}_{j=1}^\infty$ is a Cauchy sequence in $Z$.
Hence, there is a point $z_\infty\in \clZ$
such that $\lim_{j\to\infty}d_Z(z_j,z_\infty)=0$. 
Lemma~\ref{lem-length-z-to-y} again
tells us that $\lim_{j\to\infty} d_\eps(v_j,\Psi(z_\infty))=0$, 
and so as $\lim_{j\to\infty} d_\eps(v_j,x_j)=0$,
we must have $x=\Psi(z_\infty)$. Thus $\Psi$ is surjective, completing the proof.
\end{proof}

The construction of the hyperbolic filling as given above works for any
bounded metric space, but the resulting hyperbolic filling can have vertices
with infinite degree. The following proposition shows that if the metric space is
doubling, then the degree is  well-controlled.
Recall that the \emph{degree} of a vertex is the number of neighbors it has.

\begin{prop}\label{prop:bdd-degree} 
The hyperbolic filling $X$ has uniformly bounded degree if and only if $Z$ is doubling.

The uniformity and doubling constants depend only on $\alp$, $\tau$ and each other.
\end{prop}

\begin{proof}
Assume first that $Z$ is doubling.
Let $(x,n)\in X$ and set $A(x,n)\subset V$ to be the collection of all neighbors of $(x,n)$. 
For $(y,m)\in A(x,n)$ we know that $|m-n|\le1$ and 
$\tau B_Z(x,\al^{-n})\cap \tau B_Z(y,\al^{-m})$ is nonempty.
Hence $d_{Z}(x,y)<\tau(\al+1)\al^{-n}$ and so $y\in \tau(\al+1)B_Z(x,\al^{-n})$.
Since $Z$ is assumed to be doubling, there is a positive integer 
$N$ independent of $n$, such that the ball $\tau(\al+1)B_Z(x,\al^{-n})$ can be covered by 
balls $B_1,\ldots,B_N$ of radius $\tfrac12 \al^{-n-1}$, 
see Heinonen~\cite[p.~81]{Heinonen}.

Now, for each $m\in \{n-1,n,n+1\}$, we have
$\tfrac12 \al^{-m} \ge \tfrac12 \al^{-n-1}$ and the balls $B_Z(y,\tfrac12\al^{-m})$,
$y\in A_m$, are pairwise disjoint.
It follows that each ball $B_j$, $j=1,\ldots,N$, can
contain at most one point from $A_m$.
Hence, there are at most $N$ such $y\in A_m$ satisfying $(y,m)\in A(x,n)$.
Since this is true for each $m\in \{n, n\pm1\}$, we conclude that
the cardinality of $A(x,n)$ is at most $3N$, that is,
$X$ is of uniformly bounded degree.

Conversely, assume that $X$ has a uniformly bounded degree.
Let $\z \in Z$ and $0<r\le1$. Let $k$ be the smallest nonnegative integer such that 
$\min\{1,3r\}\le \al^{-k}$. For this choice of $k$, let
$n$ be the smallest integer such that $n\ge k$  and $\al^{-n}\le \tfrac12r$.
Note that $n\le k+l'$ for some $l'$ depending only on $\al$.

Since the balls $B_Z(z,\al^{-n})$ with $z\in A_n$ cover $Z$, for every
$\xi\in B_Z(\z,r)$ there is $z\in A_n$ such that
\[
  \xi \in B_Z(z,\al^{-n}) \subset B_Z(z,\tfrac12r).
\]
Moreover, 
\[
  d_Z(z,\z) \le d_Z(z,\xi) + d_Z(\xi,\z) < \al^{-n} + r \le \tfrac32r.
\]
It follows that the balls $B_Z(z,\al^{-n})$ with $z\in A:=A_n\cap B_Z(\z,\tfrac32r)$ 
cover $B(\z,r)$. To estimate the cardinality of $A$, note that any two points $z,z'\in
A$ satisfy $d_Z(z,z')\le \min\{1,3r\}\le \al^{-k}$.
Lemma~\ref{lem-length-z-to-y} and the above observation that $n\le k+l'$ then imply that
(with $l$ as in Lemma~\ref{lem-length-z-to-y}),
\[
  d_X((z,n),(z',n)) \le 2n+1+2(l-k) \le 1+2(l+l'),
\]
which only depends on $\alp$ and $\tau$.
By assumption there is a uniform bound on the
degrees in $X$ and hence also on the number of vertices in balls with a fixed radius.
Thus there is a uniform bound on the cardinality of $A$, i.e.\ $Z$ is doubling.
\end{proof}

\begin{prop} \label{prop-finite-degree}
Let $\eps>0$ and let $X_\eps$ be the uniformization of $X$, as
defined in Section~\ref{sect-Gromov}. Then the  following are equivalent\/\textup{:}
\begin{enumerate}
\item \label{a-1}
$Z$ is totally bounded\/\textup{;}
\item \label{a-2}
each vertex layer $V_n$ \textup{(}as defined in \eqref{eq-Vn}\textup{)} 
is finite\/\textup{;}
\item \label{a-3}
every vertex in $X$ has finite degree\/\textup{;}
\item \label{a-4}
$X$, and equivalently $X_\eps$, is locally compact\/\textup{;}
\item \label{a-5}
$\clXeps$ is compact.
\setcounter{saveenumi}{\value{enumi}}
\end{enumerate}
If $\eps \le \log \alp$, then the following condition is also
equivalent to those above\/\textup{:}
\begin{enumerate}
\setcounter{enumi}{\value{saveenumi}}
\item \label{a-6}
$\bdy_\eps X$ is compact.  
\end{enumerate}

Moreover, $\clXeps$ is geodesic whenever \ref{a-1}--\ref{a-5} hold.
\end{prop}

Note that we do not require $X_\eps$ to be uniform.
Since \ref{a-1}--\ref{a-3} are independent of $\eps$,
so are \ref{a-4} and \ref{a-5}. Furthermore,
\ref{a-6} is also independent of $\eps$ provided that $\eps \le \log \alp$.
Example~\ref{ex-countable-union-intervals} shows
that \ref{a-6} is not equivalent to the other statements
when $\eps > \log \alp$.

\begin{proof}
\ref{a-1} $\eqv$ \ref{a-2}
It follows directly from the definition of total boundedness
that $Z$ is totally bounded if and only if
all $A_n$ are finite sets,
or equivalently all $V_n$ are finite.

\ref{a-2} $\imp$ \ref{a-3}
Let $v=(x,n)$ be a vertex. Then all neighbors of $v$ belong
to the finite set $V_{n-1} \cup V_n \cup V_{n+1}$,  i.e.\ $v$ has finite degree.

$\neg$\ref{a-2} $\imp$ $\neg$\ref{a-3} Let $m$ be the least
index such that $V_m$ is infinite (which exists as \ref{a-2} fails).
As $V_0=\{v_0\}$, we must have $m \ge 1$.
Each vertex in $V_m$ has at least one parent in $V_{m-1}$.
As $V_{m-1}$ is finite and $V_m$ is infinite, there must be a 
vertex in $V_{m-1}$ which has infinitely many children, and
hence has infinite degree.

\ref{a-3} $\eqv$ \ref{a-4}
This is easily seen to be true. Note that $X$ and $X_\eps$ have the same topology, 
and are thus simultaneously locally compact or not.

\ref{a-2} $\imp$ \ref{a-5}
Let $F_n=\{x \in X : d_X(x,v_0) \le n\}$. Since each $V_j$ is finite,
it follows that $F_n$ is a union of finitely many
compact intervals and so is compact.

Consider $x \in X \setm F_n$ and  let $\ga$ be a geodesic from $x$ to
$v_0$. As $d_X(x,v_0)>n$, there is some point
$v\in\ga$ such that $d_X(v,v_0)=n$. Since $n$ is an integer,
it follows that $v\in V_n$. Hence
\[
   \dist_\eps(x,V_n) \le \int_n^\infty e^{-\eps t} \,dt
  = \frac{e^{-\eps n}}{\eps}.
\]
This inequality also holds for $x \in \bdy_\eps X$,
since $\bdy_\eps X \subset \itoverline{X \setm F_n}$ (where the closure is
with respect to the $d_\eps$ metric).

Let $\eta >0$ and choose $n$ so that $e^{-\eps n}/\eps < \eta$.
Then $\clXeps \setm F_n \subset \bigcup_{y \in V_n} B_\eps(y,\eta)$,
and as $F_n$ is compact and $V_n$ is finite we
see that there is a finite $\eta$-net for $\clXeps$.
Since $\eta >0$ was arbitrary, $\clXeps$ is totally bounded and
thus compact (as it is complete by definition).

\ref{a-5} $\imp$ \ref{a-2}
As $V_n \subset \clXeps$, it is also compact
with respect to the metric $d_\eps$.
The metrics $d_X$ and $d_\eps$ are biLipschitz equivalent
on $V_n$ and thus $V_n$ is compact also with respect to $d_X$.
Since distinct points in $V_n$ are at least a distance $1$ apart,
it follows that $V_n$ is a finite set.

Next,  we assume that $\eps \le \log \alp$ and consider \ref{a-6}.

\ref{a-5} $\imp$ \ref{a-6}
This is trivial.

\ref{a-6} $\imp$ \ref{a-1}
It follows from Proposition~\ref{prop-Z-biLip-Xeps}
that $\clZ$ is homeomorphic to $\bdy_\eps X$,
and is thus also compact. Hence $Z$ is totally bounded.

Finally, as $\clXeps$ is a length space, it follows from Ascoli's theorem
that it is geodesic if it is compact.
\end{proof}

\section{Uniformizing a hyperbolic filling with parameter 
\texorpdfstring{$\eps\le\log\alpha$}{}}
\label{sect-uniformize}

\emph{In this section, we assume that $Z$ is a metric space with $\diam Z <1$, 
and let $X$ be a hyperbolic filling of $Z$ with parameters $\al, \tau>1$.}

\medskip

The aim of this section is to show that the uniformization $X_\eps$
is a uniform domain when $\eps\le\log\alpha$.
(Recall Definition~\ref{def:uniform} of 
uniform spaces and uniform curves.)
Since $X$ is Gromov hyperbolic (by Theorem~\ref{thm-Gromov-hyp}),
it follows from Bonk--Heinonen--Koskela~\cite[Theorem~2.6]{BHK-Unif}
(see Theorem~\ref{thm-eps0})
that $X_\eps$ is a uniform space for sufficiently small $\eps>0$, when
$X$ is locally compact. In the later part of this paper we
are interested in uniformizing (a locally compact) $X$
with respect to $\eps=\log\alpha$, and so we cannot
rely on Theorem~\ref{thm-eps0} or~\cite{BHK-Unif}. 
Therefore we provide a direct proof here,
which also avoids assuming local compactness.
As we saw in Section~\ref{sect-bdyepsX},
it is natural to assume that  $\eps \le \log \alp$ in order
for the boundary $\bdy_\eps X$ to be homeomorphic 
to $Z$ (or its completion $\itoverline{Z}$ if $Z$ is not complete).

\begin{thm}  \label{thm:alph-hyp-fill-eps-uniformize-new}
For all $0<\eps\le\log\alpha$, the uniformized space $X_\eps$ is 
uniform with the uniformity constant depending only on $\alpha$, $\tau$ and $\eps$. 
Moreover, for all $x',x''\in X$,
\begin{equation}   \label{eq-comp-deps-dX}
d_\eps (x',x'') \simeq      e^{-\eps(x'|x'')_{v_0}} \min\{d_X(x',x''),1\},
\end{equation}
with comparison constants depending only on $\al$, $\tau$ and $\eps$. 
\end{thm}

\begin{proof}
First we show that curves given by Lemma~\ref{lem-length-z-to-y},
connecting vertices in $X$, are quasiconvex curves (i.e.\ having length at most a
constant multiple of the distance between their endpoints).
These curves will be subsequently used to construct uniform curves connecting pairs of points that are
``far apart". Throughout the proof, we let $l$ be
the smallest nonnegative integer such that $\al^{-l}\le\tau-1$.

Let $v=(z,n)$ and $w=(y,m)$  be two distinct vertices in $V$ with  
$z\in A_n\subset Z$ and  $y\in A_m\subset Z$.
We can assume that $n\le m$.
Let $k$ be the largest nonnegative integer such that 
$k\le n$ and $d_Z(z,y)\le\al^{-k}$.
Consider a curve $\ga$ connecting $v$ to $w$, as in
\eqref{eq-2-lem-length} of Lemma~\ref{lem-length-z-to-y}.
Then by Lemma~\ref{lem-length-z-to-y} we have
\[
  \ell_\eps(\ga) \le \frac{4}{\eps}e^{-\eps(k-l)} 
  = \frac{4e^{\eps(l+1)}}{\eps} (\al^{-k-1})^\sig,
\quad \text{where } \sigma=\frac{\eps}{\log\al}\le 1.
\] 
We now show that $\ell_\eps(\ga)$ is comparable to $d_\eps(v,w)$.
If $k<n$ then the comparison follows from the choice of $k$ and from 
\eqref{eq-dZ-le-deps}, which together imply that
\[
  \ell_\eps(\ga) \simle (\al^{-k-1})^\sig < d_Z(z,y)^\sig 
  \le (2\tau\al)^\sig d_\eps(v,w).
\] 
On the other hand, if $k=n$, then any injective curve $\ga'$ connecting $v$ to $w$
starts with an edge $v \sim v'$, and thus
\[
  \ell_\eps(\ga') \ge d_\eps(v,v') \ge \int_{k}^{k+1} e^{-\eps t} \,dt 
   \ge e^{-\eps(k+1)} = (\alp^{-k-1})^\sig \simge \ell_\eps(\ga).
\]
Taking infimum over all such curves $\ga'$ shows that 
even when $k=n$ we have 
\begin{equation}\label{eq:gamma-is-quasiconvex}
  \ell_\eps(\ga)\simle d_\eps(v,w).
\end{equation}

Now assume that $x',x''\in X$ are two arbitrary distinct points and consider
an injective curve $\gah$ from $x'$ to $x''$ with $\ell_\eps(\gah) < 2 d_\eps(x',x'')$.
If $\gah$ contains at most one vertex, then $d_X(x',x'')<2$,
and by~\eqref{eq-deps-by-constr},
\begin{equation}   \label{eq-ell-le-deps}
  \ell_\eps(\gah) \simeq d_\eps(x',x'') \simeq e^{-\eps d_X(x',x_0)} d_X(x',x'')
  \simle d_\eps(x) \quad \text{for all } x \in \gah,
\end{equation}
and thus $\gah$ is a uniform curve. By the triangle inequality,
\begin{align*}
  (x'|x'')_{v_0} &= \tfrac12 [d_X(x',x_0) + d_X(x'',x_0) - d_X(x',x'')] \\
  &\ge \tfrac12 [d_X(x',x_0) + d_X(x',x_0) - 2d_X(x',x'')] >  d_X(x',x_0)-2,
\end{align*}
and by the triangle inequality again, $(x'|x'')_{v_0} \le d_X(x',x_0)$. 
Inserting this into~\eqref{eq-ell-le-deps} shows that
\eqref{eq-comp-deps-dX} holds in this case.

If $\gah$ contains at least two vertices, then let $v=(z,n)$ and $w=(y,m)$
be the first resp.\ last vertex in $\gah$.
Let $k$ be the largest nonnegative integer such that both
$k\le \min\{n,m\}$ and $d_Z(z,y)\le\al^{-k}$.
Let $\gamma$ be a curve as described at the beginning of this
proof, connecting the vertices $v$ and $w$.
The desired uniform curve $\gat$ between $x'$ and $x''$ is then obtained
by appending the segments $[x',v]$ and $[w,x'']$ to $\ga$.
By \eqref{eq:gamma-is-quasiconvex}, 
\begin{align*}
\ell_\eps(\gat) & \le d_\eps(x',v) + \ell_\eps(\ga) + d_\eps(w,x'') \\
&\simle d_\eps(x',v) + d_\eps(v,w) + d_\eps(w,x'')
= \ell_\eps(\gah) < 2 d_\eps(x',x'').
\end{align*}
We next show that $\gat$ satisfies the twisted cone condition~\eqref{eq-twisted-cone}
in Definition~\ref{def:uniform}.  Note that $\gat$ need not be injective and
recall from Section~\ref{sect-Gromov} what
$y \in \gat$ and   $\gat_{y,y'}$ mean in such a case.

Let  $v'=(z_h,h)\in\ga$ and $w'=(y_h,h)\in\ga$,
where $h=\max\{k-l,0\}$,  $z_{h}$ and $y_{h}$ are as
given for $v$ and $w$ by Lemma~\ref{lem-length-z-to-y}.
Recall that $\ga$ consists of two vertical segments, one
from $v$ to $v'$ and the other from $w$ to $w'$,
together with the (possibly collapsed) horizontal edge $[v',w']$.
Let $x\in\gat$ be arbitrary and consider the subcurve $\gat_{x,x'}$ of $\gat$  from $x$ 
to $x'$. We shall distinguish three basic situations and their symmetric equivalents.
If $x\in\gat_{x',v}$, then 
\[
  \ell_\eps(\gat_{x,x'}) = d_\eps(x,x') \le e^{-\eps(n-1)} \simeq d_\eps(x).
\]
If $x \in \gat_{v,v'}$, then $\gat_{v,x}$ is a vertical segment,
and hence, using~\eqref{eq-deps-by-constr} and that $\pi_2(v)=n \ge d_X(x,v_0)$,
\[
  \ell_\eps(\gat_{x,x'}) = d_\eps(x',v) + \int_{d_X(x,v_0)}^n e^{-\eps t} \,dt \le
  e^{-\eps(n-1)} + \int_{d_X(x,v_0)}^\infty e^{-\eps t} \,dt \simle d_\eps(x).
\]
If $x\in \gat_{v',w'}$, then $\pi_2(v')=h\le n$,  and thus 
\begin{align*}
\ell_\eps(\gat_{x,x'}) &\le d_\eps(x',v) +  \int_{h}^n e^{-\eps t} \,dt 
       + 2\int_{h}^{h+1/2} e^{-\eps t} \,dt \\
       &\le e^{-\eps(n-1)} + 3 \int_{h}^\infty e^{-\eps t} \,dt \simle e^{-\eps h}
       \simeq d_\eps(x). 
\end{align*}
The case when $x\in \gat_{w',x''}$ is treated similarly. Thus, $\gat$ is a uniform curve.

To prove \eqref{eq-comp-deps-dX} also in this case, note that by 
Lemma~\ref{lem-length-z-to-y}, and using that $k$ is the
largest nonnegative integer $\le \min\{n,m\}$ such that $d_Z(z,y)\le\al^{-k}$,
\begin{align*}
   d_\eps(x',x'') \le d_\eps(x',v) + d_\eps(v,w) + d_\eps(w,x'') 
     &\simle e^{-\eps n}  + e^{-\eps k} + e^{-\eps m}\\
     &\simle d_Z(z,y)^\sig+ \al^{-\sig n} + \al^{-\sig m}.
\end{align*}
Conversely, \eqref{eq-dZ-le-deps} shows that
\[
  2 d_\eps(x',x'') > \ell_\eps(\gah) \ge d_\eps(v,w) \simge d_Z(z,y)^\sig.
\]
Since also $d_\eps(v,w) \simge e^{-\eps n} + e^{-\eps m} = \al^{-\sig n} + \al^{-\sig m}$, 
we get by  Lemma~\ref{lem-comp-d_Z-(v|w)} that
\[
  d_\eps(x',x'') \simeq d_Z(z,y)^\sig + \al^{-\sig n} + \al^{-\sig m}
     \simeq \al^{-\sig (v|w)_{v_0}} \simeq e^{-\eps (x'|x'')_{v_0}}.\qedhere
\]
\end{proof}

\section{Roughly similar equivalence}

\label{sect-quasiisom-equiv}

In this section, we want to show that every locally compact
roughly starlike Gromov hyperbolic space is roughly similar to any
hyperbolic filling of its uniformized boundary when the uniformization
index $\eps$ is small enough to guarantee that the uniformized space is a uniform space.

Let $X$ be a locally compact roughly starlike Gromov hyperbolic space.
For $0 < \eps \le \eps_0(\de)$, where $\eps_0(\de)$ is given by Theorem~\ref{thm-eps0},
we uniformize $X$, with uniformization point $x_0$, 
to obtain $X_\eps$ equipped with the metric $d_\eps$ and having  
boundary $\dXeps$. It will be convenient to use the scaled metric
\[
     \dheps = \frac{\eps}{e} d_\eps
\]
on $\clXeps$, and correspondingly $\distheps$ and $\diamheps$.
A consequence is that $\dheps(x_0)=1/e$ and 
\begin{equation}    \label{eq-diam-heps}
    \frac{1}{e} \le   \diamheps X  \le \frac{2}{e} < 1
\end{equation}
and thus $\diamheps \dXeps \le 2/e<1$.
Note that if $X$ is the half-line, then $\dXeps$ consists of just one point 
and thus has diameter $0$.

\begin{deff}\label{deff-rough-qiso}
 Let $(W,d_W)$ and $(Y,d_Y)$ be metric spaces.
A (not necessarily continuous) map $\Phi:W\to Y$ 
is a \emph{rough similarity}  if there are $C\ge 0$ and $L\ge 1$ such that
every point in $Y$ is at most a distance $C$ from $\Phi(W)$,
and for all $x,x'\in W$,
\begin{equation}  \label{eq-rough-q-isom}
L d_W(x,x')-C \le d_Y(\Phi(x),\Phi(x'))\le Ld_W(x,x')+C.
\end{equation}
\end{deff}

We refer interested readers to
Bonk--Schramm~\cite[Section~2]{BoSc} for more on rough similarity
between Gromov hyperbolic spaces. 

\begin{thm}  \label{thm-rough-isom-new}
Assume that $X$ is a locally compact roughly starlike Gromov hyperbolic space
and that $0 < \eps \le \eps_0(\de)$, where $\eps_0(\de)$
is given by Theorem~\ref{thm-eps0}.
Let $Z=\bdy_\eps X$ be the uniformized boundary of $X$
equipped with the metric $\dheps$.
Let $\Xhat$ be any  hyperbolic filling of $Z$, constructed 
with parameters $\al,\tau>1$ and the maximal $\al^{-n}$-separated sets $A_n\subset Z$.

Consider the mapping $\Phi:X\to\Xhat$, defined for $x\in X$ with 
\[
\al^{-n-1}<\dheps(x)\le \al^{-n},  \quad n=0,1,\ldots,
\]
by $\Phi(x)=(z,n)$, where $z$ is chosen to be a nearest 
point in $A_n$ to $x$, i.e.\  $\dheps(x,z)=\distheps(x,A_n)$.

Then $\Phi$ is a rough similarity with $L=(\log\al)/\eps=1/\sig$ and $C$
depending only on $\al$, $\tau$, $\eps_0(\de)$, $\de$ and $M$.
\end{thm}

By combining Bonk--Schramm~\cite[Theorem~8.2]{BoSc} with
Proposition~\ref{prop-Z-biLip-Xeps}  we know that $X$ and $\Xhat$ are
roughly similar if $\eps>0$ is small enough to guarantee
that $\partial X_\eps$ is snowflake equivalent to the 
visual boundary of $X$ as in~Bonk--Heinonen--Koskela~\cite{BHK-Unif}.
However, as this rough similarity is obtained from~\cite{BoSc} via comparison 
with the cone $\text{Con}(\partial Y_\eps)$, here we 
give a more direct construction of the rough similarity between 
$X$ and $\Xhat$. In so doing, we also demonstrate how the 
parameters $\al$, $\tau$ and $\eps$ affect the rough similarity constants.

We chose to point out the dependence on $\eps_0(\delta)$ separately
in Theorem~\ref{thm-rough-isom-new}
even though $\eps_0(\delta)$ is supposed to be determined solely by $\delta$,
because should a better upper bound for $\eps$ in
Theorem~\ref{thm-eps0} and Corollary~\ref{cor-comp-d-deps} be known
in the future, the result in this theorem will also be valid for
the enlarged range of $\eps$.
Note that $Z= \bdy_\eps X$ is compact by Theorem~\ref{thm-eps0},
and thus so is $A_n$, which shows that the nearest points above exist.
Proposition~\ref{prop-char-tree} shows that $Z$ can be nondoubling.

To prove this theorem, we need   the following lemma.
Let $N_0$ be the smallest integer $\ge 1/{\log \alpha}$.
For each $n\ge N_0$  we set 
\begin{equation}\label{eq:Sn-def}
  S_n = \{ x\in X_\eps: \dheps(x) = \alpha^{-n} \}.
\end{equation}
By the choice of $N_0$ we know that $S_n \ne \emptyset$.

\begin{lem}\label{lem:dist-compare-bdy-new}
Suppose that the assumptions in Theorem~\ref{thm-rough-isom-new} hold.
Fixing $n \ge N_0$, let $z\in \bdy_\eps X$ and let $x$ be 
a nearest point in $S_n$ to $z$. Then 
\[
  \alpha^{-n} = \dheps(x)\le \dheps(z,x) \le \al^{-n}e.
\]
\end{lem}

As noted above, $\clX_\eps$ is compact. Therefore $S_n$
is also compact, and hence the nearest points referred to above do exist.

\begin{proof}
Find a sequence $z_k\in X$ such that $z_k\to z$ with respect to $\dheps$.
Since $X$ is roughly starlike, there is a sequence  of
arc length parametrized $d_X$-geodesic rays $\gamma_k:[0,\infty) \to X$
starting from $x_0$, and a sequence of points $w_k \in \ga_k$, 
such that $\dist_{X}(w_k,z_k)\le M$. For all $y\in [z_k,w_k]$ (where $[z_k,w_k]$ 
is any $d_X$-geodesic from $z_k$ to $w_k$),
\[
  d_X(y,x_0) \ge d_X(z_k,x_0) - M \quad \text{and hence} \quad 
  \rho_\eps(y) \le e^{\eps M} \rho_\eps(z_k).
\]
It follows that 
\[
  \dheps(z_k,w_k) \le \frac{\eps}{e}\int_{[z_k,w_k]} \rho_\eps\,ds 
  \le \frac{M\eps}{e} e^{\eps M} \rho_\eps(z_k) \to 0, \quad \text{as }k\to\infty,
\]
showing that $w_k\to z$. The ray $\gamma_k$  intersects $S_n$
and for sufficiently large $k$ there exists 
$y_k=\ga_k(t_k)\in \ga_k\cap S_n$ such that $w_k\in \ga_k((t_k,\infty))$.
As $y_k\in S_n$, using~\eqref{eq-BHK-d-rho} we see that 
\[
  \al^{-n} = \dheps(y_k)= \frac{\eps}{e}d_\eps(y_k)
  \ge \frac{e^{-\eps d_X(y_k,x_0)}}{e^2} =\frac{e^{-\eps t_k}}{e^2}.
\]
Consequently,
\[
\dheps(w_k, y_k)
  \le \frac{\eps}{e}\int_{t_k}^\infty e^{-\eps t}\, dt
  = \frac{e^{-\eps t_k}}{e} \le \alpha^{-n}e.
\]
This finally implies, since $w_k \to z$, that
\begin{alignat*}{2}
  \distheps(z,S_n) &\le \dheps(z,y_k)
 \le \dheps(z,w_k) +\dheps(w_k,y_k) \\
 &\le  \dheps(z,w_k) + \alpha^{-n}e 
  \to \alpha^{-n}e, &\quad& \text{as }k\to\infty,
\end{alignat*}
which proves the last inequality in the statement of the lemma.
The remaining (in)equalities are clear from the definitions.
\end{proof}

\begin{proof}[Proof of Theorem~\ref{thm-rough-isom-new}]
Let $x,x' \in X$. Let $n$ and $m$ be the largest 
integers such that $\dheps(x) \le \al^{-n}$ and $\dheps(x') \le \al^{-m}$, respectively.
Note that $n,m\ge0$ by \eqref{eq-diam-heps}. Let $v=\Phi(x)=(z,n)$ and  
$w=\Phi(x')=(y,m)$, with $z\in A_n$ and $y\in A_m$. Note that 
\begin{equation} \label{eq-2}
  \dheps(x,z) < 2\al^{-n} \quad \text{and} \quad \dheps(x',y) < 2\al^{-m}.
\end{equation}
The triangle inequality and Lemma~\ref{lem-comp-d_Z-(v|w)} imply that
\[
\dheps(x,x') \le \dheps(z,y) + 2(\al^{-n} + \al^{-m}) \simeq \al^{-(v|w)_{v_0}}.
\]
It then follows from Corollary~\ref{cor-comp-d-deps} that either $\eps d_X(x,x')<1$ or 
\[ 
  \exp (\eps d_X(x,x')) 
  \simeq \frac{d_\eps(x,x')^2}{d_\eps(x)d_\eps(x')}
  \simle \frac{\alp^{-2(v|w)_{v_0}}}{\dheps(x)\dheps(x')} 
  \le \frac{\al^{-(n+m-d_{\Xhat}(v,w))}}{\al^{-n-1} \al^{-m-1}} 
  \simeq \al^{d_{\Xhat}(v,w)}.
\] 
Taking logarithms proves the first inequality in \eqref{eq-rough-q-isom}
of Definition~\ref{deff-rough-qiso}.

For the second inequality we distinguish two cases.
Assume, without loss of generality, that $n\le m$.
If $\dheps(x,x')>(1-1/\al)\dheps(x)$, then 
\(
\al^{-n} < \al \dheps(x) \simle \dheps(x,x')
\)
and hence, by Lemma~\ref{lem-comp-d_Z-(v|w)} and \eqref{eq-2},
\[
  \al^{-(v|w)_{v_0}} \simeq
  \dheps(z,y) + \al^{-n} + \al^{-m} < \dheps(x,x') + 3(\al^{-n} + \al^{-m})
  \simle \dheps(x,x').
\]
This together with Corollary~\ref{cor-comp-d-deps} implies that
\[ 
  \al^{d_{\Xhat}(v,w)} =
  \frac{\al^{-2(v|w)_{v_0}}}{\al^{-n} \al^{-m}} 
  \simle \frac{\dheps(x,x')^2}{\dheps(x) \dheps(x')} 
  \simle \exp (\eps d_X(x,x')).
\] 
Taking logarithms proves the second inequality in~\eqref{eq-rough-q-isom} in this case.

If $\dheps(x,x')\le(1-1/\al)\dheps(x)$, then 
\[
  \al^{-m} \ge \dheps(x') \ge \dheps(x) - \dheps(x,x')
  \ge \frac{\dheps(x)}{\al} > \al^{-n-2},
\]
and hence $n+1\ge m\ge n$. Let $l$ and $t$ be the smallest nonnegative integers 
such that $\al^{-l}\le\tau-1$ and $5 \le \alp^t$. Then, by \eqref{eq-2},
\[
\dheps(z,y) \le \dheps(z,x) + \dheps(x,x') + \dheps(x',y)
    < 2 \alp^{-n} + \dheps(x) + 2 \alp^{-m} \le 5 \alp^{-n} \le \alp^{t-n}.
\]
If $n \ge t$, then by Lemma~\ref{lem-length-z-to-y} (with $k=n-t \ge 0$), 
\[
  d_{\Xhat}(v,w) \le n+m+ 1+2l-2(n-t) \le 2(l+t+1).
\]
If on the other hand  $n < t$, then
\[
  d_{\Xhat}(v,w) \le n+m \le 2n+1 <  2t+1\le 2(l+t+1).
\]
Thus, the second inequality in \eqref{eq-rough-q-isom}
holds also in the case $\dheps(x,x')\le(1-1/\al)\dheps(x)$ by choosing $C\ge 2(l+t+1)$.

To verify that some $C'$-neighborhood of $\Phi(X)$
contains $\Xhat$, note that every point in
$\Xhat$ is within a distance $\tfrac12$ of the set $V$ of vertices in $\Xhat$.
So it suffices to show that if $(y,n)\in V$, then there is some $x\in X$ such that
$d_{\Xhat}((y,n),\Phi(x))\le C''$. 

As before, let $l$ and $t$ be the smallest nonnegative integers 
such that $\al^{-l}\le\tau-1$ and $5 \le \alp^t$. Note that $t\ge1$.
Recall the definition of $S_n$ from~\eqref{eq:Sn-def}.
Let $x$ be a nearest point in $S_{n+l+t}$ to $y$.
Then $\dheps(y,x)\le \al^{-n-l-t}e$ by 
Lemma~\ref{lem:dist-compare-bdy-new}.
By the construction of $\Phi$, the point $x$ has a nearest
point $z_{n+l+t}$ in $A_{n+l+t}$ such that 
$(z_{n+l+t},n+l+t)=\Phi(x)$ and 
$\dheps(x,z_{n+l+t}) < 2\alp^{-n-l-t}$. For $j=n,\ldots,n+l+t-1$, find
$z_j\in A_j$ such that $\dheps(z,z_j)<\al^{-j}$.
Hence, by the choice of $l$ and $t$, 
\[
  \dheps(y,z_n) \le \dheps(y,x) + \dheps(x,z) + \dheps(z,z_n) < 
  \al^{-n-l-t}e + 2\al^{-n-l-t} + \al^{-n} < \tau \al^{-n}.
\]
Since also $y\in\tau B_Z(y,\al^{-n})$ and 
$z \in B_Z(z_j,\alp^{-j})$, $j=n,\ldots,n+l+t$, we see that
\[
  (y,n) \sim (z_n,n) \sim (z_{n+1},n+1) \sim \ldots \sim (z_{n+l+t},n+l+t) = \Phi(x),
\]
where the first  edge collapses into a single vertex if $y=z_n$.
This implies that $d_{\Xhat}((y,n),\Phi(x)) \le l+1$.
\end{proof}

\section{Trees}\label{sec:trees}

In this section, we will obtain a sharper version of 
Theorem~\ref{thm-rough-isom-new} in the
case when $X$ is a tree, namely we get an isometry 
rather than a mere rough similarity, 
provided that the parameters are chosen appropriately.
Note that since $X$ is a rooted tree, $\de=0$ and we can choose $\eps(0)$
arbitrarily, by Theorem~\ref{thm-eps0}, so in this case
there is no upper bound on $\eps$ in Theorem~\ref{thm-rough-isom-new}.
Recall that an \emph{isometry} is a $1$-biLipschitz map, i.e.\
a rough similarity with $L=1$ and $C=0$.

\begin{thm} \label{thm-tree-back}
Let $X$ be a metric tree, rooted at $x_0$, such that every vertex $x\in X$ has at least one child.
Consider the uniformized boundary $\bdy_\eps X$ of $X$, with parameter $\eps>0$,
and let $1<\tau<\al=e^\eps$ be fixed. Let $Z=\bdy_\eps X$ be equipped with the metric 
\[
  d_Z(\z,\xi):= \frac{\eps\tau}{2\al} d_\eps(\z,\xi).
\]
Then $X$ is isometric to any hyperbolic filling of $Z$, constructed with the parameters $\al$ and $\tau$. 
\end{thm}

Note that $\diam Z \le \tau/\al <1$. In Remark~\ref{rmk-tree-large-tau} below
we show that if $\tau \ge \alp$, then the hyperbolic filling is never a tree.
Thus the range $1<\tau<\alp$ for $\tau$ in Theorem~\ref{thm-tree-back} is optimal.

\begin{proof}
Let $\Xhat$ be a hyperbolic filling of $Z$, constructed from a maximal 
$\al^{-n}$-separated set $A_n\subset Z$, with parameters $\tau$ and $\al$.
It suffices to show that the sets of vertices in $X$ and $\Xhat$,
respectively, are isometric, since the extension to the edges is straightforward.

To start with, note that if $\z,\xi\in Z$ have a common ancestor $x\in X$ at
distance  $n\ge0$ from the root $x_0$,  then
\begin{equation}   \label{eq-common-ancestor}
  d_Z(\z,\xi) \le \frac{\eps\tau}{\al}\int_{n}^\infty e^{-\eps t}\,dt = \tau \al^{-n-1} <\al^{-n},
\end{equation}
with equality if and only if $\z$ and  $\xi$  do not have a common ancestor 
at distance $n+1$ from $x_0$.

It follows that for every $n\ge0$, the metric space $Z$ can be written as a finite union of 
open balls of radius  $\al^{-n}$, namely those consisting exactly of all descendants in $Z$ of 
some vertex $x\in X$ with $d_X(x,x_0)=n$.
Moreover, every two points in such a ball
satisfy~\eqref{eq-common-ancestor}, and these balls are disjoint and
can also be written as balls centered at any of the points in it,
with radius $\tau\al^{-n}$.  Indeed, if $d_Z(\z,\eta)<\tau\al^{-n}$, then we know by
\eqref{eq-common-ancestor} and the comment after it that $\z$ and  $\eta$ have a common 
ancestor at distance $n$ from the root $x_0$,
and so $d_Z(\z,\eta)\le \tau\al^{-(n+1)}<\al^{-n}$.
Thus, $A_n$ contains exactly one point in each of these balls  
and this correspondence defines a bijection $F$
between the vertices in $X$ at level $n$ and the set $A_n\subset Z$.
More precisely, $F(x)$ is the unique descendant of $x$ belonging to $A_n$.
Define the mapping $\Fhat$ from vertices in $X$ to vertices in $\Xhat$ as
\[
  \Fhat(x)=(F(x),d_X(x,x_0)).
\]
To show that $\Fhat$ is an isometry 
between the two sets of vertices, assume that $x\sim y$ in  $X$.
Without loss of generality, we may assume that $x$ is a parent of $y$ and that 
$d_{X}(x,x_0)=n$. Then both $F(x)$ and $F(y)$ have $x$ as a common ancestor and hence
by~\eqref{eq-common-ancestor}, $d_Z(F(x),F(y)) < \al^{-n}$, which yields 
\[
  F(y)\in B_Z(F(x),\al^{-n}) \cap B_Z(F(y),\al^{-n-1}).
\]
Therefore $\Fhat(x)=(F(x),n)\sim (F(y),n+1)=\Fhat(y)$.

Conversely, assume that $\Fhat(x)\sim\Fhat(y)$.
Then $|n-m|\le1$, where $n=d_X(x,x_0)$ and $m=d_X(y,x_0)$.
By the construction of $\Xhat$, there exists $\z\in Z$ such that 
\begin{equation*}   
d_Z(\zeta,F(x))<\tau \al^{-n} \quad \text{and} \quad d_Z(\zeta,F(y))<\tau \al^{-m}.
\end{equation*}
As pointed out after~\eqref{eq-common-ancestor}, the first inequality implies
that $\z$ and  $F(x)$ have a common ancestor at distance $n$ from the root $x_0$
and this ancestor must be $x$ since there is only one ray in $X$ from $x_0$ to $F(x)$.
Similarly, $\z$ and  $F(y)$ have $y$ as a common ancestor at distance
$m$ from the root $x_0$.

Because there is only one ray in $X$ from $x_0$ to $\z$, this implies
that $x=y$ when $m=n$ and contradicts the assumption
$\Fhat(x)\sim\Fhat(y)$. Consequently, there are no horizontal edges in $\Xhat$.
If $m\ne n$, then we can assume that $m=n+1$ and  conclude that 
$x$ is the parent of $y$ in the above ray, and so $x\sim y$.
Thus, $\Fhat:X\to\Xhat$ is an isometry.
\end{proof}

\begin{remark} \label{rmk-tree-large-tau}
If $Z$ is not a singleton in Theorem~\ref{thm-tree-back}
(that is, $X$ is not the half-line),  then for 
$\tau\ge\al$ and a scaling of $d_Z$ so that $\diam Z<1$, the hyperbolic filling
$\Xhat$ of $Z$ will always contain horizontal edges and is thus not a tree. 

More precisely, let $d_Z(\z,\xi):= \tfrac12 \eps\kappa d_\eps(\z,\xi)$,
where $0 < \kappa <1$, so that $\diam Z \le \kappa <1$.
Let $l\ge0$ be the smallest integer such that $\ka<\al^{-l}$.
Then  \eqref{eq-common-ancestor}  becomes $d_Z(\z,\xi) \le \kappa \al^{-n}<\al^{-n-l}$,
and so  $Z$ can be written as a finite disjoint union of open balls of radius 
$\al^{-n-l}$, each consisting exactly of all descendants in $Z$ of some vertex
$x\in X$ with $d_X(x,x_0)=n$. Thus, $A_{n+l}$ contains exactly one point in each of these balls.

However, if $\z,\xi\in A_{n+l}$ are descendants 
of two distinct vertices $x,y\in X$ at distance $n$ from the root,
having the same parent, then because $\tau\ge\al$, we have
\[
d_Z(\z,\xi) \le \kappa \al^{1-n}<\al^{1-n-l} \le \tau\al^{-n-l}.
\]
We therefore see that  $\z \in \tau B_Z(\z,\alp^{-n-l}) \cap \tau B_Z(\xi,\alp^{-n-l})$ and
thus the vertices $(\z,n+l)$ and $(\xi,n+l)$ in $\Xhat$ will be neighbors connected by a horizontal edge.
\end{remark}

We also give the following characterizations.

\begin{prop} \label{prop-char-tree}
Let $X$ be a rooted tree such that every vertex $x\in X$ has at least one child,
and let $\eps>0$. Then the following are true\/\textup{:}
\begin{enumerate}
\item \label{y-i}
  The uniformized boundary $\bdy_\eps X$ is compact
  if and only if every vertex has a finite number of children.
\item \label{y-ii}
  The uniformized boundary $\bdy_\eps X$ is doubling
  if and only if there is a uniform bound on the number of children for each vertex.
\end{enumerate}  
\end{prop}    

\begin{proof}
Let $1<\tau<\al=e^\eps$ and let  $\Xhat$ be any hyperbolic filling of $\bdy_\eps X$, with 
parameters $\alp$ and $\tau$. By Theorem~\ref{thm-tree-back}, $\Xhat$ is isometric to $X$.
Part~\ref{y-i} now follows from Proposition~\ref{prop-finite-degree},
while part~\ref{y-ii} follows from Proposition~\ref{prop:bdd-degree}.
\end{proof}

\section{Geodesics in the hyperbolic filling}\label{sec:geodesics-large-tau}

\emph{In this section, except for 
Examples~\ref{example-long-dist-geod} and~\ref{ex-not-hyperbolic},
  we fix an arbitrary parameter $\al>1$ and assume that
\begin{equation} \label{eq-special-tau}
\tau \ge \frac{\alpha+1}{\alpha-1}. 
\end{equation}
We also assume that $Z$ is a metric space with $\diam Z <1$, and
  let $X$ be a hyperbolic filling of $Z$  with the parameters $\al$ and $\tau$.
}

\medskip

In this section we study how the geodesics 
in a hyperbolic filling look like under  the above restriction relating
$\tau$ and $\al$. We do not use these
precise properties of geodesics in the rest of the paper,
and so in the other sections we do not require this limit on $\tau$. However, in other
applications the structure of the geodesics is
quite useful to know. As we gain control of the geodesics in a
straightforward manner under the above constraint on $\tau$, we have 
included this study here as well for the convenience of the reader
and for possible future applications. In Example~\ref{example-long-dist-geod} we show
that most of the properties obtained in
this section can fail when $\tau$ is close to~$1$.
We end the section with Example~\ref{ex-not-hyperbolic}
showing that   when $\tau=1$ it is possible for the ``hyperbolic filling'' to be nonhyperbolic.

Uniformizations will not play any role in this section.
We will only study geodesics between vertices in $X$
and all the geodesics are with respect to the $d_X$-metric.

\begin{lem}  \label{lem-skip-triangle}
Assume that \eqref{eq-special-tau} holds. If $(x,n)\sim (z,n+1) \sim (y,n)$ is a segment 
in a path with $x \ne y$, then $(x,n) \sim (y,n)$ and the path is not a geodesic.
\end{lem}

\begin{proof}
By the hypothesis of this lemma, we have 
$B_Z(x,\al^{-n})\cap B_Z(z,\al^{-(n+1)})\ne\emptyset$
and so, using \eqref{eq-special-tau},
\[
  d_Z(x,z)<\frac{\alpha+1}{\alpha} \alpha^{-n} < \tau\al^{-n}
  \quad \text{and similarly,} \quad d_Z(y,z)<  \tau\al^{-n}.
\] 
It follows that 
$z\in\tau B_Z(x,\alpha^{-n})\cap \tau B_Z(y,\alpha^{-n})$ and consequently 
$(x,n)\sim (y,n)$.  So the length of the path can be
reduced by one by replacing the segment $(x,n)\sim (z,n+1) \sim (y,n)$
with the edge $(x,n)\sim(y,n)$. Hence it cannot be a geodesic.
\end{proof}

\begin{lem} \label{lem:add-rung-in-ladder}
Assume that \eqref{eq-special-tau} holds.
If $(x_1,n)\sim (x_2,n+1) \sim (y_2,n+1) \sim (y_1,n)$
is a segment in a path with $x_1 \ne y_1$, 
then $(x_1,n)\sim(y_1,n)$ and the path is not a geodesic.
\end{lem}

\begin{proof}
Since $(x_2,n+1)\sim (y_2,n+1)$, there exists 
\[
  b\in \tau B_Z(x_2,\al^{-(n+1)})\cap \tau B_Z(y_2,\al^{-(n+1)}).
\] 
Similarly, as $(x_1,n)\sim(x_2,n+1)$, we have
$B_Z(x_1,\al^{-n})\cap B_Z(x_2,\al^{-(n+1)})\ne\emptyset$, and so 
\[
   d_Z(x_1,x_2)<\frac{\al+1}{\al}\al^{-n}.
\]
Therefore,
\[
  d_Z(x_1,b)
  \le d_Z(x_1,x_2) + d_Z(x_2,b) <\frac{\al+1}{\al}\al^{-n}
  +\tau\alpha^{-(n+1)} \le\tau\alpha^{-n}.
\]
Similarly, $d_Z(y_1,b)<\tau \alpha^{-n}$.
Hence $b\in \tau B_Z(x_1,\alpha^{-n})\cap \tau B_Z(y_1,\alpha^{-n})$, 
which shows that $(x_1,n)\sim (y_1,n)$.

Finally, replacing the segment 
$(x_1,n)\sim (x_2,n+1) \sim (y_2,n+1) \sim (y_1,n)$ 
with the edge $(x_1,n)\sim(y_1,n)$ reduces
the length of the path by $2$, and thus the original path is not a geodesic.
\end{proof}

Next, we show that there are no geodesics going
first down (i.e.\ away from the root) and then back up. The first part of this
lemma will also be used when proving the structure Lemma~\ref{lem:simplify}.

\begin{lem}\label{lem:not-down-up}
Assume that \eqref{eq-special-tau} holds.
If $(x,n)\sim(y,n+1)\sim(z,n+1)$ is a segment in a geodesic, 
then there is some $y'\in A_n$ such that $(x,n)\sim(y',n)\sim(z,n+1)$, 
and replacing the segment $(x,n)\sim(y,n+1)\sim(z,n+1)$ with 
$(x,n)\sim(y',n)\sim(z,n+1)$ also gives a geodesic.

Consequently, if $(x_0,n_0)\sim(x_1,n_1)\sim\cdots \sim(x_m,n_m)$
is a geodesic, then there are no indices $0 \le i < j < k \le m$
with $n_j > \max\{n_i,n_k\}$.
\end{lem}

\begin{proof}
By the choice of $A_n$ there is some $y'\in A_n$ such that 
$d_Z(z,y')<\alpha^{-n}$.  We immediately have that $(y',n)\sim(z,n+1)$.
Since we started with a geodesic, we see that $x \ne y'$.
Lemma~\ref{lem:add-rung-in-ladder} therefore
implies that $(x,n)\sim(y',n)$. Since $(x,n)\sim(y,n+1)\sim(z,n+1)$
is a geodesic segment, it follows that $(x,n)\sim(y',n)\sim(z,n+1)$ is
also a geodesic segment, which proves the first claim.

Next, assume that there would exist a geodesic violating 
the second part. Because of Lemma~\ref{lem-skip-triangle}, after
restricting to a subpath we may assume that it is of the form 
\begin{equation} \label{eq-down-up}
  (x_0,n) \sim (x_1,n+1) \sim \ldots \sim (x_{m-1},n+1) \sim (x_m,n).
\end{equation}
Applying the first part of the lemma iteratively shows that
there are $y_1,\ldots, y_{m-2}$ such that
\[
  (x_0,n) \sim (y_1,n) \sim \ldots \sim (y_{m-2},n) \sim  (x_{m-1},n+1) \sim (x_m,n).
\]
As it has the same length as \eqref{eq-down-up},  it is also a geodesic, but 
this contradicts Lemma~\ref{lem-skip-triangle}.
\end{proof}

\begin{lem}\label{lem:lens}
Assume that \eqref{eq-special-tau} holds.
Let $v=(x,n) \in V$. Then the following are true\/\textup{:}
\begin{enumerate}
\item \label{i-a}
If  
\[
  v_0\sim v_1\sim\cdots\sim v_n=v \quad \text{and} \quad 
  v_0\sim w_1\sim\cdots\sim w_n=v
\] 
are two geodesics, then for $j=1,\cdots, n-1$ we have $v_j\sim w_j$. 
\item \label{i-b}
If $\vh=(y,n-1) \sim v$, then there
is a geodesic $v_0\sim v_1\sim\cdots\sim v_n=v$ such that $v_{n-1}=\vh$.
\end{enumerate}
\end{lem}

In Example~\ref{example-long-dist-geod} below
we show that \ref{i-a} can fail drastically if $\tau$ is close to $1$.
Recall that here we assume that $\tau$ satisfies \eqref{eq-special-tau},
and that $v_0$ is the root of $X$. Part~\ref{i-b} holds for any hyperbolic filling, also
without the requirement \eqref{eq-special-tau}.

\begin{proof}
To verify \ref{i-a}, note that by Lemma~\ref{lem:lens-new},
$v_j=(x_j,j)$ and $w_j=(y_j,j)$ for some $x_j, y_j\in A_j$, $j=0,\ldots,n$. Then
\[
  w_{n-1}=(y_{n-1},n-1) \sim v=(x,n) \sim (x_{n-1},n-1)=v_{n-1}
\]
and it follows from Lemma~\ref{lem-skip-triangle} that $v_{n-1}\sim w_{n-1}$. 
Now an inductive application of Lemma~\ref{lem:add-rung-in-ladder} gives 
the desired conclusion.

The second claim follows from the fact that the concatenation of the edge 
$v\sim \vh$ to any of the geodesics connecting $\vh$ to the root vertex $v_0$ 
gives a geodesic.
\end{proof}

\begin{lem}\label{lem:bdd-horizon}
Assume that \eqref{eq-special-tau} holds. Let $n_0$ be the smallest 
positive integer such that 
\[ 
n_0\alpha^{1-n_0} \le  \frac{1}{\alpha+1}.
\] 
Then there is no horizontal geodesic of length $\ge 2n_0$,
i.e., if  $m \ge 2n_0$ and
\begin{equation} \label{eq-10-geod}
  (y_0,n) \sim (y_1,n) \sim \ldots \sim (y_{m},n),
\end{equation}
then \eqref{eq-10-geod} is not a geodesic.
\end{lem}

If we drop the assumption~\eqref{eq-special-tau}
then the proof below shows that the same conclusion holds provided
that $n_0$ is the smallest positive integer  such that 
$2 n_0 \alpha^{1-n_0} \le 1 - 1/\tau$.

\begin{proof}
We may assume that $m=2n_0$. For $0 \le j \le n$,
there is some $x_j \in A_j$ so that $y_0 \in B_Z(x_j,\alpha^{-j})$.
Necessarily, $x_0=z_0$ and $x_n=y_0$.
Then for each $j=0,\cdots, n-1$ we have that
$y_0\in B_Z(x_j,\alpha^{-j})\cap B_Z(x_{j+1},\alpha^{-(j+1)})$, and so
\begin{equation} \label{eq-path-z0-x0}
  (y_0,n) \sim (x_{n-1},n-1) \sim (x_{n-2},n-2) \sim \ldots \sim (x_1,1) \sim v_0.
\end{equation}
Similarly, we can find $z_j \in A_j$ so that 
$y_m \in B_Z(z_j,\alpha^{-j})$, $ 1 \le j <n$, and
\begin{equation} \label{eq-path-z10-x0}
  (y_m,n) \sim (z_{n-1},n-1) \sim (z_{n-2},n-2) \sim 
       \ldots \sim (z_1,1) \sim v_0.
\end{equation}
If $n \le n_0-1$, then combining \eqref{eq-path-z0-x0} and  \eqref{eq-path-z10-x0}
gives us a path from $(y_0,n)$ to $(y_m,n)$ (through the root $v_0$) 
of length at most $2n \le 2(n_0-1)<m$, and thus
\eqref{eq-10-geod} is not a geodesic. 

Now assume that $n \ge n_0$. Since 
\[
  d_Z(y_j,y_{j+1}) < 2\tau \alpha^{-n}  \quad  \text{for } 0 \le j \le n_0-1,
\]
we get that $d_Z(y_0,y_{n_0}) < 2n_0\tau \alpha^{-n}$.
Let $k=n+1-n_0$. Then $1\le k\le n$ and
\begin{align*}
  d_Z(x_{k},y_{n_0}) &\le d_Z(x_{k},y_{0}) + d_Z(y_0,y_{n_0}) \\
  &< \alpha^{-k} + 2n_0\tau \alpha^{-n}
   =\alpha^{-k}(1+2n_0\tau\alpha^{1-n_0}) \le \tau\alpha^{-k},   
\end{align*}
and in particular $y_{n_0} \in \tau B_Z(x_{k},\al^{-k})$.
Similarly, $y_{n_0} \in \tau B_Z(z_{k},\al^{-k})$, and thus
\[
  (x_{k},k)\sim(z_{k},k).
\]
It follows that
\[
   (y_0,n) \sim (x_{n-1},n-1) \sim \ldots \sim (x_{k},k) \sim 
    (z_{k},k) \sim \ldots \sim (y_{2n_0},n) 
\]
is a path of length $2(n-k)+1=2n_0-1<2n_0$ showing that 
\eqref{eq-10-geod} is not a geodesic. 
\end{proof}

More general geodesics can be more complicated. However, we have the 
following lemma, which can be used to obtain potentially simpler geodesics.

\begin{lem}\label{lem:simplify}
Assume that \eqref{eq-special-tau} holds. If $v=(x,k_0)$ and $w=(y,k_m)$ are 
two distinct vertices with $d_X(v,w)=m$, then there is a geodesic 
\[
  v=(x_0,k_0)\sim(x_1,k_1)\sim\cdots \sim (x_{m-1},k_{m-1})\sim(x_m,k_m)=w
\] 
consisting of at most two vertical and one horizontal segments.
More precisely,  there exist 
integers $0\le j_0\le j_1\le m$, with $j_1-j_0\le 2n_0-1$
where $n_0$ is as in Lemma~\ref{lem:bdd-horizon}, such that 
\begin{align*}
  k_{j+1}=k_j-1&\quad \text{for } 0\le j<j_0,\\
  k_{j+1}=k_j+1&\quad \text{for }j_1<j\le m,\\
  k_j=k_{j_0}=k_{j_1}&\quad \text{for }j_0\le j\le j_1.
\end{align*}
\end{lem}

This geodesic minimizes $\sum_j k_j$ over all the geodesics
between $v$ and $w$, and has a similar shape to the path
identified in the latter part of Lemma~\ref{lem-length-z-to-y}.

\begin{proof}
As $X$ is connected, there is a geodesic between $v$ and $w$.
By the second part of Lemma~\ref{lem:not-down-up} this geodesic
does not contain any subpath going first down and then up.
We can therefore modify this geodesic iteratively
using the first part of Lemma~\ref{lem:not-down-up}
to obtain a geodesic of the type described above. That
$j_1-j_0\le 2n_0-1$ follows from Lemma~\ref{lem:bdd-horizon}.
\end{proof}

We end this section with two examples. In the first one we show that, 
in contrast to the fact obtained in Lemma~\ref{lem:lens}\,\ref{i-a},
the distance between different geodesics connecting a pair of points
can be large when $\tau$ is small in comparison with $\alp$.

The first example is tailored so that it can be used iteratively in the second 
example, producing a nonhyperbolic ``hyperbolic filling'' when $\tau=1$.
This illustrates the dependence of the Gromov constant $\de$
on $\al$ and $\tau$.

\begin{example} \label{example-long-dist-geod}
Let $n \ge 3$, $\alp=2$, $1 \le \tau < \tfrac54$
and $\tfrac14(\tau-1)\le \rho <2^{-n-1}$ with $\rho>0$ as well when $\tau=1$. Set 
\[
  z_0=0, \quad z_1=\tfrac12, \quad z_\limpm = \tfrac14\pm \rho \quad \text{and} \quad
  Z= [z_0,\zm] \cup [\zp,z_1]
  =: \Zms \cup \Zps.
\]
Next we choose, $A_0,A_1,\ldots$\,, as follows: 
\begin{alignat*}{2}
  A_0 &=\{z_0\}, \quad A_1 = A_2 = \{z_0,z_1\},\\
  A'_j &=\{2^{-j}k : k=0,1,\ldots,2^{j-1}\}\cap Z, & \quad &j=3,4\ldots,  \\
  A_j &= A'_j, & \quad &j=3,\ldots,n-1,\\
  A_j&=(A_j'\setm \{z^j_{\limpm}\}) \cup \{z_{\limpm}\}, & \quad &j=n,n+1,\ldots,
\end{alignat*}
where $z^j_{\limpm}$ is the point in $A_j'$ closest to $z_{\limpm}$.
We then construct a ``hyperbolic filling'' based on this. In the first three levels we have
\begin{align*}
  &(z_0,0)\sim(z_0,1)\sim(z_1,1)\sim(z_0,0), \quad \\
  &(z_0,2)\sim(z_0,1)\sim(z_1,2) \quad \text{and} \quad (z_0,2)\sim(z_1,1)\sim(z_1,2).
\end{align*}
Next, for $j=2,\ldots,n-1$, the distance between $A_j\cap \Zms$ and $\zp$,
as well as between $A_j\cap \Zps$ and $\zm$, is
\[
  |z^j_{\limminus}-\zp| = |z^j_{\limplus}-\zm| = 2^{-j} + \rho \ge \tau 2^{-j}.
\]
Hence, there are no horizontal edges between the vertices $(z,j)$ and $(y,i)$ with
\begin{equation}   \label{eq-no-horiz-edges}
  z\in A_j\cap \Zms, \quad y\in A_i\cap \Zps \quad \text{for } 2 \le i,j \le n-1.
\end{equation}
On the other hand, since $\zp \in B_Z(\zm,2^{-n}) \cap B_Z(z^{n-1}_{\limplus},2^{1-n})$,
we see that
\begin{equation}  \label{eq-V-edges}
  (\zm,n) \sim (z^{n-1}_{\limplus},n-1) \quad \text{and similarly,} \quad
  (\zp,n) \sim (z^{n-1}_{\limminus},n-1).
\end{equation}
Hence, there are (at least) two upward-directed geodesics $\ga_\limpm$ between
$(\zm,n)$ and $(z_0,1)$, with $\ga_\limminus$ passing only through vertices with 
the  first coordinate in $\Zms$,
while the vertices in $\ga_\limplus$ have the first coordinate in $\Zps$, 
except for the starting and ending vertices.
It follows that the midpoints in $\ga_\limplus$ and  $\ga_\limminus$
have distance $\tfrac12 (n-1)$ to $\ga_\limminus$ and $\ga_\limplus$, respectively,
and so the Gromov constant $\de\ge \tfrac12 (n-1)$.
(If $n$ is even the midpoints of $\ga_\limpm$ are not vertices.)

With $n=3$, this also shows that (at least) for $\tau<\tfrac54$, it can happen that
\[
  (z_0,2) \sim (\zm,3) \sim (z_1,2), \quad \text{while} \quad
  (z_0,2) \not\sim  (z_1,2),
\]
i.e.\ both conclusions in Lemma~\ref{lem-skip-triangle},
the last conclusion in Lemma~\ref{lem:not-down-up},
as well as Lemma~\ref{lem:lens}\,\ref{i-a} all
fail in this case. Similarly, since 
\[
 (z_0,2) \sim (\zp,3) \sim (z_1,3) \sim (z_1,2)
\]
and there is no $y$ such that $(z_0,2) \sim (y,2) \sim (z_1,3)$, the first conclusions in both 
Lemma~\ref{lem:add-rung-in-ladder} and~\ref{lem:not-down-up} fail.

When $n=4$, the geodesics
$(\frac{1}{8},3) \sim (z_\limpm,4) \sim (\frac{3}{8},3)$
are the only geodesics between $(\frac{1}{8},3)$ and $(\frac{3}{8},3)$,
and thus Lemma~\ref{lem:simplify} fails. Moreover, the geodesic
\[
  (z_0,3) \sim (\tfrac18,4) \sim (z_\limminus,4) \sim (\tfrac38,3)
\]
violates both conclusions of Lemma~\ref{lem:add-rung-in-ladder}
and the last conclusion of Lemma~\ref{lem:not-down-up}.
\end{example}

\begin{example} \label{ex-not-hyperbolic}
Let $\alp=2$ and $\tau=1$. Let $\{n_j\}_{j=0}^\infty$ be an increasing sequence 
of positive integers $n_j\ge3$ and let $N_k=\sum_{j=0}^k n_j$.
Choose $\rho_j <2^{-n_j-1}$ and repeat the 
construction in Example~\ref{example-long-dist-geod} with $n=n_j$
and $\rho=\rho_j$, and call the resulting space $Z_j$, $j=0,1,\ldots$\,.

Now, replace the interval $[0,2^{-N_0-1}]\subset Z_0$ by a $2^{-N_0}$-scaled
copy of $Z_1$ to form the new space
\[
  Z_1'=Z_{0} \setm (z'_\limminus,z'_\limplus),
  \quad \text{where } z'_\limpm:= 2^{-N_0}(\tfrac14\pm\rho_1).
\]
The first two points $z_0$ and $2^{-N_0-1}$  in $A_{N_0+1}\subset Z'_1$ are still
next to each other and the corresponding vertices form the horizontal edge
\[
  (z_0,N_0+1)\sim(2^{-N_0-1},N_0+1)
\]
similarly to $(z_0,1)\sim(z_1,1)$ in $Z_0$.

On the other hand, in the following levels $j=N_0+2,\ldots,N_1-1$, 
similarly to \eqref{eq-no-horiz-edges}, there are no horizontal edges
between the left-most interval $[0,z'_\limminus]$ and the rest of $Z'_1$.
At the same time, similarly to \eqref{eq-V-edges}, 
the points $z'_\limpm \in A_{N_1}$
have upward-directed edges both to the interval $[0, z'_\limminus]$ 
and the second interval in $Z'_1$.
Consequently, the vertex $(z'_\limminus,N_1)$ has
two upward-directed geodesics $\ga_\limpm'$ to $(z_0,N_0+1)$,
such that the midpoints of $\ga_\limplus'$ and $\ga_\limminus'$
have distance $\tfrac12 (n_1-1)$ to $\ga_\limminus'$ and $\ga_\limplus'$, respectively.

Next, the interval $[0,2^{-N_1-1}]\subset Z'_1$ can be replaced by a $2^{-N_1}$-scaled
copy of $Z_2$, i.e.\  we get the new space
$Z'_2 = Z'_1 \setm (2^{-N_1}(\frac14-\rho_2),2^{-N_1}(\frac14+\rho_2))$.
Continuing in this way, we obtain a compact doubling space 
\[
  Z'=\bigcap_{j=1}^\infty Z'_j.
\]
Moreover, if $z \in Z'$ and $0<r<\tfrac12$, then $m(B_Z(z,r) \cap Z') \simeq m(B_Z(z,r))$,
where $m$ denotes the Lebesgue measure.

Since $\lim_{j \to \infty} n_j =\infty$, we can for each $k$ find two vertices having two 
upward-directed geodesics $\gat_\limpm$ between them such that the midpoint of 
$\gat_\limplus$ has distance $\ge k$ to $\gat_\limminus$, i.e.\ the hyperbolic filling of $Z'$
does not satisfy the Gromov $\de$-condition, and is thus not Gromov hyperbolic.
\end{example}

\section{Measures, function spaces and capacities}
\label{sect-prelim-not}

\emph{In this section, we assume that $1 \le  p<\infty$ and that $(Y,d)$ is a metric space 
equipped with a complete  Borel  measure $\nu$ 
such that $0<\nu(B)<\infty$ for all balls $B \subset Y$.
We call\/ $Y=(Y,d,\nu)$ a metric measure space.}
\medskip

In the rest of the paper we are interested in studying the
metric space $Z$, considered in the previous sections,
together with a doubling measure on $Z$
and Besov spaces on $Z$ with respect to this
measure. In particular, as mentioned in the introduction,
we wish to associate Besov functions on $Z$ with upper gradient-based
Sobolev functions  on the uniformization $X_\eps$ of the hyperbolic
filling $X$ of $Z$.  In this section we will explain the notions related to 
measures and function spaces.

We follow Heinonen--Koskela~\cite{HeKo98} in introducing
upper gradients as follows (they are referred to 
as very weak gradients in~\cite{HeKo98}).
For proofs of the facts on upper gradients and Newtonian functions
discussed in this section, we refer the reader to Bj\"orn--Bj\"orn~\cite{BBbook} and
Heinonen--Koskela--Shanmugalingam--Tyson~\cite{HKST}.

\begin{deff} \label{deff-ug}
A Borel function $g:Y \to [0,\infty]$ is an \emph{upper gradient} 
of a function $u:Y \to [-\infty,\infty]$ if 
for each nonconstant compact rectifiable curve $\gamma$ in $Y$, we have
\begin{equation} \label{ug-cond}
  |u(x)-u(y)|\le \int_\gamma g\, ds.
\end{equation}
Here $x$ and $y$ denote the two endpoints of $\gamma$. 
The above inequality should be interpreted as
also requiring that $\int_\gamma g\, ds=\infty$ if at least one of $u(x)$ and $u(y)$ is not finite.
If $g$ is a nonnegative measurable function on $Y$
and if (\ref{ug-cond}) holds for \p-almost every curve (see below), 
then $g$ is a \emph{\p-weak upper gradient} of~$u$. 
\end{deff}

We say that a property holds for \emph{\p-almost every curve}
if the family $\Ga$ of all nonconstant compact rectifiable curves for
which the property fails has \emph{zero \p-modulus}, 
i.e.\ there is a Borel function $0\le\rho\in L^p(Y)$ such that 
$\int_\ga \rho\,ds=\infty$ for every curve $\ga\in\Ga$.
The \p-weak upper gradients were introduced in
Koskela--MacManus~\cite{KoMc}.
It was also shown therein
that if $g \in L^p(Y)$ is a \p-weak upper gradient of $u$,
then one can find a sequence $\{g_j\}_{j=1}^\infty$
of upper gradients of $u$ such that $\|g_j-g\|_{L^p(Y)} \to 0$.

If $u$ has an upper gradient in $L^p(Y)$, then
it has a \emph{minimal \p-weak upper gradient} $g_u \in L^p(Y)$
in the sense that $g_u \le g$ a.e.\
for every \p-weak upper gradient $g \in L^p(Y)$ of $u$,
see Shan\-mu\-ga\-lin\-gam~\cite{Sh-harm}.
The minimal \p-weak upper gradient is well defined
up to a set of measure zero.

Following Shanmugalingam~\cite{Sh-rev}, 
we define a version of Sobolev spaces on $Y$. 

\begin{deff}
A function $u:Y \to [-\infty,\infty]$
is in the Newtonian space $\tNp(Y)$ if $\int_Y|u|^p\, d\mu<\infty$ and
$u$ has a \p-weak upper gradient $g\in L^p(Y)$. 
This space
is a vector space and a
lattice, equipped with the seminorm $\Vert u\Vert_{\Np(Y)}$ given by
\[
  \Vert u\Vert_{\Np(Y)} := \biggl(\int_Y|u|^p\, d\nu  +  \inf_g\int_Y g^p\, d\nu \biggr)^{1/p},
\]
where the infimum is taken over all upper gradients $g$ of $u$,
or equivalently all \p-weak upper gradients $g$ of $u$ (see the comments above).

The \emph{Newtonian space}
$\Np(Y)=\tNp(Y)/{\sim}$, where $\sim$ is the equivalence relation on $\tNp(Y)$ given by
$u\sim v$ if and only if $\Vert u-v\Vert_{\Np(Y)}=0$.
To specify the measure with respect to which the Newtonian space
is taken, we will also write $\tNp(Y,\nu)$ and $\Np(Y,\nu)$.
\end{deff}

\begin{deff}\label{deff:capacity}
The \emph{$C_p^{Y}$-capacity} of a set $E\subset Y$ is defined as 
\[
  C_p^{Y}(E) = \inf_{u} \|u\|^p_{\Np(Y)},
\]
where the infimum is taken over all $u\in\tNp(Y)$ satisfying $u\ge1$ on $E$.
\end{deff}

Note that since functions in $\tNp(Y)$ are defined pointwise everywhere, the requirement 
$u\ge1$ on $E$ in the definition of $C_p^Y(E)$ makes sense for an arbitrary set $E\subset Y$.

A property is said to hold \emph{quasieverywhere}
(q.e.\ or $\CpY$-q.e.)\ if the set of all points 
at which the property fails has $\CpY$-capacity zero. 
The capacity is the correct gauge 
for distinguishing between two Newtonian functions. 
If $u \in \tNp(Y)$, then $v \sim u$ if and only if $v=u$ q.e.
Moreover, if $u,v \in \tNp(Y)$ and $u= v$ a.e., then $u=v$ q.e.
That means that the equivalence classes in $\Np(Y)$ are precisely
made up of functions which are equal q.e., and not a.e.\ as in
the usual Sobolev spaces. By an abuse of notation, just as for $L^p$-spaces, 
we will often not distinguish between a function  in $\tNp(Y)$ and 
the corresponding equivalence class in $\Np(Y)$.

\begin{deff}
We say that $Y$ (or the measure $\nu$) supports a \emph{\p-Poincar\'e inequality}
if there exist $C,\lambda>0$ such that for each ball $B=B(x,r)$  and for 
all integrable functions $u$ and upper gradients $g$ of $u$ on $\la B$, 
\[ 
\vint_{B}|u-u_{B}|\, d\nu \le C r \biggl( \vint_{\la B}g^p\, d\nu\biggr)^{1/p},
\] 
where $u_B:= \vint_B u\,d\nu = \nu(B)^{-1} \int_B u\,d\nu$.
\end{deff}

See Bj\"orn--Bj\"orn~\cite{BBbook} and
Heinonen--Koskela--Shanmugalingam--Tyson~\cite{HKST}
for equivalent formulations of the $\CpY$ capacity and
the \p-Poincar\'e inequality.

\begin{remark} \label{rmk-Np-on-graph}
We will primarily be interested in Newtonian spaces on
the uniformization $X_\eps$ of a hyperbolic filling of $Z$,
and on its closure $\clXeps$,
in both cases equipped with the measure $\mube$, $\be >0$,
defined by \eqref{eq-smear-out-muh-be} below.
In particular, each edge in $X$ is measured by a multiple of the Lebesgue
measure. It is then quite easy to see that the only family of 
nonconstant compact rectifiable curves in $X_\eps$ 
which has zero \p-modulus  (with respect to $\mube$) is the empty family.
Functions in Newtonian spaces are absolutely continuous
on \p-almost every line, see Shanmugalingam~\cite{Sh-rev}.
Thus all functions $ u \in \tNp(X_\eps,\mube)$
are continuous on $X_\eps$ and absolutely continuous
on each edge. Moreover, $g_u=|du/ds_\eps|$ a.e.\ on each edge,
  where $ds_\eps$ denotes the arc length with respect to $d_\eps$.
In particular, each equivalence class in $\Np(X_\eps,\mube)$
contains just one function, and that function is continuous.
Moreover, points in $X_\eps$ have positive capacity.
\end{remark}

For functions on $\clXeps$, the situation is not quite as simple,
but the following result will be useful. A function $u$ on $Y$ is 
\emph{$\CpY$-quasicontinuous} if for each
$\eta>0$ there is an open set $G\subset Y $ with $\CpY(G)<\eta$ such that
$u|_{Y \setm G}$ is continuous.  Note that
any $E\subset\clXeps$ with $\CpXeps(E)=0$ must satisfy $E\subset\partial_\eps X$.

\begin{thm}   \label{thm-quasicont}
\textup{(Bj\"orn--Bj\"orn--Shanmugalingam~\cite{BBS-qcont})}
Assume that $Y$ is complete and that $\nu$ is doubling and 
supports a \p-Poincar\'e  inequality.
Then every $u\in\tNp(Y)$ is $\CpY$-quasicontinuous.
Moreover, $\CpY$ is an outer capacity, i.e.
\[ 
  \CpY(E)=\inf_{\substack{ G\supset E \\ G \text{ open}}}  \CpY(G).
\] 
\end{thm}

We will use these facts together with our trace and extension results
to show that Besov functions on $Z$ have $\Capp_{B_{p,p}^\theta(Z)}$-quasicontinuous
representatives (which is defined just as $\CpY$-quasicontinuous),
see Proposition~\ref{prop-Besov-qcont}.
Similarly, we will obtain density of Lipschitz functions and existence
of H\"older continuous representatives in Besov spaces using our trace
and extension results, and corresponding theorems for Newtonian functions.

We now give the definition of Besov spaces on metric measure spaces.

\begin{deff}   \label{def:Besov}
Let $\theta>0$.  We say that $u \in L^p(Y)$ is in the Besov space $B^\theta_{p,p}(Y)$ if 
\begin{equation*} 
  \|u\|_{\theta,p}^p:=
  \int_Y\int_Y \frac{|u(\zeta)-u(\xi)|^p}{d(\zeta,\xi)^{p\theta}} 
      \frac{d\nu(\xi)\, d\nu(\zeta)}{\nu(B(\zeta,d(\zeta,\xi)))} <\infty.
\end{equation*}
\end{deff}

\begin{remark}  \label{rem-Btheta-Banach}
Note that $B^\theta_{p,p}(Y)$ is a Banach space with the norm given by 
\[
  \|u\|_{B^\theta_{p,p}(Y)}=\|u\|_{\theta,p}+\|u\|_{L^p(Y)}.
\] 
Indeed, it is clear that this function space is a normed vector space. 
To see that it is complete, we argue as follows. 
Let $\{u_k\}_{k=1}^\infty$ be a Cauchy sequence in $B^\theta_{p,p}(Y)$. 
Then it is a Cauchy sequence in $L^p(Y)$, and hence it is convergent
to some function $u\in L^p(Y)$.  By passing to a subsequence if necessary, we 
also ensure that  $u_k\to u$ $\nu$-a.e.~in $Y$.
Setting the measure $\nu_0$ on $Y\times Y$ by
\[
  \nu_0(E)=\int_E\frac{d(\nu\times\nu)(\xi,\zeta)}
         {d(\zeta,\xi)^{p\theta} \nu(B(\zeta,d(\zeta,\xi)))},
\]
and defining $v_k:Y\times Y\to\R$ as $v_k(\xi,\zeta)=u_k(\xi)-u_k(\zeta)$, we note that
\[
  \|u_k\|_{\theta,p}=\|v_k\|_{L^p(Y\times Y,\nu_0)}.
\]
Thus, $\{v_k\}_{k=1}^\infty$ is also a Cauchy sequence in the complete space 
$L^p(Y\times Y,\nu_0)$,  and so converges therein to a function $v:Y\times Y\to\R$. 
Again, by passing to yet another subsequence if necessary, we may also assume that
$v_k\to v$ $\nu_0$-a.e.~in $Y\times Y$. Setting $w:Y\times Y\to\R$ by $w(\xi,\zeta)=u(\xi)-u(\zeta)$, 
we see that necessarily $v=w$ $\nu_0$-a.e.~in $Y\times Y$. 
Therefore $v_k\to w$ in $L^p(Y\times Y,\nu_0)$, that is,
$u_k\to u$ in $B^\theta_{p,p}(Y)$. 
\end{remark}

We recall  the following lemma. For a proof see 
Gogatishvili--Koskela--Shan\-mu\-ga\-lin\-gam~\cite[Theorem~5.2 and (5.1)]{GKS} 
(where the factor of $2$ should  be replaced with $\al>1$).

\begin{lem}\label{lem:time-series}
Assume that $\nu$ is doubling and $\theta>0$. If $u\in B^\theta_{p,p}(Y)$, then
\[ 
  \Vert u\Vert_{\theta,p}^p \simeq
       \sum_{n=0}^\infty \int_{Y}\vint_{B(\zeta,\alpha^{-n})}
       \frac{|u(\zeta)-u(\eta)|^p}{\alpha^{-n\theta p}}\, d\nu(\zeta)\, d\nu(\eta).
\] 
\end{lem}

\begin{deff}
We set the \emph{Besov capacity} of $E \subset Y$ to be the number
\[
  \Capp_{B_{p,p}^\theta(Y)}(E):=\inf_u (\|u\|_{\theta,p}^p+\|u\|_{L^p(Y)}^p),
\]
where the infimum is taken over all $u\in B_{p,p}^\theta(Y)$ satisfying 
$u\ge 1$ a.e.\ on a neighborhood of $E$.
\end{deff}

\section{Lifting doubling measures from 
\texorpdfstring{$Z$}{Z} to its hyperbolic filling \texorpdfstring{$X$}{X}}
\label{sect-lift-up}

\emph{From now on, we let $X$ be a hyperbolic filling, constructed with parameters
$\al,\tau>1$, of a compact metric space $Z$ with
$0 <\diam Z <1$ and equipped with a doubling measure $\nu$.
In this section, we also let
$X_\eps$ be the uniformization of $X$ with parameter $0<\eps\le \log\al$.}

\medskip

We now focus on lifting up $\nu$ from $Z$ as follows:
Recall that the vertices in $X$ are denoted $v=(z,n)$, where $z$ belongs
to the maximal $\al^{-n}$-separated set $A_n\subset Z$.
Note that $n$ is the graph distance from the root $v_0:=(z_0,0)$ to $(z,n)$.
For $(z,n)\in V$ we set 
\begin{equation} \label{eq-muh-deff}
  \muh(\{(z,n)\})=\nu(B_Z(z,\alpha^{-n})). 
\end{equation}
The measure $\mu$ on $X$ is then given by ``smearing out" $\muh$
to $X$: for a Borel set $A\subset X$,
\begin{equation}   \label{eq-smear-out-mu}
  \mu(A) =  \sum_{v\in V} \sum_{w\sim v} 
\bigl(\muh(\{v\}) + \muh(\{w\})\bigr) \LL(A\cap [v,w]),
\end{equation}
where $[v,w]$ denotes the unit interval that  connects the two vertices 
$v$ and $w$, and $\LL$ denotes the Lebesgue measure.

Strictly speaking, it is $\muh(v)\deg v$ that is smeared out, but because $X$ has 
uniformly bounded degree (by Proposition~\ref{prop:bdd-degree}) this is  
comparable to $\muh(v)$.

Note that the vertex set $V_n$ of points in $X$ that are at level $n$
from the root is composed of a maximal $\al^{-n}$-separated set of points from $Z$. 
So by the work of Gill and Lopez~\cite{GillLop}, \cite{Lop}, we know that  $V_n$, equipped with 
the neighborhood relationship inherited from $V$ and with the measure
$\muh\vert_{V_n}$, is doubling and that a subsequence converges in the pointed measured
Gromov--Hausdorff sense to a measure $\muh_\infty$ on $Z$ 
such that $\muh_\infty\approx \nu$.

\begin{lem}\label{lem:nbrs-hv-sim-mass}
Let $(z,n), (y,m)\in V$ with $(y,m)\sim(z,n)$. Then 
\[
  \frac{1}{C_d^N} \muh(\{(z,n)\}) \le \muh(\{(y,m)\})\le C_d^N\muh(\{(z,n)\}), 
\]
where $N$ is the smallest integer such that $2^N\ge \alpha(1+\tau)+\tau$
and $C_d$ is the doubling constant associated with $\nu$.
\end{lem}

\begin{proof}
Since the two points are neighbors, we have that $|n-m|\le 1$ and thus $\al^{-n} \le \al^{1-m}$. 
Since $\tau B_Z(z,\al^{-n})\cap \tau B_Z(y,\al^{-m})\ne\emptyset$ by the construction
of the hyperbolic filling, every $\zeta\in B_Z(z,\al^{-n})$ satisfies
\[
  d_Z(\zeta,y) \le d_Z(\zeta,z)+ d_Z(z,y) < \al^{-n} + \tau(\al^{-n}+\al^{-m}) 
     \le 2^N \al^{-m},
\]
and so $B_Z(z,\al^{-n})\subset 2^N B_Z(y,\al^{-m})$.
The doubling property of $\nu$ then implies that
\[
  \muh(\{(z,n)\}) = \nu(B_Z(z,\al^{-n}))
   \le C_d^N \nu(B_Z(y,\al^{-m})) = C_d^N\muh(\{(y,m)\}).
\]
Reversing the roles of $z$ and $y$ in the above argument gives 
the desired double inequality.
\end{proof}

\begin{thm}   \label{thm-doubl+PI-in-graph}
Let $Y$ be a metric graph equipped with the length metric $d$ such that all
edges have length $1$ and assume that $Y$ has uniformly bounded degree, 
i.e.\ every vertex has at most $K$ neighbors. 
Let $\muh$ be a discrete measure defined on the vertices of $Y$ and such that 
$\muh(v) \simeq \muh(w)$ whenever $v\sim w$, 
with comparison constants independent of $v$ and $w$.
Consider the smeared out measure $\mu$ on $Y$ given by~\eqref{eq-smear-out-mu}.

Then  for each $R_0 >0$ there is a constant $C_0\ge1$ such that 
for all balls $B=B(x,r)$ with $r\le R_0$, and every
integrable function $u$ and upper gradient $g$ of $u$ on~$B$, 
\begin{equation}   \label{eq-doubl+PI}
  \mu(2B)\le C_0 \mu(B) \quad \text{and} \quad
  \vint_B |u-u_B|\,d\mu \le C_0r \vint_B g\,d\mu.
\end{equation}
The constant $C_0$ depends only on $R_0$, $K$ and the comparison constants in
$\muh(v) \simeq \muh(w)$. 
\end{thm}

\begin{proof}
Since $Y$ is a length space, it follows from Lemma~3.5 and Theorem~5.3 in
Bj\"orn--Bj\"orn--Shanmugalingam~\cite{BBSunifPI}
that it suffices to consider only the case $R_0=\frac14$.
Let $v$ be a nearest vertex to the center $x$ of $B$, i.e.\ $d(x,v)=\dist(x,V)$.
As $r\le\frac14$, the ball $2B$ contains at most one vertex, namely $v$. Hence,
\begin{equation}   \label{eq-comp-ds-dmu-on-edges}
  B\subset 2B\subset \bigcup_{w\sim v}[v,w] \quad \text{and} \quad d\mu=
  (\muh(\{v\}) + \muh(\{w\}))\,ds \simeq \muh(\{v\})\,ds \text{ on each } [v,w],
\end{equation}
by Lemma~\ref{lem:nbrs-hv-sim-mass}. Thus
\[
  \mu(2B) \simle Kr \muh(\{v\}) \simle \mu(B).
\]

To prove the $1$-Poincar\'e inequality
in~\eqref{eq-doubl+PI}, observe that if $v\notin B$, 
then $B$ is an interval, and so the $1$-Poincar\'e inequality for $B$ follows from the 
$1$-Poincar\'e inequality on $\R$ and the fact that $d\mu=C_B d\LL$ on $B$. On the 
other hand, if $v\in B$, then 
\[
  B = \bigcup_{w\sim v} I_w, \quad \text{where }  I_w=B\cap[v,w].
\]
We therefore obtain from \eqref{eq-comp-ds-dmu-on-edges} and
the definition of upper gradients that for each $w\sim v$,
\begin{align*}
  \int_{I_w} |u-u(v)|\,d\mu
    \le \int_{I_w} \int_{I_w} g(s)\,ds \,d\mu 
  = \mu(I_w)\int_{I_w}g\, ds  &\simle r \muh(\{v\}) \int_{I_w}g\, ds\\
  & \simeq r \int_{I_w}g\, d\mu.
\end{align*}
Summing over all $w\sim v$ yields
\[
  \int_B |u-u(v)|\,d\mu = \sum_{w\sim v} \int_{I_w} |u-u(v)|\,d\mu
  \simle \sum_{w\sim v} r \int_{I_w} g\,d\mu = r \int_B g\,d\mu.
\]
A standard argument based on the triangle inequality
allows us to replace $u(v)$ on the left-hand side
by $u_B$ at the cost of an extra factor 2 on the right-hand side.
\end{proof}

To obtain a doubling measure on $X_\eps$ with respect to the uniformized
metric $d_\eps$, we can equip $X_\eps$ with the uniformized measure
\begin{equation}  \label{eq-def-unif-meas}
  d\muh_\be(x) = \rho_\be(x)\,d\mu(x), \quad \text{where } 
  \rho_\be(x)= e^{-\beta d_X(x,v_0)} \simeq d_\eps(x)^{\be/\eps}.
\end{equation}
Theorems~4.9 and~6.2 in Bj\"orn--Bj\"orn--Shanmugalingam~\cite{BBSunifPI}
(which hold for general Gromov hyperbolic spaces)  then
guarantee that for sufficiently large $\be$, the obtained measure $\muh_\be$ is doubling and
supports a $1$-Poincar\'e inequality on $X_\eps$ as well as on $\clXeps$. 
More precisely, this holds for $\be > \be_0$, 
where $\be_0$ is determined by the limitations given in~\cite[Theorems~4.9 and~6.2]{BBSunifPI} 
based on the doubling constant from Theorem~\ref{thm-doubl+PI-in-graph}.

In considering the specific case of hyperbolic fillings, 
we are able to show that the requirement $\be >\be_0$ can be omitted.
Just as we showed in Theorem~\ref{thm:alph-hyp-fill-eps-uniformize-new} that 
the full range $0 < \eps \le \log \alp$ is possible, 
we can allow for the full range $\beta>0$ in our setting. This will be important in
Sections~\ref{sect-traces}--\ref{sect-Besov-applications} for our trace
and extension results, and their applications. 

Since we consider measures (on hyperbolic fillings) that are constructed from doubling measures on $Z$,
it is possible to start directly with the weighted discrete measure 
\begin{equation}   \label{eq-def-muh-be}
    \rho_\be(v) \muh(\{v\}) = e^{-\be n} \nu(B_Z(z,\al^{-n})) \quad \text{for } v=(z,n)\in V,
\end{equation}
defined on the vertices in  $X$, and smear it out as in \eqref{eq-smear-out-mu}:
\begin{equation}   \label{eq-smear-out-muh-be}
  \mube(A) = \sum_{v\in V} \sum_{w\sim v} 
  \bigl(\rho_\be(v) \muh(\{v\}) + \rho_\be(w) \muh(\{w\})\bigr) \LL(A\cap [v,w]).
\end{equation}
Since $\rho_\be(v)\simeq\rho_\be(w)$ whenever $v\sim w$, the
measures $\mube$ and $\muh_\be$ are clearly comparable,
and the following result also holds for $\muh_\be$.

\begin{thm}    \label{thm-muh-be-doubl-all-be}
For every $\be>0$,  $\mube$ is doubling and supports a $1$-Poincar\'e inequality 
on $X_\eps$ as well as on $\clXeps$.

Furthermore, for all $\zeta\in \bdy_\eps X$ and
$0<r\le 2\diam_\eps X_\eps\le 4/\eps$,
\begin{equation}    \label{eq-muh-be-est-on-bdy}
  \mube(B_\eps(\zeta,r)) \simeq (\eps r)^{\be/\eps} \nu(B_Z(\zeta,(\eps r)^{1/\sig})),
  \quad \text{where } \sigma=\frac{\eps}{\log\al}\le 1
\end{equation}
and the comparison constants depend only on $\eps$, $\be$, $\al$, $\tau$
and the doubling constant associated with $\nu$.
\end{thm}  

In particular, $\nu$ is comparable to the $\be/\eps$-codimensional measure on $\bdy_\eps X$
generated by $\mube$.

\begin{lem} \label{lem-comp-balls-X-Xeps}
Let $0<\eps\le\log\al$. Then
\[ 
B_X\biggl(x,\frac{C_1r}{\eps d_\eps(x)}\biggr)
  \subset B_\eps(x,r)
  \subset B_X\biggl(x,\frac{C_2r}{\eps d_\eps(x)}\biggr), 
  \quad \text{if } x\in X \text{ and } 0< r\le\tfrac12d_\eps(x),
\] 
where $C_1,C_2>0$ are independent of $\eps$.
\end{lem}

\begin{proof}
This was proved when $0 < \eps \le \eps_0(\de)$
for general Gromov hyperbolic spaces
(where $\de$ is the Gromov hyperbolicity constant) in Theorem~2.1
in Bj\"orn--Bj\"orn--Shanmugalingam~\cite{BBSunifPI},
with Remark~2.11 in~\cite{BBSunifPI} showing
that $C_1,C_2>0$ are independent of $\eps$.
That proof only relies on the following facts which hold for
\emph{hyperbolic fillings} for all $0<\eps\le \log\al$:
\begin{itemize}
  \setlength{\itemsep}{0pt}%
  \setlength{\parskip}{0pt plus 1pt}%
  \item
  That $X$ is geodesic, which follows from the definition.
\item
  That $X$ is locally compact and $\clXeps$ is geodesic, which follows
  from Proposition~\ref{prop-finite-degree}, as $Z$ is equipped
  with a doubling measure.
\item 
  Lemma~\ref{lem:dist-to-eps-bdry} (i.e.\ \cite[Lemma~4.16]{BHK-Unif})
  which holds for arbitrary $\eps>0$. \qedhere
 \end{itemize}
\end{proof}

\begin{proof}[Proof of Theorem~\ref{thm-muh-be-doubl-all-be}]
We first concentrate on the doubling property.
We need to consider three types of balls, namely, 
subWhitney balls, balls centered in (or near) $\bdy_\eps X$  and intermediate balls. 
Recall that the graph $X$ has uniformly bounded degree, by Proposition~\ref{prop:bdd-degree}.

1.\ For \emph{subWhitney balls}, that is, balls $B_\eps(x,r)$
with $r\le\tfrac14 d_\eps(x)$, we note that $2C_2r/\eps d_\eps(x) \le C_2/2\eps$.
Hence, Lemmas~\ref{lem:nbrs-hv-sim-mass}, \ref{lem-comp-balls-X-Xeps}
and Theorem~\ref{thm-doubl+PI-in-graph} imply that 
\begin{align*}
  \mube(B_\eps(x,2r)) &\le \mube\biggl( B_X\biggl(x,\frac{2C_2r}{\eps d_\eps(x)}\biggr) \biggr)
  \simeq \rho_\be(x) \mu\biggl( B_X\biggl(x,\frac{2C_2r}{\eps d_\eps(x)}\biggr) \biggr) \\
  &\simeq \mube\biggl( B_X\biggl(x,\frac{C_1r}{\eps d_\eps(x)}\biggr) \biggr) \le \mube(B_\eps(x,r)).
\end{align*}
2.\ If $x\in \clX_\eps$ and $r\ge2d_\eps(x)$, then for some $\zeta\in\bdy_\eps X$,
\[
  B_\eps(x,2r) \subset B_\eps(\zeta,\tfrac52r) \quad \text{and} \quad 
  B_\eps(\zeta,\tfrac12r) \subset B_\eps(x,r).
\]
It therefore suffices to estimate $\mube(B_\eps(\zeta,r))$ for all $\zeta\in\bdy_\eps X$ 
and $0<r\le2\diam_\eps X_\eps$.
From the construction of $\mube$, and using also the uniformly bounded degree of
$X$, it is clear that
\[ 
 \mube(B_\eps(\zeta,r)) \simeq \sum_{v\in V\cap B_\eps(\zeta,r)} e^{-\be \pi_2(v)} \muh(v).
\] 
Let  $v=(z,n)\in V\cap B_\eps(\z,r)$. 
Then $e^{-\eps n} = \eps d_\eps(v)< \eps r$ by \eqref{eq-deps-by-constr}, and hence
$n \ge N$, where $N$ is the smallest nonnegative integer such that 
\begin{equation}   \label{eq-def-N}
  N \ge \frac{1}{\eps}\log \frac{1}{\eps r}.
\end{equation}
Let $z_j\in A_j\subset Z$ be such that $d_Z(z_j,\zeta)<\al^{-j}$, $j=0,1,\ldots$\,.
Proposition~\ref{prop-Z-biLip-Xeps} shows that
\[
  d_Z(z,z_j)^\sig \le (2\tau\al)^\sig d_\eps(v,(z_j,j)) 
   \le (2\tau\al)^\sig \bigl( d_\eps(v,\zeta)+d_\eps(\zeta,(z_j,j)) \bigr). 
\]
Since the path $(z_0,0)\sim\cdots\sim(z_j,j)\sim\cdots$ is a geodesic ray in $X$ ending 
at $\zeta$ (see Lemma~\ref{lem:lens-new} and Proposition~\ref{prop-Z-biLip-Xeps}), we have
\begin{equation}\label{eq:ObviousPerhaps}
  d_\eps(\zeta,(z_j,j))= e^{-\eps j}/\eps.
\end{equation} 
Letting $j\to\infty$ then shows that
\[ 
d_Z(z,\zeta) \le d_Z(z,z_j) + d_Z(z_j,\zeta) 
\le 2\tau \al \biggl( r + \frac{e^{-\eps j}}{\eps} \biggr)^{1/\sig}  + \al^{-j} 
\to 2\tau \al r^{1/\sig}.
\] 
We therefore obtain that
\[
  \mube(B_\eps(\zeta,r)) \simle \sum_{n\ge N} e^{-\be n}
       \sum_{z\in A_n\cap B_Z(\zeta,3\tau\al r^{1/\sig})} \nu(B_Z(z,\al^{-n})).
\]
Since $\al^{-n}= (e^{-\eps n})^{1/\sig} < (\eps r)^{1/\sig}$, the
bounded overlap of the balls $B_Z(z,\al^{-n})$ in $A_n\subset Z$ and 
the doubling property of $\nu$ now imply that
\begin{equation}   \label{eq-mube-eps-r}
  \mube(B_\eps(\zeta,r)) \simle \sum_{n\ge N} e^{-\be n} \nu(B_Z(\zeta,(\eps r)^{1/\sig}))
  \simeq (\eps r)^{\be/\eps} \nu(B_Z(\zeta,(\eps r)^{1/\sig})).
\end{equation}

To verify the reverse comparison, observe that
by~\eqref{eq-def-N} and~\eqref{eq:ObviousPerhaps} we have
$d_\eps(\z,(z_N,N))=e^{-\eps N}/\eps < r$. It follows that the 
edge $(z_N,N)\sim(z_{N+1},N+1)$ is contained in $B_\eps(\zeta,r)$. Consequently,
using the doubling property of $\nu$ and the fact that 
$d_Z(\zeta,z_N)<\al^{-N} = (e^{-\eps N})^{1/\sig}
\simeq (\eps r)^{1/\sig}$, we conclude from \eqref{eq-smear-out-muh-be} that
\begin{align*}
  \mube(B_\eps(\zeta,r)) &\simge e^{-\be N} \muh(\{(z_N,N)\}) \\
  &\simeq (\eps r)^{\be/\eps} \nu(B_Z(z_N,\al^{-N}))
  \simeq (\eps r)^{\be/\eps} \nu(B_Z(\zeta,(\eps r)^{1/\sig})),
\end{align*}
which, together with \eqref{eq-mube-eps-r}, proves 
\eqref{eq-muh-be-est-on-bdy}.

3. If $x\in X_\eps$ and $\tfrac14d_\eps(x) \le r\le 2d_\eps(x)$, then clearly
\[
  B_\eps(x,\tfrac14 d_\eps(x)) \subset B_\eps(x,r)
  \subset B_\eps(x,2r) \subset B_\eps(x,4d_\eps(x)).
\]
From Proposition~\ref{prop-finite-degree} we know that $\clXeps$ is
compact, and thus there is $\zeta \in \bdy_\eps X$
such that $d_\eps(\zeta,x)=d_\eps(x)$.
Let $\ga$ be a $d_\eps$-geodesic from $x$ to $\z$ and let
$v=(z,k)$ be the vertex in $\ga$ nearest to $x$.
As in case~1, using \eqref{eq-smear-out-muh-be},
Lemma~\ref{lem:nbrs-hv-sim-mass}, 
and the uniform boundedness of the degrees, we see that
\begin{equation} \label{eq-mube-2}
  \mube(B_\eps(x,\tfrac14 d_\eps(x))) \simeq \rho_\be(x) \muh(\{v\})
   =  \rho_\be(x) \nu(B_Z(z,\al^{-k})).
\end{equation}
On the other hand, by \eqref{eq-muh-be-est-on-bdy} (proved when considering 
case~2 above),
\[
  \mube(B_\eps(x,4 d_\eps(x))) \le \mube(B_\eps(\zeta,5 d_\eps(x)))
  \simeq (\eps d_\eps(x))^{\be/\eps} \nu(B_Z(\zeta,(5\eps d_\eps(x))^{1/\sig})).
\]
Note that the right-hand side of the above is comparable to the right-hand side 
of \eqref{eq-mube-2} since 
\(
d_\eps(z,\z)\le d_\eps(z,v) + d_\eps(v,\z) \le 2d_\eps(x),
\)
$\nu$ is doubling, and 
\[
  (\eps d_\eps(x))^{\be/\eps} \simeq \rho_\eps(x)^{\be/\eps}= \rho_\be(x)
  \quad \text{and} \quad
  \alp^{-k} = e^{-\eps k/\sig}
  \simeq \rho_\eps(x)^{1/\sig}
  \simeq (\eps d_\eps(x))^{1/\sig}.
\]
The doubling property of $\mube$ 
now follows directly in all three cases from the above estimates.

To show that $\mube$ supports a $1$-Poincar\'e inequality on $\clXeps$,
we proceed as in the
proof of Lemma~6.1 in Bj\"orn--Bj\"orn--Shanmugalingam~\cite{BBSunifPI}.
This is possible for all $\be>0$, not only $\be>\be_0$ as in \cite{BBSunifPI}, 
since we already know that $\mube$ is doubling on $X_\eps$.
Together with Theorem~\ref{thm-doubl+PI-in-graph},
it shows that there exists $c_0>0$ such that \eqref{eq-doubl+PI} holds
for all subWhitney balls $B=B_\eps(x,r)$ with $x\in X$, the measure $\mu_\beta$,
and $0<r\le c_0 d_\eps(x)$. Since $X_\eps$ is a uniform length space
and $\mube$ is doubling on $X_\eps$,
we conclude from  \cite[Proposition~6.3]{BBSunifPI} that
$\mube$ supports a $1$-Poincar\'e
inequality on $X_\eps$ as well as on $\clXeps$.
\end{proof}

\begin{cor}    \label{cor-mube-all-balls}
With the assumptions as in Theorem~\ref{thm-muh-be-doubl-all-be}, we
have for all $x\in \clX_\eps$ and $0<r\le 2\diam_\eps X_\eps$,
\begin{align*}  
  &\mube(B_\eps(x,r)) \\ 
  &\quad \simeq \begin{cases}
    (\eps r)^{\be/\eps}\nu(B_Z(\zeta,(\eps r)^{1/\sig})),
      &\text{if $r\ge d_\eps(x)$ and $\zeta\in Z$ is a nearest point to $x$}, \\
    r (\eps d_\eps(x))^{\be/\eps-1} \muh(\{v\}), &\text{if $r\le d_\eps(x)$ and 
                       $v\in X$ is a nearest vertex to $x$}. \end{cases}
\end{align*}
In both cases, the nearness is with respect to the metric $d_\eps$.
\end{cor}

Recall from \eqref{eq-deps-by-constr} that  
$\al^{-n} = (e^{-\eps n})^{1/\sig} = (\eps  d_\eps(v))^{1/\sig}$ if $v=(z,n)\in V$.
Moreover, if $\zeta\in Z$ and $v=(z,n) \in V$ are nearest points to $x$ in $Z$ and $V$,
respectively, then by Proposition~\ref{prop-Z-biLip-Xeps},
\[
  d_Z(\zeta,z)^\sig\simeq d_\eps(\zeta,z) \simle d_\eps(x) \simeq d_\eps(v).
\]
It therefore follows from \eqref{eq-muh-deff} and the doubling property of $\nu$ that
\[
  \muh(\{v\})=\nu(B_Z(z,(\eps  d_\eps(v))^{1/\sig}) \simeq \nu(B_Z(\z,(\eps d_\eps(x))^{1/\sig})),
\]
which further simplifies the formula in Corollary~\ref{cor-mube-all-balls}.
Also note that since $\nu$ is doubling,  $B_Z(\z,(\eps d_\eps(x))^{1/\sig}))$
can be replaced by any ball 
$B_Z(\xi,(\eps d_\eps(x))^{1/\sig}))$ with $\xi\in Z$ such that $d_\eps(\xi,x)\simle d_\eps(x)$,
and that $\nu(B_Z(\z,(\eps d_\eps(x))^{1/\sig}))\simeq \nu(B_\eps(\z,\eps d_\eps(x)))$.

\begin{proof}[Proof of Corollary~\ref{cor-mube-all-balls}]
The first case follows directly from \eqref{eq-muh-be-est-on-bdy}
together with the inclusions
\[
  B_\eps(\z,r) \subset B_\eps(x,2r) \subset B_\eps(\z,3r)
\]
and the doubling property of $\mube$ and $\nu$.

In the second case, Lemma~\ref{lem-comp-balls-X-Xeps} implies that
\[
  B_\eps(x,\tfrac12r) \subset
  B_X\biggl(x,\frac{C_2r}{2\eps d_\eps(x)}\biggr)
   \subset B_X\biggl(x,\frac{C_2}{2\eps}\biggr)
\]
Recall from Proposition~\ref{prop:bdd-degree} that
the graph $X$ has uniformly bounded degree.
Therefore we have by \eqref{eq-smear-out-muh-be} and Lemma~\ref{lem:nbrs-hv-sim-mass} that
\begin{align*}
  \mube(B_\eps(x,\tfrac12r)) 
    & \simeq \frac{r}{\eps d_\eps(x)} \biggl( \rho_\be(v)\muh(\{v\})
  + \sum_{w\in V \cap B_\eps(x,r/2)} \rho_\be(w) \muh(\{w\}) \biggr) \\
   & \simeq \frac{r}{\eps d_\eps(x)} \rho_\be(v)\muh(\{v\}).
\end{align*}
Since $\rho_\be(v)\simeq (\eps d_\eps(x))^{\be/\eps}$,
the doubling property of $\mube$ concludes the proof.
\end{proof}

\begin{lem}    \label{lem-dimension-mu-be}
Assume that the measure $\nu$ on $Z$ satisfies for all $\zeta\in Z$ and
$0<r'\le r\le 2 \diam Z$,
\begin{equation}   \label{eq-def-s-nu}
   \frac{\nu(B_Z(\zeta,r'))}{\nu(B_Z(\zeta,r))} \simge \biggl( \frac{r'}{r} \biggr)^\snu.
\end{equation}
Let $\be>0$ and $\eps=\log\al$.
Then the measure $\mube$, defined by \eqref{eq-def-muh-be} and \eqref{eq-smear-out-muh-be},
satisfies for all $x\in X_\eps$ and $0<r'\le r\le 2\diam_\eps X_\eps$,
\begin{equation}   \label{eq-s-mube}
  \frac{\mube(B_\eps(x,r'))}{\mube(B_\eps(x,r))} \simge \biggl( \frac{r'}{r} \biggr)^\sbe,
  \quad \text{where } \sbe= \max\{1,\be/\eps+\snu\}.
\end{equation}
\end{lem}

It is well known that every doubling measure $\nu$ satisfies~\eqref{eq-def-s-nu} for some 
$\snu>0$, see for example~\cite[(4.16)]{Heinonen}.

\begin{proof}
Note that $\sig=1$. We shall distinguish three cases:

1.\ If $r\le d_\eps(x)$ then the second case in
Corollary~\ref{cor-mube-all-balls} applies both to $r$ and $r'$ and hence 
\[
  \frac{\mube(B_\eps(x,r'))}{\mube(B_\eps(x,r))}  \simeq \frac{r'}{r}.
\]
2.\ If $r'\ge d_\eps(x)$ then the first case in
Corollary~\ref{cor-mube-all-balls} applies both to $r$ and $r'$ and hence 
\[
  \frac{\mube(B_\eps(x,r'))}{\mube(B_\eps(x,r))} 
  \simeq \frac{(\eps r')^{\be/\eps}  \nu(B_Z(\zeta,\eps r'))}
                  {(\eps r)^{\be/\eps}\nu(B_Z(\zeta,\eps r))}
   \simge \biggl( \frac{r'}{r}  \biggr)^{\be/\eps + \snu}.
\]
3.\ If $r'\le d_\eps(x)\le r$ then by the already proved cases~1 and~2,
\[ 
  \frac{\mube(B_\eps(x,r'))}{\mube(B_\eps(x,r))} 
   = \frac{\mube(B_\eps(x,r'))}{\mube(B_\eps(x,d_\eps(x)))} 
           \frac{\mube(B_\eps(x,d_\eps(x)))}{\mube(B_\eps(x,r))} 
  \simge \frac{r'}{d_\eps(x)} \biggl( \frac{d_\eps(x)}{r}\biggr)^{\be/\eps + \snu}.
\] 
If $\be/\eps+\snu\ge1$ then 
\[
  \frac{r'}{d_\eps(x)} \biggl( \frac{d_\eps(x)}{r}\biggr)^{\be/\eps + \snu}
  \ge \biggl(\frac{r'}{d_\eps(x)}\biggr)^{\be/\eps + \snu} \biggl( \frac{d_\eps(x)}{r}\biggr)^{\be/\eps + \snu}
  = \biggl(\frac{r'}{r}\biggr)^{\be/\eps + \snu},
\]
and if $\be/\eps+\snu \le 1$, then
\[
  \frac{r'}{d_\eps(x)} \biggl( \frac{d_\eps(x)}{r}\biggr)^{\be/\eps + \snu}
  \ge \frac{r'}{d_\eps(x)} \frac{d_\eps(x)}{r} = \frac{r'}{r}.
\]
From the above three cases we conclude that \eqref{eq-s-mube} holds.
\end{proof}

\section{Traces to \texorpdfstring{$Z$}{Z} from the hyperbolic filling
  \texorpdfstring{$X$}{X}} 
\label{sect-traces}

\emph{Recall the standing assumptions from Section~\ref{sect-lift-up}.
Here and in the rest of the paper, we also let $1 \le p <\infty$ and  consider
the uniformized space $\clXeps$ equipped with the measure $\mube$
where $\eps=\log\al$ and $\be >0$.}

\medskip

Theorem~\ref{thm:alph-hyp-fill-eps-uniformize-new} shows that $X_\eps$
is a uniform space. From Proposition~\ref{prop-Z-biLip-Xeps} with 
$\eps=\log\al$ we know that  $\partial_\eps X$ is biLipschitz equivalent to $Z$. 
Hence we can replace $\partial_\eps X$ by
$Z$ as well, since the Besov spaces are biLipschitz invariant. Of course, the 
measure on $Z$ is pushed forward to $\partial_\eps X$ via the biLipschitz identification
$\Psi:Z\to\bdy_\eps X$. In the following, we shall therefore not distinguish between 
$(\bdy_\eps X,d_\eps)$  and $(Z,d_Z)$.

We equip the uniformized space $X_\eps$ with the doubling measure $\mube$,
obtained in \eqref{eq-smear-out-muh-be}.
Equivalently, the uniformized measure $\muh_\be$ from \eqref{eq-def-unif-meas},
based on the smeared out measure $\mu$ from \eqref{eq-smear-out-mu}, can be used.

For the vertices in $X$, consider the projections $\pi_1((z,n))=z$ and $\pi_2((z,n))=n$.
Whenever a nonvertex $x\in X$ belongs to the edge $[v,w]\subset X$, let
\begin{equation}    \label{eq-def-j(x)}
  \pi_2(x):=\min\{\pi_2(v),\pi_2(w)\}. 
\end{equation}

\begin{thm} \label{thm-trace-fund}
Let $u\in\tNp(X_\eps,\mube)$ and $0<\theta\le 1-\be/\eps p$.
Then $u$ has a trace $\ut\in \Bppal(Z)$ given by
\eqref{eq-def-un-An} and~\eqref{eq-def-trace} below, with the 
{\rm(}semi\/{\rm)}norm estimates
\begin{equation} \label{eq-norm-est-Bp-1}
   \|\ut\|_{\theta, p} \simle
  \|g_u\|_{L^p(X_\eps,\mube)}
\end{equation}
and 
\begin{equation} \label{eq-norm-est-Lp-1}
  \|\ut\|_{L^p(Z)} \simle  |u(v_0)| + \|g_u\|_{L^p(X_\eps,\mube)}
  \simle \|u\|_{\Np(X_\eps,\mube)}.
\end{equation}
If $u \in \Lip(X_\eps)$ then $\ut=\uhat|_{Z}$, where $\uhat$ is the unique 
Lipschitz extension of $u$  to $\clXeps$.
\end{thm}

Note that the equivalence classes in $\Np(X_\eps,\mu_\be)$
consist of one function each, see Remark~\ref{rmk-Np-on-graph}.

\begin{proof} 
Let $u\in\tNp(X_\eps,\mube)$ with an upper gradient $g\in L^p(X_\eps,\mube)$. 
For $\z\in Z$ and $n=0,1,\ldots$\,, let $A_n(\z)=A_n\cap B_Z(\z,\al^{-n})$.
We define
\begin{equation}   \label{eq-def-un-An}
  u_n(\z) = \frac{1}{\# A_n(\z)} \sum_{z\in A_n(\z)} u((z,n)),
\end{equation}
where $\# A_n(\z)$ is the cardinality of $A_n(\z)$.
Note that the construction of $A_n$, together with the doubling property of $Z$,
shows that $1\le \# A_n(\z) \le K$ for some $K$ independent of $n$ and $\z$.

For each fixed  $z\in A_n$, the function $\chi_{B_Z(z,\al^{-n})}$ is
lower semicontinuous and thus $\nu$-measurable.
Hence also the linear combinations
\[
  \sum_{z\in A_n(\z)} u((z,n)) = \sum_{z\in A_n} u((z,n)) \chi_{B_Z(z,\al^{-n})}(\z)
  \quad \text{and} \quad \# A_n(\z) = \sum_{z\in A_n} \chi_{B_Z(z,\al^{-n})}(\z)
\]
are $\nu$-measurable, and hence so is $u_n$.
We shall show that for $\nu$-a.e.\ $\z\in Z$, the limit
\begin{equation}  \label{eq-def-trace}
   \ut(\z)= \lim_{n\to\infty} u_n(\z)
\end{equation}
exists and defines the trace $\ut: Z\to\R$. 
To this end, note that $(z,j)\sim(y,j+1)$ whenever $z\in A_j(\z)$ and $y\in A_{j+1}(\z)$,
$j=0,1,\ldots$\,, since $\z \in B(z,\alp^{-j}) \cap B(y,\alp^{-j-1})$. Also,
\[
  \frac{1}{\# A_j(\z)\# A_{j+1}(\z)} \le  1.
\]
We then have for each $j$,
\begin{align}    \label{eq-uj-uj+1}
  |u_j(\z)-u_{j+1}(\z)| 
  &\le  \sum_{z\in A_j(\z)} \sum_{y\in A_{j+1}(\z)} |u((z,j))-u((y,j+1))| \nonumber\\
  &\le \sum_{z\in A_j(\z)} \sum_{y\in A_{j+1}(\z)} \int_{[(z,j),(y,j+1)]} g \,ds_\eps.
\end{align}
On the edge $E=[(z,j),(y,j+1)]$, we have by \eqref{eq-smear-out-muh-be} that
\begin{equation}   \label{eq-ds-dmube}
  ds_\eps \simeq e^{-\eps j}\,d\LL =\frac{\al^{-j}\,d\mube}{ \mube(E)}.
\end{equation}

If $p>1$, then \eqref{eq-ds-dmube} and H\"older's inequality applied to \eqref{eq-uj-uj+1} give
\begin{align}   \label{eq-first-est}
  |u_j(\z)-u_{j+1}(\z)| 
   &\le \al^{-j} \sum_{z\in A_j(\z)} \sum_{y\in A_{j+1}(\z)} 
             \vint_{[(z,j),(y,j+1)]} g \,d\mube \nonumber \\
  &\le \al^{-j} \sum_{z\in A_j(\z)} \sum_{E\in\EE(z,j)} 
             \biggl( \vint_{E} g^p \,d\mube \biggr)^{1/p},
\end{align}
where $\EE(z,j)$ consists of all downward-directed edges emanating from 
the vertex $(z,j)$. Choose $0<\ka<\theta p$ and insert 
$\alpha^{-j\ka/p} \alpha^{j\ka/p}$ into \eqref{eq-first-est}.
Summing over $j$, together with another use of H\"older's inequality, 
this time on the sum,  shows that for all $m>n\ge0$,
\begin{align}   \label{eq-first-sum}
  |u_n(\z)-u_{m}(\z)| 
  &\le \sum_{j=n}^{m-1} |u_j(\z)-u_{j+1}(\z)| \\
   &\le \sum_{j=n}^\infty \al^{-j\ka/p} \al^{-j(1-\ka/p)} \sum_{z\in A_j(\z)} \sum_{E\in\EE(z,j)} 
     \biggl( \vint_{E} g^p \,d\mube \biggr)^{1/p}  \nonumber\\
  &\simle \alpha^{-n\ka/p} \biggl( \sum_{j=n}^\infty \al^{-j(p-\ka)} 
  \sum_{z\in A_j(\z)} \sum_{E\in\EE(z,j)} \vint_{E} g^p \,d\mube \biggr)^{1/p}, \nonumber
\end{align}
where we have also used the fact that 
\[
  \biggl( \sum_{j=n}^\infty \alpha^{-j\ka/(p-1)} \biggr)^{1-1/p} \simeq \alpha^{-n\ka/p},
\]
together with $\# A_j(\z)\le K$ and $\#\EE(z,j)\le K$.
For $p=1$ the estimate is simpler and H\"older's inequality is not needed, 
and the above estimate holds as well. We shall now see that  
\eqref{eq-first-sum} tends to zero as $m>n\to\infty$ for $\nu$-a.e.\ $\z\in Z$. 
Thus, the sequence $\{u_n(\z)\}_{n=0}^\infty$ is a Cauchy sequence, and has
a limit as $n\to\infty$, for $\nu$-a.e.\ $\z\in Z$.

To this end, note that for $E\in\EE(z,j)$,
\begin{equation}         \label{eq-int-E-chi} 
  \vint_{E} g^p \,d\mube 
  \simeq \frac{\al^{j\be/\eps}} {\nu(B_Z(z,\al^{-j}))} \int_{X_V} g(x)^p \chi_E(x) \,d\mube(x),
\end{equation}
where $X_V$ denotes the union of all vertical  edges in $X$.
Also, $z\in A_j(\z)$ if and only if $z\in A_j$ and $\z\in B_Z(z,\al^{-j})$.
Integrating over all $\z\in Z$ we then obtain 
from~\eqref{eq-first-sum} by means of Tonelli's theorem that
\begin{align}  \label{eq:exstimates-cascade}
  &\int_{Z} |u_m(\z)-u_n(\z)|^p \,d\nu(\z)  \nonumber\\
  &\quad \quad \quad \simle
  \al^{-n\ka} \int_{Z} \sum_{j=n}^\infty 
  \sum_{z\in A_j} \frac{\al^{-j(p-\be/\eps-\ka)}}{\nu(B_Z(z,\al^{-j}))}\chi_{B_Z(z,\al^{-j})}(\z)
  \nonumber  \\
  &\quad  \quad \quad  \quad  \times
     \sum_{E\in\EE(z,j)} \int_{X_V} g(x)^p \chi_E(x) \,d\mube(x) \,d\nu(\z)\nonumber\\
  &\quad \quad \quad = \al^{-n\ka} \int_{X_V} g(x)^p  \sum_{j=n}^\infty \al^{-j(p-\be/\eps-\ka)}\nonumber \\
  &\quad  \quad \quad  \quad  \times  
          \sum_{z\in A_j}  \sum_{E\in\EE(z,j)} 
\int_{Z} \frac{\chi_{B_Z(z,\al^{-j})}(\z)}{\nu(B_Z(z,\al^{-j}))} \,d\nu(\z) \chi_E(x) \,d\mube(x).
\end{align}
The integral over $Z$ is clearly equal to $1$. Moreover, for a.e.\ $x \in X$,
\begin{equation*}  
  \chi_E(x)\ne0 \quad \text{only if} \quad x\in E\in\EE(z,j)  \text{ with } j=\pi_2(x),
\end{equation*}
and so for a.e.\ $x\in X$ we have
\begin{align} \label{eq-j=pi-x}
  \sum_{z\in A_j}  \sum_{E\in\EE(z,j)} 
  \int_{Z} \frac{\chi_{B_Z(z,\al^{-j})}(\z)}{\nu(B_Z(z,\al^{-j}))} \,d\nu(\z)\chi_E(x)
  &=\sum_{z\in A_j}  \sum_{E\in\EE(z,j)}\chi_E(x) \nonumber \\  
  &=\chi_{\{y\in X_V:\pi_2(y)=j\}}(x).
\end{align}
We therefore conclude that
\[  
  \int_{Z} |u_m(\z)-u_n(\z)|^p \,d\nu(\z)
  \simle \al^{-n\ka} \int_{\{y\in X_V: \pi_2(y)\ge n\}} g(x)^p \al^{-\pi_2(x)(p-\be/\eps-\ka)}\,d\mube(x).
\] 
Since $p-\be/\eps-\ka \ge \theta p-\ka >0$, we obtain that
\[ 
  \int_{Z} |u_m(\z)-u_n(\z)|^p \,d\nu(\z) 
  \simle \alpha^{-n(p-\be/\eps)}  \int_{X_V} g^p \,d\mube \to 0,  \quad \text{as } m>n\to\infty.
\] 
Hence, the sequence $\{u_n\}_{n=0}^\infty$ is a Cauchy sequence
both in $L^p(Z)$ and $\nu$-a.e.\ in $Z$ (recall that we have $m> n$ in the above computations). 
The limit \eqref{eq-def-trace} therefore exists for $\nu$-a.e.\ $\z\in Z$, and $\ut\in L^p(Z)$.

This also shows, by letting $n=0$ and $u_m\to\ut$, that 
\[   
  \biggl(\int_{Z} |\ut-u(v_0)|^p \,d\nu \biggr)^{1/p}
  \simle \int_{X_V} g^p \,d\mube,
\] 
where $v_0=(z_0,0)$. Thus the first inequality
\[
  \|\ut\|_{L^p(Z)} \simle |u(v_0)| + \|g\|_{L^p(X_\eps,\mube)}
\]
holds in \eqref{eq-norm-est-Lp-1}. For the second inequality, recall the notion of capacity 
from Definition~\ref{deff:capacity}. Since 
$|u(v_0)|^p  \OCpXeps(\{v_0\}) \le \|u\|_{\Np(X_\eps,\mube)}^{p}$
by the definition of $\OCpXeps(\{v_0\})$,
and $\OCpXeps(\{v_0\}) >0$, by Remark~\ref{rmk-Np-on-graph},
we conclude that the second inequality in \eqref{eq-norm-est-Lp-1} holds as well.

To estimate $\|\ut\|_{\theta,p}$, we let
$m\to\infty$ in~\eqref{eq-first-sum} to obtain for $\nu$-a.e.\ $\z\in Z$ and any $n\ge0$,
\begin{equation}     \label{eq-est-utilde-un}
  |\ut(\z)-u_n(\z)|^p 
  \simle \al^{-n\ka} \sum_{j=n}^\infty \al^{-j(p-\ka)} 
     \sum_{z\in A_j(\z)} \sum_{E\in\EE(z,j)} \vint_{E} g^p \,d\mube. 
\end{equation}
A similar estimate holds for $\nu$-a.e.\ $\xi\in Z$.
As in Lemma~\ref{lem-length-z-to-y}, we let
$l\ge 0$ be the smallest integer such that $\al^{-l}\le \tau-1$. Also let 
\[ 
  Z_n(\z)=\{\xi\in Z: \al^{-n-l-1}<d_Z(\xi,\z)\le \al^{-n-l}\}, \quad n=1,2,\ldots, 
\] 
and $Z_0(\z)=Z\setm \itoverline{B_Z(\z,\al^{-l-1})}$. 
Note that $\xi\in Z_n(\z)$ if and only if $\z\in Z_n(\xi)$, in
which case also $\z\in \tau B_Z(z,\al^{-n})\cap \tau B_Z(y,\al^{-n})$ and thus
$(z,n)\sim (y,n)$  for all $z\in A_n(\z)$ and $y\in A_n(\xi)$ with $y \ne z$.
Hence for all $\xi\in Z_n(\z)$, $n=1,\ldots$\,,
\begin{equation}  \label{eq-est-un-un}  
  |u_n(\z)-u_n(\xi)|^p 
  \simle \sum_{z\in A_n(\z)} \sum_{\substack{y\in A_n(\xi)\\ y \ne z}} |u((z,n))-u((y,n))|^p,
\end{equation}
while $u_0(\z)=u(v_0)=u_0(\xi)$ for all $\z,\xi\in Z$.
H\"older's inequality and \eqref{eq-ds-dmube} with $E=[(z,n),(y,n)]$  give
\begin{equation}  \label{eq-est-un-un-2}
  |u((z,n))-u((y,n))|^p \le \biggl( \int_{E} g\,ds_\eps \biggr)^p \simle \al^{-np} \vint_E g^p \,d\mube. 
\end{equation}
Next, note that
\begin{equation}\label{eq:three-terms}
  |\ut(\z)-\ut(\xi)|^p\simle |\ut(\z)-u_n(\z)|^p+|u_n(\z)-u_n(\xi)|^p+|u_n(\xi)-\ut(\xi)|^p,
\end{equation}
and that each of the three terms can be estimated with the aid of
\eqref{eq-est-utilde-un}--\eqref{eq-est-un-un-2}.
We shall insert \eqref{eq:three-terms} into the Besov norm
\[ 
  \|\ut\|^p_{\theta,p} 
  = \int_Z \int_{Z \setm \{\z\}} \frac{|\ut(\xi)-\ut(\z)|^p} {d_Z(\xi,\z)^{\theta p}} 
   \frac{d\nu(\xi)\,d\nu(\z)}{\nu(B_Z(\z,d_Z(\xi,\z)))},
\] 
and obtain three terms corresponding to the three terms on the right-hand side of~\eqref{eq:three-terms}.
We next use the comparisons $d_Z(\xi,\z)^{\theta p} \simeq \al^{-n\theta p}$ and 
\[
  \nu(B_Z(\z,d_Z(\xi,\z))) \simeq \nu(B_Z(\xi,d_Z(\xi,\z))) \simeq \nu(B_Z(\z,\al^{-n}))
  \simeq \nu(B_Z(\xi,\al^{-n}))
\]
whenever $\xi\in Z_n(\z)$ (or equivalently, $\z\in Z_n(\xi)$). We then get 
$\|\ut\|^p_{\theta,p} \simle I_0 + II_0 + III_0$, where 
\begin{align*}
I_0 &:=\int_Z\sum_{n=0}^\infty\int_{Z_n(\z)}\frac{|\ut(\z)-u_n(\z)|^p}{\al^{-n\theta p}} \frac{d\nu(\xi) \,d\nu(\z)}{\nu(B_Z(\z,\al^{-n}))},\\
II_0 &:=\int_Z\sum_{n=0}^\infty\int_{Z_n(\z)}\frac{|u_n(\z)-u_n(\xi)|^p}{\al^{-n\theta p}} \frac{d\nu(\xi) \,d\nu(\z)}{\nu(B_Z(\z,\al^{-n}))},\\
III_0 &:=\int_Z\sum_{n=0}^\infty\int_{Z_n(\xi)}\frac{|\ut(\xi)-u_n(\xi)|^p}{\al^{-n\theta p}} \frac{d\nu(\z) \,d\nu(\xi)}{\nu(B_Z(\xi,\al^{-n}))}.
\end{align*}
Observe that $III_0$ is the same as $I_0$ once the roles
of $\z$ and $\xi$ are switched, and so it suffices to find estimates for $I_0$
and $II_0$. Using~\eqref{eq-est-utilde-un}--\eqref{eq-est-un-un-2}, we find that
\begin{align*}
 I_0  \simle \int_{Z} \sum_{n=0}^\infty& \frac{\al^{-n(\ka-\theta p)}}{\nu(B_Z(\z,\al^{-n}))} 
   \int_{Z_n(\z)} \sum_{j=n}^\infty \al^{-j(p-\ka)} \\
 &   \times \sum_{z\in A_j(\z)} 
            \sum_{E\in\EE(z,j)} \vint_{E} g(x)^p \,d\mube(x)\,d\nu(\xi)\,d\nu(\z) =:I, \\
 II_0 \simle \int_{Z} \sum_{n=1}^\infty &  \frac{\al^{-n(p-\theta p)}}{\nu(B_Z(\z,\al^{-n}))} \\
         &  \times \int_{Z_n(\z)}  \sum_{z\in A_n(\z)}
  \sum_{\substack{y\in A_n(\xi) \\ y \ne z}} 
                        \vint_{[(z,n),(y,n)]} g(x)^p \,d\mube(x)\,d\nu(\xi)\,d\nu(\z) =:II.
\end{align*}

To estimate $I$,  we use \eqref{eq-int-E-chi} and that
$z\in A_j(\z)$ if and only if $z\in A_j$ and $\z\in B_Z(z,\al^{-j})$.
Now an argument using Tonelli's theorem as in the verification
of~\eqref{eq:exstimates-cascade}  yields that
\begin{align*}  
  I &\simeq \sum_{n=0}^\infty \al^{-n(\ka-\theta p)} \sum_{j=n}^\infty \al^{-j(p-\be/\eps-\ka)} \sum_{z\in A_j}
      \sum_{E\in\EE(z,j)}  \\
  &\quad   \times \int_{Z}  \frac{\chi_{B_Z(z,\al^{-j})}(\z)}{\nu(B_Z(z,\al^{-j}))} 
      \int_{Z_n(\z)} \frac{d\nu(\xi) }{\nu(B_Z(\z,\al^{-n}))} \,d\nu(\z)
      \int_{X_V} g(x)^p \chi_E(x) \,d\mube(x).
\end{align*}
Since $Z_n(\z)  \subset B_Z(\z,\al^{-n})$, the integral over $Z_n(\z)$ followed by the integral over $Z$ is
clearly at most $1$. Another use of Tonelli's theorem therefore implies that 
\[  
  I \simle \int_{X_V} g(x)^p \sum_{n=0}^\infty \al^{-n(\ka-\theta p)} 
  \sum_{j=n}^\infty  \al^{-j(p-\be/\eps-\ka)} \sum_{z\in A_j}  \sum_{E\in\EE(z,j)} \chi_E(x) \,d\mube(x).
\] 
The last three sums are simplified using the last identity in~\eqref{eq-j=pi-x} and we obtain
\begin{align*}  
  I &\simle \int_{X_V} g(x)^p  \al^{-\pi_2(x)(p-\be/\eps-\ka)} 
    \sum_{n=0}^{\pi_2(x)} 
    \al^{-n(\ka-\theta p)} \,d\mube(x) \\
&\simeq \int_{X_V} g(x)^p  \al^{-\pi_2(x)(p-\be/\eps-\theta p)} \,d\mube(x),
\end{align*}
because of the choices $\ka<\theta p$ and $\al>1$.
Since $p-\be/\eps-\theta p \ge0$, this yields
\[ 
  I \simle \int_{X_V} g^p \,d\mube. 
\] 

To estimate $II$, we proceed similarly. As in \eqref{eq-int-E-chi}, we have that when $(z,n)\sim(y,n)$,
\begin{equation*}        
\vint_{[(z,n),(y,n)]} g^p \,d\mube
  \simeq \frac{\al^{n\be/\eps}} {\nu(B_Z(z,\al^{-n}))} \int_{X_H} g(x)^p 
     \chi_{[(z,n),(y,n)]}(x) \,d\mube(x),
\end{equation*}
where $X_H$ denotes the union of all horizontal edges in $X$.

Moreover, $z\in A_n(\z)$  and $y\in A_n(\xi)$ if and only if $z,y\in A_n$, $\z\in B_Z(z,\al^{-n})$
and $\xi\in B_Z(y,\al^{-n})$. Tonelli's theorem then yields that
\begin{align*}  
  II &\simeq \sum_{n=1}^\infty \al^{-n(p-\be/\eps-\theta p)}
  \sum_{\substack{z,y\in A_n \\(z,n)\sim (y,n)}} 
        \int_{Z}  \frac{\chi_{B_Z(z,\al^{-n})}(\z)}{\nu(B_Z(z,\al^{-n}))} 
        \int_{Z_n(\z)} \frac{\chi_{B_Z(y,\al^{-n})}(\xi)}{\nu(B_Z(\z,\al^{-n}))} \,d\nu(\xi) \,d\nu(\z)\\
  &\quad   \times  \int_{X_H} g(x)^p \chi_{[(z,n),(y,n)]}(x) \,d\mube(x).
\end{align*}
Since $Z_n(\z)\subset B_Z(\z,\al^{-n})$, the integral over $Z_n(\z)$ followed by the integral
over  $Z$ is at most $1$, and another use of Tonelli's theorem shows that
\[  
  II \simle \int_{X_H} g(x)^p  \sum_{n=1}^\infty \al^{-n(p-\be/\eps-\theta p)}
  \sum_{\substack{z,y\in A_n \\ (z,n)\sim (y,n)}}
          \chi_{[(z,n),(y,n)]}(x) \,d\mube(x).
\] 
Moreover, $\chi_{[(z,n),(y,n)]}(x)\ne0$ if only if $x\in [(z,n),(y,n)]$, in which case also
$n=\pi_2(x)$. We  therefore conclude that
\[ 
  II \simle \int_{X_{H}} g(x)^p \al^{-\pi_2(x)(p-\be/\eps-\theta p)} \,d\mube(x) 
     \le \int_{X_{H}} g^p \,d\mube, 
\] 
because $p-\be/\eps-\theta p \ge0$. Combining the estimates for
$I$ and $II$  gives the desired bound~\eqref{eq-norm-est-Bp-1}.

The fact that $\ut=\hat{u}\vert_Z$ when $u$ is Lipschitz continuous on
$X_\eps$ follows from the definition of $\ut$ and the fact that $u$ has
a unique Lipschitz extension to $\clX_\eps$.
\end{proof}

Recall the notion of capacity from Definition~\ref{deff:capacity}.
The following proposition shows that the boundary measure $\nu$ on $Z=\partial_\eps X$
is absolutely continuous with respect to the $\CpXeps$-capacity.
Note that points in $X$ have positive $\CpXeps$-capacity, but that it is possible to
have nonempty subsets of $\partial_\eps X=Z$ with zero capacity.

\begin{prop} \label{prop-cp-ae}
Let $E \subset \bdy_\eps X$. If $p>\be/\eps$ and $\CpXeps(E)=0$, then $\nu(E)=0$.
\end{prop}

\begin{proof}
Since $\CpXeps$ is an outer capacity by Theorem~\ref{thm-quasicont},
there are open sets $G_j \supset E$ such that $\CpXeps(G_j) < 1/j$.  Then 
\[
  E':=\bigcap_{j=1}^\infty G_j \supset E
\] 
is a Borel set with zero capacity. Let $K \subset E'$ be compact. Because $\mu_\beta$
is doubling and supports a \p-Poincar\'e inequality on $\clXeps$,
it follows from Kallunki--Shan\-mu\-ga\-lin\-gam~\cite[Theorem~1.1]{KalShan} 
(or \cite[Theorem~6.7\,(xi)]{BBbook}) that there are 
$u_k \in \Lip(\clXeps)$ such that $u_k=1$ on $K$ and
$\|u_k\|_{\Np(\clXeps,\mube)} <1/k$, $k=1,2,\ldots$\,.
By the last part of Theorem~\ref{thm-trace-fund} with  $\theta=1-\be/\eps p>0$,
\[
   \nu(K)^{1/p} 
   \le \lim_{k \to \infty}\|u_k\|_{L^p(Z)} \simle
  \lim_{k \to \infty} \|u_k\|_{\Np(X_\eps,\mu_\be)}  =0,
\]
i.e., $\nu(K)=0$. Since $E'$ is a Borel set and $\nu$ is a Borel regular measure, we conclude that 
\[
  \nu(E) \le \nu(E')= \sup_{K\subset E' \text{ compact}}\nu(K)=0.\qedhere
\]
\end{proof}

The following result is a refinement of Theorem~\ref{thm-trace-fund}.
In the case of regular trees, it provides a more precise trace result than 
Proposition~6.1 in Bj\"orn--Bj\"orn--Gill--Shanmugalingam~\cite{BBGS}.
Recall that by Theorem~\ref{thm-tree-back}, every rooted tree $X$
can be seen as a hyperbolic filling of its uniformized boundary $\bdy_\eps X$.

\begin{thm} \label{thm-trace-fund-refined}
Let $u\in \tNp(X_\eps,\mu_\be)$ and $0<\theta\le 1-\be/\eps p$.
Then  $u$ has an extension 
$\uhat \in \tNp(\clXeps,\mube)$.
Furthermore,
the restriction $\ut:=\uhat|_Z$
agrees with the trace of $u$ defined earlier $\nu$-a.e.~in $Z$, and 
belongs to $\Bppal(Z)$ with the {\rm(}semi\/{\rm)}norm estimates
\[ 
  \|\ut\|_{\theta,p} \simle \|g_u\|_{L^p(X_\eps,\mube)}
\] 
and 
\[ 
   \|\ut\|_{L^p(Z)} \simle  |u(x_0)| + \|g_u\|_{L^p(X_\eps,\mube)}
  \simle \|u\|_{\Np(X_\eps,\mube)}.
\]
Moreover, for $\CpXeps$-q.e.\ {\rm(}and thus $\nu$-a.e.{\rm)} $\z\in Z$ we have that 
\begin{equation}   \label{eq-ext-as-Leb-pt}
  \lim_{r\to0^+} 
  \vint_{X_\eps\cap B_\eps(\z,r)}|u-\ut(\z)|^p\, d\mube=0.
\end{equation}
\end{thm}

Note that the extension $\uhat$ is not unique, but it is unique up
to sets of capacity zero and thus $\nu$-a.e.\
(by Proposition~\ref{prop-cp-ae}), since if 
$\uhat_1$ and $\uhat_2$ are two extensions, then they are 
equal $\mube$-a.e., and thus $\CpXeps$-q.e.
We may therefore take the restriction of any such extension.
The key observation that makes the above statement true is that the representatives
in Newtonian spaces are equal q.e., not just a.e.\ as for standard Sobolev spaces.

The last claim of the above theorem tells us that the trace of a function 
in $\Np(X_\eps,\mube)$, as constructed in Theorem~\ref{thm-trace-fund},
agrees with other notions of traces in the current literature, see e.g.\ Mal\'y~\cite{MalyBesov}.

\begin{proof}
By Theorems~\ref{thm:alph-hyp-fill-eps-uniformize-new}
and~\ref{thm-muh-be-doubl-all-be}, $X_\eps$ is a uniform domain in $\clXeps$ and
$\mube$ is doubling and supports a \p-Poincar\'e inequality on $\clXeps$.
Thus by Proposition~5.9 in Bj\"orn--Shanmugalingam~\cite{BS-JMAA},
$X_\eps$ is an extension domain, and thus $u$ has an extension to $\clXeps$, denoted
$\uhat \in \Np(\clXeps,\mu_\be)$.

By Shanmugalingam~\cite[Theorem~4.1 and Corollary~3.9]{Sh-rev},
there is a sequence $u_j \in \Lip(\clXeps)$  such that $\|u_j -\uhat \|_{\Np(\clXeps,\mube)} \to 0$
and $u_j(x) \to \uhat(x)$ for $\CpXeps$-q.e.\ $x \in \clXeps$, as $j \to \infty$.
Let $\ut:=\uhat|_Z$ and $\ut_j:=u_j|_Z$. 
By Proposition~\ref{prop-cp-ae}, $\ut_j(\z) \to \ut(\z)$ for $\nu$-a.e.\ $\z \in Z$.
Moreover, by Theorem~\ref{thm-trace-fund}, $\{\ut_j\}_{j=1}^\infty$ is a Cauchy sequence in the 
norm $\|\cdot\|_{\Bppal(Z)}$. By Remark~\ref{rem-Btheta-Banach},
$\Bppal(Z)$ is complete with respect to $\|\cdot\|_{\Bppal(Z)}$, 
and hence we see that  $\|\ut_j-\ut\|_{\Bppal(Z)}\to 0$ as $j\to\infty$, 
with the (semi)norm estimates from Theorem~\ref{thm-trace-fund} preserved for $\ut$.

By \cite[Theorem~5.62]{BBbook} or \cite[Theorem~9.2.8]{HKST} (for $p>1$) and
Kinnunen--Korte--Shanmugalingam--Tuominen~\cite[Theorem~4.1 and Remark~4.7]{KKST}
(for $p=1$, see below) we know that $\CpXeps$-q.e.\ point
in $\clXeps$ is an $L^p(\mube)$-Lebesgue point of $\uhat$; hence \eqref{eq-ext-as-Leb-pt} holds 
because $\mube(\bdy_\eps X)=0$.

In \cite[p.~404]{KKST} it is assumed that $\mu(X)=\infty$, which is used in 
their proof of the  boxing inequality. In M\"ak\"al\"ainen~\cite{MakaRev},
the boxing inequality is proved also when $\mu(X)< \infty$,
and thus the Lebesgue point result in \cite{KKST} holds also here
where $\mu_\be(\clXeps)<\infty$.
\end{proof}

\section{Extension from 
\texorpdfstring{$Z$}{Z} to its hyperbolic filling 
\texorpdfstring{$X$}{X}}

\label{sect-extension}

\emph{Recall the standing assumptions from Sections~\ref{sect-lift-up}
  and~\ref{sect-traces}.}

\medskip

Theorem~\ref{thm-trace-fund-refined} related the Newtonian space 
$\Np(X_\eps,\mube)$ to a certain range of Besov spaces of functions on 
$Z\equiv\bdy_\eps X$. The principal goal of this section
is to find a counterpart of this theorem in the opposite direction.
This is the purpose of the theorem below.

\begin{thm}\label{thm:filling-Extension}
For $f\in B^\theta_{p,p}(Z)$, consider the extension 
\[
  Ef((z,n)):=\vint_{B_Z(z,\alpha^{-n})} f\, d\nu,
  \quad \text{if } (z,n)\in V\subset X,
\]
extended piecewise linearly {\rm(}with respect to $d_\eps${\rm)} 
to each edge in $X_\eps$, and then to the boundary $\bdy_\eps X$ by letting
\begin{equation} \label{eq-Ef-bdyeps}
  Ef(\z)=\limsup_{r \to 0^+} \vint_{B_\eps(\z,r)} Ef \, d\mu_\be,
  \quad \z \in \bdy_\eps X,
\end{equation}
so that $Ef:\clXeps\to[-\infty,\infty]$.

If $\theta\ge 1-\beta/p\eps$, then $Ef\in \Np(\clXeps,\mube)$ with 
\begin{equation} \label{eq-Ef-ineq}
  \int_{\clXeps}g_{Ef}^p \,d\mu_\beta \simle \Vert f\Vert_{\theta,p}^{p}   
  \quad \text{and} \quad
  \int_{\clXeps}|Ef|^p \,d\mube \simle \int_Z|f|^p\, d\nu.
\end{equation}
Moreover, if $\z\in Z$ is an $L^q(\nu)$-Lebesgue point of $f$ for some
$q\ge1$ then $\z$ is an $L^q(\mube)$-Lebesgue point of $Ef$, and
\[ 
     Ef(z)=f(z).
\] 
Let $\z$ be an $L^1(\nu)$-Lebesgue point of $f$.  Then for any choice of
$z_n\in A_n$ with $d_Z(z_n,\z)<\al^{-n}$ for $n=1,2,\ldots$\,, we have,
\[ 
  \lim_{n\to\infty} Ef((z_n,n))=f(\z).
\] 
\end{thm}

Note that $E$ is a linear operator.

\begin{proof}
If $v=(z,n)\sim (y,m)=w$, then $|m-n|\le1$ and so, by the choice of $\eps=\log\al$,
\[
  d_\eps(v,w)\approx e^{-\eps n} = \alpha^{-n}.
\]
The function given for $x\in [v,w]$ by
\[
  g_{[v,w]}(x):=\frac{|Ef(v)-Ef(w)|}{d_\eps(v,w)}  
\]
is an upper gradient of $Ef$  on $[v,w]$ with respect to the uniformized metric $d_\eps$.
Note that $g_{[v,w]}$ is a constant function. 
Because of $|m-n|\le1$ and $(z,n)\sim(y,m)$, we have for all $\eta\in B_Z(z,\al^{-n})$ that
\[
  B_Z(y,\al^{-m}) \subset 4 \tau B_Z(\eta,\al^{1-n}).
\]
Hence,
\begin{align*}
  g_{[v,w]}&\approx \alpha^n\biggl\vert \vint_{B_Z(z,\alpha^{-n})}f(\zeta)\, d\nu(\zeta)-
           \vint_{B_Z(y,\alpha^{-m})}f(\eta)\, d\nu(\eta)\biggr\vert\\
  &\simle \alpha^n \vint_{B_Z(z,\alpha^{-n})}\vint_{B_Z(y,\alpha^{-m})}|f(\zeta)-f(\eta)|\, d\nu(\eta)\, d\nu(\zeta)\\
  &\simle \alpha^n \vint_{B_Z(z,4\tau\al^{1-n})}\vint_{B_Z(\eta,4\tau\al^{1-n})}
             |f(\zeta)-f(\eta)|\, d\nu(\eta)\, d\nu(\zeta).
\end{align*}
Now by H\"older's inequality, we see that
\[
  g_{[v,w]}^p \simle  \alpha^{np(1-\theta)} \vint_{B_Z(z,4\tau\al^{1-n})}\vint_{B_Z(\eta,4\tau\al^{1-n})}
    \frac{|f(\zeta)-f(\eta)|^p}{\alpha^{-np\theta}}\, d\nu(\eta)\, d\nu(\zeta).
\]
Therefore, letting $g:X\to\R$ be given by $g=g_{[v,w]}$ on each edge $[v,w]$, 
and noting from \eqref{eq-def-muh-be}--\eqref{eq-smear-out-muh-be} that 
\[
  d\mube \simeq \al^{-\be n/\eps}\nu(B_Z(z,\al^{-n})) \,d\LL \quad \text{on $[v,w]$ with $n=\pi_2(v)$,}
\]
we obtain
\[
  \int_{[v,w]}  g^p\, d\mu_\beta
  \simle \alpha^{n(p(1-\theta)-\beta/\eps)} \int_{B_Z(z,4\tau\al^{1-n})}\vint_{B_Z(\eta,4\tau\al^{1-n})}
    \frac{|f(\zeta)-f(\eta)|^p}{\alpha^{-np\theta}}\, d\nu(\eta)\, d\nu(\zeta),
\]
where $z=\pi_1(v)\in A_n\subset Z$. For each nonnegative integer $n$ set 
\[
  X(n):=\{x\in X : n\le \pi_2(x)< n+1\},
\]
where $\pi_2(x)$ is as in \eqref{eq-def-j(x)}.  
By Proposition~\ref{prop:bdd-degree}, each vertex in $X$ has degree at most
$K$ and thus we get integrating over $X(n)$ that
\begin{align*}
  &\int_{X(n)} g^p\, d\mu_\beta
  \le \sum_{z\in A_n}\sum_{V\ni w\sim (z,n)} \int_{[(z,n),w]}g^p\, d\mu_\be\\
  &\quad \quad \quad
  \simle \al^{n(p(1-\theta)-\be/\eps)} \sum_{z\in A_n} \int_{B_Z(z,4\tau\al^{1-n})}
  \vint_{B_Z(\eta,4\tau\al^{1-n})}\frac{|f(\zeta)-f(\eta)|^p}{\al^{-np\theta}}
   \, d\nu(\eta)\, d\nu(\zeta)\\
   &\quad \quad \quad \simle \alpha^{n(p(1-\theta)-\beta/\eps)}
   \int_Z  \vint_{B_Z(\eta,4\tau\al^{1-n})}
       \frac{|f(\zeta)-f(\eta)|^p}{\alpha^{-np\theta}}\, d\nu(\eta)\, d\nu(\zeta).
\end{align*}
In the last line we used the fact that the balls $B_Z(z,4\tau \al^{1-n})$, 
$z\in A_n$, have a bounded overlap in $Z$  because of the doubling property 
of $Z$. As $X=\bigcup_{n=0}^\infty X(n)$ with 
$X(n)\cap X(m)=\emptyset$ if $m\ne n$, it follows that
\[
 \int_{X_\eps} g^p\, d\mu_\beta
   \simle \sum_{n=0}^\infty \alpha^{n(p(1-\theta)-\be/\eps)} \int_Z  \vint_{B_Z(\eta,4\tau\al^{1-n})}
   \frac{|f(\zeta)-f(\eta)|^p}{\alpha^{-np\theta}}\, d\nu(\eta)\, d\nu(\zeta).
\]
If $\theta\ge 1-\be/p\eps$, it then follows from Lemma~\ref{lem:time-series} that
\begin{equation} \label{eq-g-est}
  \int_{X_\eps}   g^p\, d\mu_\beta \simle \Vert f\Vert_{\theta,p}^p<\infty. 
\end{equation}
As $\mube$ supports a $1$-Poincar\'e inequality
on $X_\eps$, by Theorem~\ref{thm-muh-be-doubl-all-be}, and $X_\eps$ is bounded,
it follows that  $Ef\in \Np(X_\eps,\mu_\be)$. 

As in the proof of Theorem~\ref{thm-trace-fund-refined}, we have
an extension $u \in \Np(\clXeps,\mu_\be)$ of $Ef$, and $\CpXeps$-q.e.\ point
in $\clXeps$ is a Lebesgue point of $u$. As $\mube(\bdy_\eps X)=0$, we see that $u(x)=Ef(x)$ 
for $\CpXeps$-q.e.~$x\in\bdy_\eps X$, where $Ef|_{\bdy_\eps X}$ is given by~\eqref{eq-Ef-bdyeps}.
Hence $Ef$ is also in $\Np(\clXeps,\mu_\be)$. Since $X_\eps$ is open in $\clXeps$, we see that
the minimal \p-weak upper gradients of $Ef$ with respect
to $\clXeps$ and $X_\eps$ coincide almost everywhere in $X_\eps$,
and thus the first inequality in~\eqref{eq-Ef-ineq} follows from~\eqref{eq-g-est}.

To control the $L^p$-norm of $Ef$ as  stated in the theorem, note that for
$v=(z,n)\sim w=(y,m)$, 
\begin{align*}
  \int_{[v,w]}|Ef|^p\, d\mu_\beta &\le 
  \mu_\beta([v,w])[|Ef(v)|^p+|Ef(w)|^p]\\
  &\le\mu_\beta([v,w]) \biggl(\vint_{B_Z(z,\alpha^{-n})}|f|^p\, d\nu
   +\vint_{B_Z(y,\alpha^{-m})}|f|^p\, d\nu\biggr)\\
   &\simle \mu_\beta([v,w]) \vint_{B_Z(z,4\tau\alpha^{1-n})}|f|^p\, d\nu.
\end{align*}
Therefore, for each nonnegative integer $n$, with $X(n)$ as above, we have that
\begin{align*}
  \int_{X(n)}|Ef|^p\, d\mu_\beta 
  &\simle
  \alpha^{-n\beta/\eps} \sum_{z\in A_n} \nu(B_Z(z,\alpha^{-n}))
  \vint_{B_Z(z,4\tau\alpha^{1-n})}|f|^p\, d\nu\\
  &\simle
  \alpha^{-n\beta/\eps}\int_Z|f|^p\, d\nu.
\end{align*}
It follows that
\[
  \int_{X_\eps}|Ef|^p\, d\mu_\beta \simle
  \Vert f\Vert_{L^p(Z)}^p\sum_{n=0}^\infty \alpha^{-n\beta/\eps}
  \simle \Vert f\Vert_{L^p(Z)}^p
\]
as desired.

Assume that $q\ge1$ and that $\z\in Z$ is an $L^q(\nu)$-Lebesgue point of $f$.
Let $N\ge0$ be a fixed but arbitrary integer and consider all $x\in X$
such that $d_\eps(x,\z)<r:=\al^{-N}/\eps$. If $x$ belongs to an edge $[v,w]$, then 
at least one of the vertices also belongs to $B_\eps(\z,r)$, say
$w=(y,m)$, and hence
\[
  \al^{-m} = e^{-\eps m} = \eps d_\eps(w) \le \eps d_\eps(w,\z) < \eps r = \al^{-N}, 
\]
from which it follows that $m\ge N+1$ and thus $n\ge N$, where $v=(z,n)$. In particular, 
$d_\eps(z,v) = d_\eps(v) = e^{-\eps n}/\eps \le r$ and
$d_\eps(v,x) \le e^{-\eps N} = \eps r$. Proposition~\ref{prop-Z-biLip-Xeps} then yields
\[
  d_Z(z,\z) \le 2\tau\al d_\eps(z,\z) 
  \le 2\tau\al (d_\eps(z,v) + d_\eps(v,x) + d_\eps(x,\z) )
   < 2\tau \al (2+\eps)r,
\]
and similarly $d_Z(y,\zeta)<2\tau \al(2+\eps)r$. 
Since $Ef(x)$ is a convex combination of $Ef(v)$ and $Ef(w)$, we have
\[
  \int_{[v,w]} |Ef-f(\z)|^q\,d\mube \le (|Ef(v)-f(\z)|^q + |Ef(w)-f(\z)|^q) \mube([v,w]),
\]
where by the definition of $Ef$ and H\"older's inequality,
\[
  |Ef(v)-f(\z)|^q = \biggl|  \vint_{B_Z(z,\al^{-n})} f\,d\nu -f(\z) \biggr|^q
  \le \vint_{B_Z(z,\al^{-n})}  |f-f(\z)|^q\,d\nu.
\]
Noting that 
\[
  \mube([v,w]) \simeq \rho_\be(v) \muh(v) = e^{-\be n} \nu(B_Z(z,\al^{-n}))
\] 
and that every vertex belongs to at most a bounded number of edges (by 
Proposition~\ref{prop:bdd-degree}), we conclude that
\[
  \int_{B_\eps(\z,r)} |Ef-f(\z)|^q\,d\mube 
   \simle \sum_{n\ge N} e^{-\be n} \sum_{z\in A_n\cap B_Z(\z,2\tau \al (2+\eps)r)}
          \int_{B_Z(z,\al^{-n})}  |f-f(\z)|^q \,d\nu.
\]
Since for each $n\ge N$ we have that
$\al^{-n} \le \al^{-N} = \eps r$ and the balls $B_Z(z,\al^{-n})$, $z\in A_n$, have bounded overlap in $Z$,
we obtain 
\begin{align*}
  \int_{B_\eps(\z,r)} |Ef-f(\z)|^q\,d\mube 
   &\simle \sum_{n\ge N} e^{-\be n} \int_{B_Z(\z,4\tau \al (1+\eps)r)}  |f-f(\z)|^q \,d\nu\\
   &\simeq e^{-\be N} \int_{B_Z(\z,4\tau \al (1+\eps)r)}  |f-f(\z)|^q \,d\nu,
\end{align*}
where $e^{-\be N} = (\al^{-N})^{\be/\eps} = (\eps r)^{\be/\eps}$.
Dividing by $\mube(B_\eps(\z,r)) \simeq (\eps r)^{\be/\eps}\nu(B_Z(\z,\eps r))$
(because of Corollary~\ref{cor-mube-all-balls}) and letting $N\to\infty$
shows that $\z$ is an $L^q(\mube)$-Lebesgue point of $Ef$.
That $Ef(\z)=f(\z)$ now follows directly from
the definition of $Ef(\z)$ in~\eqref{eq-Ef-bdyeps} and by considering the case $q=p$ in
the above discussion (recall that $f$ is necessarily in $L^p(\nu)$ and so $\nu$-a.e.~point in $Z$
is an $L^p(\nu)$-Lebesgue point of $f$).

Moreover,  with $z_n$ as in the final claim of the theorem for $n=1,2,\ldots$\,,
the doubling property of $\nu$ and the fact that
$B_Z(z_n,\al^{-n})\subset B_Z(\z,2\al^{-n})$ yield
\begin{align*}
    |Ef((z_n,n))-f(\z)|
    &=\biggl\vert\vint_{B_Z(z_n,\al^{-n})} [f(y)-f(\z)] \,d\nu(y)\biggr\vert\\
  &\simle \vint_{B_Z(\z, 2\al^{-n})}|f(y)-f(\z)| \,d\nu(y)\to 0,
  \quad \text{as } n\to\infty.
\qedhere
\end{align*}
\end{proof}

\section{Properties of Besov functions on \texorpdfstring{$Z$}{Z}}
\label{sect-Besov-applications}

\emph{Recall the standing assumptions from Sections~\ref{sect-lift-up}
    and~\ref{sect-traces}. 
Since $\eps=\log\al$, from Proposition~\ref{prop-Z-biLip-Xeps} we know that
$\partial_\eps X$ can be identified with $Z$ in
    a biLipschitz fashion.
In this section we also fix $0<\theta<1$ and let $\be=\eps p(1-\theta)$.}

\medskip

In this section we will only consider the Besov spaces on $Z$ that arise as traces 
of Newtonian functions on $\clX_\eps$ as in Theorem~\ref{thm-main-intro}.
This makes it possible to derive various regularity properties for
$\Bppal(Z)$ from from the theory of Newtonian spaces.

\begin{prop}\label{prop-lip-dens-besov}
If $0<\theta<1$ then Lipschitz functions are dense in $B^\theta_{p,p}(Z)$.
\end{prop}

\begin{proof}
Equip the uniformized hyperbolic filling $X_\eps$ with the measure $\mube$,
where $\be= \eps p(1-\theta)$.
Theorems~\ref{thm-trace-fund-refined} and \ref{thm:filling-Extension}
tell us that $B^\theta_{p,p}(Z)$ is the trace space of $\Np(\clX_\eps,\mube)$,
with comparable norms.

Since $\mube$ is doubling and supports a $1$-Poincar\'e inequality on
$\clX_\eps$, it follows from Shanmugalingam~\cite[Theorem~4.1 and
Corollary~3.9]{Sh-rev} that Lipschitz functions are dense in $\Np(\clX_\eps,\mube)$.
Their restrictions to $Z$ are then dense in $B^\theta_{p,p}(Z)$.
\end{proof}

\begin{prop}\label{prop:Besov-null=Newt-null}
Let $E\subset Z$. Then $\Capp_{B_{p,p}^\theta(Z)}(E) \simeq \CpXeps(E)$.
\end{prop}

\begin{proof}
Let $u\in B_{p,p}^\theta(Z)$ be admissible in the definition of $\Capp_{B_{p,p}^\theta(Z)}(E)$,
i.e.\ $u\ge1$ $\nu$-a.e.\ in an open neighborhood  $G\subset Z$ of $E$.
By truncation and redefinition on a set of $\nu$-measure zero, we may assume that 
$u \equiv 1$ in $G$. Let $Eu\in \Np(\clXeps,\mube)$ be its extension as guaranteed
by Theorem~\ref{thm:filling-Extension}. As all points in $G$ are Lebesgue points for $u$, 
we see that $Eu \equiv u \equiv 1$ in $G$. Hence $Eu$ is admissible in computing 
$\CpXeps(E)$, and so by Theorem~\ref{thm:filling-Extension},
\[ 
  \CpXeps(E) \le \Vert Eu\Vert_{\Np(\clXeps,\mube)}^p \simle \Vert u\Vert_{B^\theta_{p,p}(Z)}^p.
\] 
Taking infimum over all $u$ admissible in the definition of $\Capp_{B_{p,p}^\theta(Z)}(E)$
proves one inequality in the statement of the lemma.

Conversely, since $\CpXeps$ is an outer capacity 
by Theorems~\ref{thm-quasicont} and~\ref{thm-muh-be-doubl-all-be},
for each $\eta>0$ we can find an open set $U\subset \clXeps$ with $E\subset U$ and a function
$u \in \tNp(\clXeps,\mube)$ such that $u\ge 1$ on $U$ and 
\[
  \Vert u\Vert_{\Np(\clXeps,\mube)}  < \CpXeps(E) + \eta.
\]
By Theorem~\ref{thm-trace-fund-refined},
the function $f=u\vert_Z \in B^\theta_{p,p}(Z)$ with
\[
  \Vert f\Vert_{L^p(Z)}^p + \Vert f\Vert_{\theta,p}^p   
  \simle \Vert u\Vert_{\Np(\clXeps,\mube)}^p < \CpXeps(E) + \eta.
\]
As $f\ge 1$ on the relatively open set $G:=U\cap Z$, letting $\eta\to0$ concludes the proof.
\end{proof}

We have now shown that for subsets of $Z$ the two capacities are comparable. Next
we turn our attention to the matter of continuity properties of Besov functions.
The following result shows that functions in $B^\theta_{p,p}(Z)$ have 
representatives that are  \emph{quasicontinuous} with respect to the Besov capacity, i.e.\ such that
for each $\eta>0$ there is an open set $G\subset Z$ with
$\Capp_{B_{p,p}^\theta(Z)}(G)<\eta$ such that $f\vert_{Z\setminus G}$ is continuous.

\begin{prop} \label{prop-Besov-qcont}
Let $f_0\in B^\theta_{p,p}(Z)$.  Then there is a $\Capp_{B_{p,p}^\theta(Z)}$-quasicontinuous
function $f\in B^\theta_{p,p}(Z)$ such that $f=f_0$ $\nu$-a.e.\ in $Z$.
\end{prop}

\begin{proof}
Given such a function $f_0$, let $Ef_0\in \Np(\clX_\eps,\mube)$ be its extension given by 
Theorem~\ref{thm:filling-Extension}. Then $f:=Ef_0|_Z = f_0$ $\nu$-a.e.
By Theorems~\ref{thm-quasicont} and~\ref{thm-muh-be-doubl-all-be},
$Ef_0$ is $\CpXeps$-quasi\-cont\-in\-u\-ous, i.e.\ 
for each $\eta>0$ there is an open set $U\subset \clX_\eps$ with $\CpXeps(U)<\eta$ 
such that $Ef_0\vert_{\clX_\eps\setminus U}$ is continuous. 
By choosing $G=U\cap Z$, we get that $f\vert_{Z\setminus G}$ is continuous. 
Moreover, by Proposition~\ref{prop:Besov-null=Newt-null}, we see that
\[
  \Capp_{B_{p,p}^\theta(Z)}(G)\simeq \CpXeps(G)\le \CpXeps(U)<\eta,
\]
which completes the proof. 
\end{proof}

The following result shows that Besov functions have Lebesgue points
q.e., provided that the measure on $Z$ satisfies a reverse-doubling property.

\begin{prop}  \label{prop-Leb-pt}
Assume that $\nu$ satisfies 
\eqref{eq-def-s-nu} with $\snu>0$ and that there is some $\eta>0$ such that
\begin{equation}   \label{eq-rev-doubl}
\frac{\nu(B_Z(\z,r'))}{\nu(B_Z(\z,r))} \lesssim  \Bigl( \frac{r'}{r} \Bigr)^\eta
\end{equation}
for all $\z\in Z$ and all $0<r'\le r\le 2\diam Z$.
Let $1\le q \le \snu p/(\snu - p\theta)$ and $\ut\in B^\theta_{p,p}(Z)$. 
Then  there is a function $u \in L^q(Z)$ such that $u=\ut$ $\nu$-a.e.\  and  
\[
\lim_{r \to 0^+} \vint_{B_\eps(\z,r)\cap Z} |u-u(\z)|^q\,d\nu =0
\quad \text{for $\Capp_{B_{p,p}^\theta(Z)}$-q.e.\ $\z\in Z$.}
\]
\end{prop}

Since $\nu$ is doubling, condition \eqref{eq-rev-doubl} is equivalent
to $Z$ being uniformly perfect, see Mart\'\i n--Ortiz~\cite[Lemma~7]{MarOrt}.
See~\cite{SSS} for a weaker Lebesgue point result when $Z$ is not necessarily 
uniformly perfect. Embeddings of Besov spaces into $L^q$ spaces were also obtained in
Mal\'y~\cite[Corollary~3.18\,(i)]{MalyBesov} via embeddings into Haj\l asz--Sobolev spaces.
Proposition~\ref{prop-Leb-pt} will follow from our trace and extension results and the
following two-weighted Poincar\'e type inequality, which is a special
case of Bj\"orn--Ka\l amajska~\cite[Theorem~3.1]{BjKa}.

\begin{prop}  \label{prop-PI-mube-nu}
Let $\nu$ and $\mube$ be doubling measures on $Z=\bdy_\eps X$ and $\clX_\eps$,
respectively. Assume moreover that $\mube$ supports a 
\p-Poincar\'e inequality on $\clX_\eps$ with dilation $\la$
and that $\nu$ satisfies the reverse-doubling condition~\eqref{eq-rev-doubl}.
Let $1\le p <q<\infty$ and $u\in \Np(\clX_\eps,\mube)$ be such that $\nu$-a.e.\ $z\in Z$ is
a $\mube$-Lebesgue point of $u$.  Then for all balls $B=B_\eps(\z,r)$ with $\z\in Z$, 
\[
  \biggl( \int_{B\cap Z} |u-u_{B,\mube}|^q\,d\nu \biggr)^{1/q} 
  \lesssim \Theta_q(r) \biggl( \int_{2\la B} g_u^p\,d\mube \biggr)^{1/p},
\]
where
\[
  \Theta_q(r):= \sup_{0<\rho\le r} \sup_{z\in B\cap Z} 
   \frac{\rho \nu(B_\eps(z,\rho))^{1/q}}{\mube(B_\eps(z,\rho))^{1/p}}.
\]
\end{prop}

\begin{proof}[Proof of Proposition~\ref{prop-Leb-pt}.]
By H\"older's inequality, it suffices to consider the case of $q>p$.
Using Theorem~\ref{thm:filling-Extension}, 
we can find a function $u\in\Np(\clX_\eps,\mube)$
such that $u=\ut$ $\nu$-a.e.\ on $Z$.
By~\cite[Lemma~9.2.4]{HKST} and Theorem~\ref{thm-trace-fund-refined},
we know that for $\CpXeps$-q.e.\ $\z\in Z$,
\begin{equation}   \label{eq-choice-zeta}
  \lim_{r\to 0^+} r^p \vint_{B_\eps(\z,r)} g_u^p\,d\mube=0
  \quad \text{and} \quad
  \lim_{r\to 0^+} \vint_{B_\eps(\z,r)} u\,d\mube= u(\z).
\end{equation}
In particular, Proposition~\ref{prop-cp-ae} shows that $\nu$-a.e.\ $z\in Z$ is
a $\mube$-Lebesgue point of~$u$. 

Proposition~\ref{prop:Besov-null=Newt-null} shows that
\eqref{eq-choice-zeta} holds for $\Capp_{B_{p,p}^\theta(Z)}$-q.e.\ $\z\in Z$.
For such $\z$, we have by the Minkowski inequality and
Proposition~\ref{prop-PI-mube-nu} that
\begin{align}  \label{eq-split-Theta}
  \biggl( \vint_{B_\eps(\z,r)\cap Z} |u-u(\z)|^q\,d\nu \biggr)^{1/q} 
   &\lesssim \biggl| \vint_{B_\eps(\z,r)} u\,d\mube - u(\z) \biggr| \\
  &\quad  + \frac{\Theta_q(r) \mube(B_\eps(\z,2\la r))^{1/p}}{\nu(B_\eps(\z,r))^{1/q}}
      \biggl( \vint_{B_\eps(\z,2\la r)} g_u^p\,d\mube \biggr)^{1/p}. \nonumber
\end{align}
In view of \eqref{eq-choice-zeta} and the definition of $\Theta_q(r)$, 
it suffices to show that for all 
$0<\rho\le r\le 2\diam_\eps X_\eps$ and $z\in B_\eps(\z,r))\cap Z$,
\begin{equation}   \label{eq-est-Theta}
  \frac{\rho \nu(B_\eps(z,\rho))^{1/q}}{\mube(B_\eps(z,\rho))^{1/p}}
  \frac{\mube(B_\eps(\z,2\la r))^{1/p}}{\nu(B_\eps(\z,r))^{1/q}} \simle r
\end{equation}
with a comparison constant independent of $z$, $\rho$ and $r$.
Theorem~\ref{thm-muh-be-doubl-all-be} and the doubling property show that
\[
  \frac{\rho \nu(B_\eps(z,\rho))^{1/q}}{\mube(B_\eps(z,\rho))^{1/p}}
  \simeq \rho^{1-\be/\eps p} \nu(B_\eps(z,\rho))^{1/q-1/p}
\]
and
\[
  \frac{\mube(B_\eps(\z,2\la r))^{1/p}}{\nu(B_\eps(\z,r))^{1/q}}
  \simeq r^{\be/\eps p} \nu(B_\eps(\z,r))^{1/p-1/q}
  \simeq r^{\be/\eps p} \nu(B_\eps(z,r))^{1/p-1/q}. 
\]
Since $1/q-1/p<0$ and $\nu$ satisfies \eqref{eq-def-s-nu}, the
required estimate \eqref{eq-est-Theta} holds because
\[
  1-\frac{\be}{\eps p} + \snu \biggl( \frac1q- \frac1p\biggr) \ge 0.
\]
The conclusion $u\in L^q(Z)$ follows by applying \eqref{eq-split-Theta}
to $r=2\diam \clX_\eps$.
\end{proof}

Even though  we have so far only considered compact $Z$,
we can now apply Proposition~\ref{prop-Leb-pt} to obtain the following improvement 
of Netrusov's result~\cite[Proposition~1.4]{Netr} in $\R^n$,
which was obtained for $q<np/(n-p\theta)$.
The lift from the compact to the unbounded case is somewhat subtle 
since the Besov norm is nonlocal.

\begin{prop} \label{prop-Netrusov}
Assume that $n \ge 1$ and that $p \theta <n$.
Let $q=np/(n-p\theta)$, $\ut\in B^\theta_{p,p}(\R^n)$,
$u(\z):=\limsup_{r \to 0^+} \ut_{B(\z,r)}$, and 
\[
   E=\biggl\{ \z : \limsup_{r \to 0^+} \vint |u(x)-u(\z)|^q\, dx >0\biggr\}
\]
be the set of non-$L^q$-Lebesgue points for $u$.
Then $\Capp_{B_{p,p}^\theta(\R^n)}(E)=0$.
\end{prop}

\begin{proof}
For $r>0$, let $Z_r=\itoverline{B(0,r)}$. Observe that for $x\in Z_r$ and 
$0<\rho\le 2r$ we have $m(B(x,\rho)\cap Z_r)\simeq m(B(x,\rho))\simeq \rho^n$.
It follows that $\Bppal(\R^n)\subset\Bppal(Z_r)$ for all $r>0$.
Since $\Capp_{B_{p,p}^\theta(\R^n)}$ is comparable to a countably subadditive 
Besov capacity on $\R^n$, see Adams--Hedberg~\cite[Propositions~2.3.6 and~4.4.3]{AdHed}, 
it suffices to show that $\Capp_{B_{p,p}^\theta(\R^n)}(E\cap B(0,r))=0$ for all $r\ge1$.
Fix $\eps>0$ and let $R\ge 3r$. As $u \in \Bppal(Z_{R})$, it follows from 
Proposition~\ref{prop-Leb-pt} that  $\Capp_{B_{p,p}^\theta(Z_{R})}(E \cap B(0,r)) =0$.
Proposition~\ref{prop:Besov-null=Newt-null}, together with
\cite[Lemma~6.15 and Theorem~4.21]{BBbook}
and Theorem~\ref{thm-trace-fund-refined},
now implies that there is a function $v: \R^n \to [0,1]$ with $\supp v \subset B(0,2r)$
such that $v \ge 1$ in an open neighborhood of $E \cap B(0,r)$ and 
\(
  \|v\|_{\Bppal(Z_R)}^p < \eps.
\)
Now 
\begin{align*}
  \|v\|_{\theta,p}^p + \|v\|_{L^p(\R^n)}^p 
  &\simle \|v\|_{\Bppal(Z_R)}^p + \int_{\R^n \setm Z_{R}} \int_{B(0,2r)}
   \frac{v(\z)}{|\z-\xi|^{p\theta+n}} \,d\z\,d\xi \\
  &\simle  \eps +  \int_{\R^n \setm Z_{R}} 
   \frac{(2r)^n}{(|\xi|-2r)^{p\theta+n}} \,d\xi \to \eps,
\end{align*}
as $R \to \infty$. Hence $\Capp_{B_{p,p}^\theta(\R^n)}(E \cap B(0,r)) \simle \eps$, and 
letting $\eps\to0$ concludes the proof.
\end{proof}

The following result extends Proposition~6.6 in 
Bj\"orn--Bj\"orn--Gill--Shan\-mu\-ga\-lin\-gam~\cite{BBGS} to general compact
doubling metric measure spaces $Z$, and essentially recovers
Corollary~3.18\,(iii) in Mal\'y~\cite{MalyBesov}.

\begin{prop}   \label{prop-Holder}
Assume that the measure $\nu$ on $Z$ satisfies \eqref{eq-def-s-nu} for
all $\zeta\in Z$ and $0<r'\le r\le \diam Z$, with exponent $s_\nu$. Let 
$\sbe=\max\{1,p(1-\theta)+\snu\}$. If $p>\sbe$ then every 
$f\in B^\theta_{p,p}(Z)$ has a $\nu$-a.e.\ representative
which is $(1-\sbe/p)$-H\"older continuous on $Z$.
\end{prop}

\begin{proof} 
By Theorem~\ref{thm:filling-Extension}, there is a function 
$u\in \Np(\clX_\eps,\mube)$ such that $u|_Z=f$ $\nu$-a.e. By 
Lemma~\ref{lem-dimension-mu-be} and by $\beta=\eps p(1-\theta)$, 
$\mube$ satisfies the dimension condition \eqref{eq-s-mube}.

Since $p>\sbe$, functions in $\Np(\clX_\eps,\mube)$ are
$(1-\sbe/p)$-H\"older continuous with respect to $d_\eps$, by 
\cite[Corollary~5.49]{BBbook} or \cite[Theorem~9.2.14]{HKST}.
(There is a missing local compactness assumption in~\cite[Theorem~9.2.14]{HKST}.)
It follows that the trace  $u|_Z$ is $(1-\sbe/p)$-H\"older continuous with respect to $d_Z$.
\end{proof}

Elementary calculations show that Proposition~\ref{prop-Holder} applies in the 
following two cases with $0<\theta<1$ and $p>\snu/\theta$:
\begin{itemize}
\item If $p\ge (1-\snu)/(1-\theta)$, then 
$\sbe=p(1-\theta)+\snu \ge 1$ and $1-\sbe/p=\theta-\snu/p$. Hence every
$f\in B^\theta_{p,p}(Z)$ has a $\nu$-a.e.\ representative which is
$(\theta-\snu/p)$-H\"older continuous.
\item If $1<p<(1-\snu)/(1-\theta)$ (which necessarily implies that
  $\snu<\theta$), then $p(1-\theta)+\snu <1=\sbe$ and
every $f\in B^\theta_{p,p}(Z)$ has a $\nu$-a.e.\ representative which is $(1-1/p)$-H\"older continuous.
\end{itemize}
If $\theta\ge1$ then $B^\theta_{p,p}(Z)\subset B^{\theta'}_{p,p}(Z)$
for every $0<\theta'<1$ and the above two cases with $\theta$ replaced
by $\theta'$ imply (upon letting $\theta'\to1$):
\begin{itemize}
\item If $p>\snu \ge1$, then every
$f\in B^\theta_{p,p}(Z)$ has a $\nu$-a.e.\ representative which is
$\eta$-H\"older continuous for any $0<\eta<1-\snu/p$.
\item If $0<\snu<1<p$, then every $f\in B^\theta_{p,p}(Z)$ has a $
\nu$-a.e.\ representative which is $(1-1/p)$-H\"older continuous.
\end{itemize}

\end{document}